\documentclass[a4paper,11pt]{amsart}
\usepackage{amsmath,inputenc,euscript,amssymb,geometry}
\usepackage{graphicx}
\usepackage{dsfont}
\usepackage{graphics}
\usepackage[normalem]{ ulem }
\usepackage{soul}
\usepackage{enumerate}
\usepackage{enumitem}
\usepackage{amssymb}
\usepackage{latexsym}
\usepackage{amssymb,amsbsy,amsmath,amsfonts,amssymb,amscd}
\usepackage{hyperref}
\usepackage{xcolor}

\newcommand{\N}{\mathbb{N}}
\newcommand{\T}{\mathbb{T}}
\newcommand{\Z}{\mathbb{Z}}
\newcommand{\R}{\mathbb{R}}

\newtheorem{lemma}{Lemma}[section]
\newtheorem{theorem}[lemma]{Theorem}
\newtheorem{prop}[lemma]{Proposition}

\newtheorem{remark}[lemma]{Remark}

\numberwithin{equation}{section}

\begin{document}
\title[On low regularity well-posedness of the binormal flow]{On low regularity well-posedness\\ of the binormal flow}
\author[V. Banica]{Valeria Banica}
    \address[V. Banica]{Sorbonne Universit\'e, CNRS, Universit\'e de Paris, Laboratoire Jacques-Louis Lions (LJLL), F-75005 Paris, France} 
\email{Valeria.Banica@sorbonne-universite.fr}

\author[R. Luc\`a]{Renato Luc\`a}
\address[R. Luc\`a]{Institut Denis Poisson, UMR 7013, CNRS, 
Université d’Orléans, Bâtiment de Mathématiques, rue de
Chartres, F-45100 Orléans}
\email{renato.luca@univ-orleans.fr}

\author[N. Tzvetkov]{Nikolay Tzvetkov}
    \address[N. Tzvetkov]{Ecole Normale Sup\'erieure de Lyon, UMPA, UMR CNRS-ENSL 5669, 46, all\'ee d'Italie, 69364-Lyon Cedex 07, France} 
\email{nikolay.tzvetkov@ens-lyon.fr}

\author[L. Vega]{Luis Vega}
\address[L. Vega]{BCAM - Basque Center for Applied Mathematics, Alameda de Mazarredo 14, E48009 Bilbao, and Dpto. Matem\'aticas UPV/EHU Apto. 644
48080 Bilbao, Basque Country - Spain} 
\email{luis.vega@ehu.es}
\date\today
\maketitle

\begin{abstract}
We focus on a class of solutions of the binormal flow, model of the evolution of vortex filaments, 
that generate several corner singularities in
finite time. This phenomenon has been studied in \cite{BVAnnPDE, BLTV} in the regular case, which in this context is in terms of the summability of the angles of the corners generated. Our goal here is to investigate
the lower regularity case, using further the Hasimoto approach that allows to use the 1D cubic nonlinear Schr\"odinger to study the binormal flow. We first obtain a deterministic result by proving an existence result for general binormal flow solutions at low regularity. Then we obtain improved results on the above class of solutions 
by a suitable randomization of the curvature and torsion of the vortex filament.  To do so, we prove a scattering result 
for a quasi-invariance measure associated with a suitable 1D cubic nonlinear Schr\"odinger equation that we consider of independent interest. An interesting feature of this result is that we are able to identify a limit measure, which is usually not possible when working on quasi-invariant Gaussian measures for Hamiltonian PDE’s on bounded domains.
\end{abstract}

\section{Introduction}
\subsection{The framework}
We are interested in the study of the binormal flow equation 
\begin{equation}\label{BinormalVeryFirst}
\chi_t(t,x)=\chi_x\times\chi_{xx}(t,x).
\end{equation} 
This models the evolution of a vortex filament in an ideal fluid.
The unknown is a time dependent family of curves $x \mapsto \chi(t,x)$
that are parametrized by their arc-length $x \in \R$. It is easy to see that
the arc-length parametrization is preserved under the evolution. Letting $T := \chi_x$, the tangent vector $T$ satisfies the Schr\"odinger map  
$$T_t = T \wedge T_{xx}. $$

Solutions of the Schr\"odinger map, and thus corresponding solutions of the binormal flow equation, are related to the 
focusing cubic NLS equation on the line
\begin{equation}\label{CubicNLS}
iu_t+u_{xx}+ |u|^2u=0 
\end{equation} 
via the Hasimoto transformation (\cite{H}, see \S \ref{sectHas} for details). Indeed, Hasimoto proved the striking fact that 
one can produce solutions of the Schr\"odinger map and thus of the binormal 
flow
 equation, from regular enough solutions $u$ of the NLS equation \eqref{CubicNLS}. This is in a sense similar to the Madelung transform in hydrodynamics. In particular, in the case of non-vanishing solutions, the curvature $c$ and torsion $\tau$ of the curve    
 $\chi$ are given by
 $$
 c(x,t) = |u(x,t)|, \qquad \tau(x,t) = \partial_x   {\text arg \,} (u(x,t)).
 $$
The first results for the Schr\"odinger map (and first results for the binormal flow) are existence of weak solutions, without uniqueness, for the periodic case (closed curves case respectively) and the line case (infinite curves respectively) at the $H^1$ regularity level (with tangent vector $T$ at the $H^1$ regularity level). These results were obtained by compactness techniques (\cite{ZuGu,SuSuBa,NiTa}). Then using conservation laws and latter the parallel frame approach introduced by Koiso in 1996 to avoid the non-vanishing curvature condition in Hasimoto's approach, global well-posedness was obtained in $H^2$ (\cite{DiWa,ChShUh,NSVZ,RoRuSt}). More recently, suitable uniqueness of the weak solutions in the closed case in $H^1$ was obtained (\cite{ChErTz}). For more details on these results see \S 1.2 in \cite{JeSm} and \cite{ChErTz}. We note that the binormal flow solutions for which the tangent vector $T$ is at the $H^1$ level are given by curves $\gamma\in H^2_{loc}\subset \mathcal C^{\frac 32}_{loc}$ with $\gamma_{xx}\in L^2$. The very rough case of closed Lipschitz curves was analyzed with geometric measure theory tools and existence of generalized weak solutions was obtained, together with a weak-strong uniqueness result (\cite{JeSm}). The very rough case of curves presenting one corner was analyzed in a precised way, starting with the detailed evolution of self-similar solutions of the binormal flow described in \cite{GRV}. The study of perturbations of self-similar  solutions of the binormal flow was completed in \cite{BVAENS}. The case of regular polygons closed curves was first analysed numerically in \cite{JeSm2}, then further in \cite{DHV1}. The rigorous study of the case of infinite curves with several corners, was initiated in \cite{BVAnnPDE}. For various turbulent features proved for these solutions see the recent survey \cite{BVJMP}. We shall next describe these solutions. 
\\

In this article we are interested in a specific class of solutions 
of the NLS equation \eqref{CubicNLS} of the form  
\begin{equation}\label{SumOfDeltaSolution}
u(t,x)=\sum_{j\in\mathbb Z} A_j(t)\frac{e^{i\frac{(x-j)^2}{4t}}}{\sqrt{t}}=e^{i\frac\pi 4}\sum_{j\in\mathbb Z} A_j(t)e^{it\Delta}\delta_j,
\end{equation}
where $\delta_j$ is the Dirac mass at $j$ and $t>0$, and in their counterpart at the level of the binormal flow via Hasimoto's transform. \\

\subsection{Deterministic results} The existence of this type of solutions \eqref{SumOfDeltaSolution} was proved in Theorem 1.3 in \cite{BVAnnPDE} for $s>\frac 12$, inspired by the same type of ansatz used in \cite{Kita} for the subcubic NLS. Recall that the critical spaces for \eqref{CubicNLS} are $\dot H^{-\frac 12}(\mathbb R)$ for the Sobolev scale and $\mathcal F L^\infty(\mathbb R)$ for the Fourier-Lebesgue scale, i.e. tempered distributions wih Fourier transform in $L^\infty$. 
If $\{A_j(t)\}_{j\in\mathbb Z}\in l^1$ then $e^{it\xi^2}\hat{u}(t,\xi)$ is $2\pi$-periodic and $u(t)\in H^{-\frac 12^+}(\mathbb R)\cap \mathcal F L^\infty(\mathbb R)$\footnote{By $H^{-\frac 12^+}$ 
we denote $\bigcap_{\varepsilon >0} H^{-\frac 12 - \varepsilon}$. 
}, so $u$ is almost critical on the Sobolev scale and critical on the Fourier-Lebesgue scale. Also, if $\{A_j(t)\}_{j\in\mathbb Z}\in l^{2,s}$, then $e^{-i\frac{x^2}{4t}}u(t,x)$ is $4\pi t$-periodic and belongs to $H^s_{per}(0,4\pi t)$. Here for a sequence $\{ a_j \}_{j \in \Z}$ of complex numbers, we will abbreviate to 
$\{ a_j \}$ and we define its $l^{2,s}$ norm as
$$
\| \{ a_j \} \|_{l^{2,s}}^2 := \sum_{j \in\Z } \langle j \rangle^{2s} |a_j|^2.
$$
In particular $u(t)$ defined in \eqref{SumOfDeltaSolution} is in $H^{s}_{loc}(\mathbb R)$,  provided $\{A_j(t)\}_{j\in\mathbb Z}\in l^{2,s}$. Indeed, for $s$ integer this is obtained by reduction to a finite number of intervals of size $4\pi t$, and the cases when $s$ is not an integer follow by interpolation. 
 The solutions \eqref{SumOfDeltaSolution} focalise at $t=0$ and, at least in the case in which the 
coefficients $A_j$ decay sufficiently fast, i.e. $\{A_j(t)\}\in l^{2,s}$ with $s>\frac 32$, were used in \cite{BVAnnPDE} to construct solutions of the binormal flow that are infinite curves exhibiting 
corners in the limit $t \to 0$. This strong decay of $A_j$, that we call in this paper ``the regular case", implies local regularity of the corresponding binormal flow solutions. 
%
Now we recall the main results from \cite{BLTV}, starting with the following well-posedness result for the cubic NLS equation \eqref{CubicNLS}. The construction of the solutions will be briefly described in \S 2.2-\S2.3.
\begin{theorem}[Theorem 1.4 and Lemma 2.4 in \cite{BLTV}]\label{wbl}
Let $s >0$. For any $\{ a_j \} \in l^{2,s}$ there exists a unique (in the sense of Section \ref{ssectconstr}) 
sequence of functions $\{ A_j(t) \} \in C((0, 1] ; l^{2,s})$ such that \eqref{SumOfDeltaSolution} is a solution of the
cubic NLS equation \eqref{CubicNLS} on $t \in (0,1]$ with  initial data
prescribed, at time $t =1$, by $\{ A_j(1) \} = \{ a_j \}$.  
The coefficients are bounded uniformly in time 
\begin{equation}\label{est2s}
\sup_{t\in(0,1]}\|\{A_j(t)\}\|_{l^{2,s}}\leq C(\|\{a_j\}\|_{l^{2}})\|\{a_j\}\|_{l^{2,s}},
\end{equation}
but  they fail to converge in the limit $t \to 0$ (although the modulus $|A_j(t)|$ converges, for all~$j \in \Z$).
More precisely, there exists a sequence $\{ \alpha_j\} \in l^{2,s}$ such that
\begin{equation}\label{blupth}
|A_j(t)-e^{ i(|\alpha_j|^2-2\mu)\log t}\alpha_j|\leq C(\|\{a_j\}\|_{l^{2,s}}) \,t, \quad \forall j\in\mathbb Z, t\in(0,1),
\end{equation}
where 
$$
\mu := \sum_{j \in \Z} |A_j(t)|^2 = \sum_{j \in \Z} |a_j|^2
$$
is a conserved quantity.
\end{theorem}
Using the Hasimoto procedure we can construct  solutions of the Schr\"odinger map and of the binormal flow equation starting from the solutions of NLS constructed in Theorem~{f{wbl}}, for $t\in (0,1] $.  This procedure is easily justified as long as the  coefficients $\{a_j\}$ decay sufficiently fast. For instance, it is sufficient that  $\{a_j\} \in l^{2,s}$ with~$s >3/2$ (see Section \ref{Sec:BinormalFlow}). This allows to define a solution map on the space of the  curves. More precisely, once we restrict the solution map between the initial time $t=1$ and the final time $t$, we have the  mapping:
\begin{equation}\label{HasimotoFlowForUs}
\Psi_{1,t} : \Omega_1^{s} \to \Omega_t^{s},  \qquad t \in (0,1],
\end{equation}
where 
$\Omega_t^{s}$ is the set of arclength parametrized curves $\gamma$ having a filament function $u_\gamma$ such that $e^{-i\frac{x^2}{4t}}u_\gamma(x)$ is $4\pi t$-periodic and belongs to $H^{s}_{per}([0,4\pi t])$. This set of curves is included in $H^{2+s}_{loc}(\mathbb R, \R^3)$ and it contains the curves 
with $4\pi t$-periodic curvature $c \in H^{s}_{per}([0,4\pi t])$ and with torsion~$\tau = \frac{x}{2t} + \eta$ where 
$\eta \in H^{s}_{per}([0,4\pi t])$ is also $[0,4\pi t]$-periodic. The notion of filament function, the set $\Omega_t^s$ together with its topology will be described in \S \ref{sectfil}. We will refer to the map \eqref{HasimotoFlowForUs} as the Hasimoto solution map. Using this vocabulary, Theorem ~{1.6} of \cite{BLTV} rewrites as follows.
\begin{theorem}[\cite{BLTV}]\label{thbfPreq}
Let $s > 3/2$.  Then the Hasimoto solution map~\eqref{HasimotoFlowForUs} is well defined.  More precisely,  there exists a unique continuous map 
$$
\Psi_{1,t} : \Omega_1^{s} \to \Omega_t^{s},  \qquad \forall t \in (0,1]
$$
such that for every  $\chi_1\in\Omega_1^s$,
$$
\chi(t):=\Psi_{1,t} (\chi_1)\in C((0,1]; H^{2+s}_{loc}(\R, \R^3)), \chi_{xx}(t)\in H^{-\frac 12^+}(\mathbb R, \R^3),
$$
and $\chi$ solves the binormal flow  equation on $(0,1]$ and attains the initial condition $\chi(1)=\chi_1$.
The curve $\chi(t)$ is $\mathcal C^{\frac 32+s}_{loc}$ and converges uniformly w.r.t. the 
arc-length parametrization,  in the limit $t \to 0$, to a polygonal line $\chi(0)$ (in particular, it generates several corner-singularities at time $t=0$):
$$
|\chi(t,x)-\chi(0,x)|\leq C\sqrt{t}, \quad \forall x\in\mathbb R,
$$
where $C$ depends on $\|e^{-i\frac{y^2}4}u_{\chi_1}(y)\|_{H^s_y([0,4\pi])}$, 
and the trajectories in time $\chi(\cdot, x)$ have H\"older regularity $\mathcal C^{\frac 12}$.

 Moreover, for any polygonal line $\chi_0$ with angles $\theta_j$ at the integer locations $x=j$ such that $\{\pi -\theta_j\}_{j\in\mathbb Z}\in l^{2,s}$, there exists a {\it unique} curve $\chi(t)\in\Omega^s_t$ that solves the binormal flow 
 equation on $(0,1]$ converging to $\chi_0$ when $t \to 0$. 
\end{theorem}

With respect to the quoted previous results for the binormal flow here we have strong solutions at a very much lower global in space regularity, $\chi_{xx}(t)\in H^{-\frac 12^+}(\mathbb R)$, that is actually at the critical level. The global regularity in Sobolev spaces of $\chi_{xx}$ reflects the behavior of the curve at infinity. Moreover, we also have a description of the limit of the curves $\chi(t)$ at the time when the NLS solution ceases to be smooth. Finally, let us note that the H\"older regularity $\mathcal C^\frac 12$ of the trajectories in time $\chi(\cdot, x)$ is the one exhibited by the self-similar solutions detailed in \cite{GRV}. \\
 
The main goal of this paper is to strongly relax the regularity/summability assumption $s > 3/2$ in Theorem \ref{thbfPreq}. To do so, we need to
 generalize the notion of solution of the binormal flow equation. Let
 $I$ be an open time interval and $s>0$.   We say that a curve $\chi \in C(I; H^{2+s}_{loc}(\R, \R^3))$ is a weak solution of the binormal flow equation on 
$(I\times \R )$ if
\begin{equation}\label{eq:WeakBinInTHMIntroduction}
- \int \chi \cdot \phi_t  \,dxdt = \int ( \chi_x \wedge \chi_{xx} ) \cdot \phi \,dxdt, \qquad  
|\chi_x(t,x)| = 1,
\end{equation}
for all test vectors $\phi$ with components in  $C^{\infty}_c (I\times \R, \R^3)$.
We will prescribe the initial condition at time $t=1 \in I$ by considering a curve $\chi_1 \in \Omega_1^s\subset H^{2+s}_{loc}(\mathbb R, \R^3)$ and we will
look for weak solutions of the binormal flow $\chi$ such that $\chi(1)=\chi_1$. 
\\  
 
We are now ready to state our first result, that extends Theorem \ref{thbfPreq} 
to lower values of~$s$.

\begin{theorem}\label{thbf} 
The  Hasimoto solution map 
$$
\Psi_{1,t} : \Omega_1^{s} \to \Omega_t^{s},  \qquad t \in (0,1]
$$
can be continuously extended for $s>0$. More precisely, for every $\chi_1\in \Omega^s_1$, $s>0$ let 
 $\chi_1^n\in \Omega^\sigma_1$, $\sigma>3/2$, $n\in\N$ be a sequence of curves converging to  $\chi_1$ in $\Omega^s_1$.
 Set 
 $$
\chi_n(t)=\Psi_{1,t} (\chi_1^n), \quad t\in (0,1].
$$
Then for every $t\in (0,1]$, $\chi_n(t)$ converges to a limit curve $\chi(t)$ in $ \Omega^s_t$. In addition, $(\chi_n(t))_{n\in\N}$ also converges to $\chi(t)$ in the space-time norm $ C((0,1]; H^{2+s}_{loc}(\R, \R^3))$. Also, $\chi(t)$ is $\mathcal C^{\frac 32+s}_{loc}$ and is a weak solution of the binormal flow  equation on~$(0,1]$ with $\chi(1)=\chi_1$.
 
 Moreover, if $s>1/6$ the curve 
 $\chi(t)$ converges to a continuous curve as $t \to 0$, uniformly w.r.t. the 
arc-length parametrization, and  the trajectories in time $\chi(\cdot, x)$ have H\"older regularity $\mathcal C^{\frac 12}$. 
\end{theorem}

In this result the limit curves $\chi(t)$ have essentially the same local regularity as the ones in \cite{ChErTz}, but we allow more general behavior at infinity. In addition,  for $s>\frac 16$ we have a qualitative description of the dynamics of the curve. 
It is worth noticing that when $s>3/2$ the binormal flow solutions in Theorem \ref{thbf} are the ones in Theorem \ref{thbfPreq}, thus have a limit curve as $t \to 0$ identified to be a polygonal line. Moreover, they enter the framework of \cite{BVAnnPDE} where they were also proved to have converging tangent vectors as $t \to 0$. An interesting question, that will not be further investigated in this paper, is whether these properties remain true for lower values of $s$. \\

\subsection{Probabilistic results} 
Our second  goal is to provide a probabilistic counterpart of Theorem \ref{thbf} for low regularity random curves. 
An interesting aspect of this probabilistic approach is that it will allow us to relax the regularity assumption $s > 1/6$, 
needed for the existence of a continuous limit curve at $t=0$, essentially down to $s>0$.  
The idea is to randomize the coefficients $a_j = A_j (t)_{\vert_{t=1}}$ of the ansatz \eqref{SumOfDeltaSolution} and exploit decorrelation properties. 
Such an idea goes back to the work by  Paley-Zygmund   \cite{PZ} on Fourier series in 1930. 
Much later a similar idea was used in the context of PDE's in \cite{B94_pak,Bo96,BT1,BT2, CO} and many subsequent works. 
In our context, the randomization is described as follows. Let~$\{g_j^{\omega}\}_{j\in\Z}$ be a sequence of independent standard complex Gaussians (see \eqref{DefComplexGaussians} below) on a probability space $(\Omega, \mathcal A,p)$.   Set 
\begin{equation}\label{CoeffRandomized}
a_j^{\omega} := \frac{g^{\omega}_j}{(1 + |j|^{2s +1} )^{\frac12}}\,\, e^{i \frac{j^2}{4}}.
\end{equation}
Given any $s' < s$, the sequence of random variables $\{ a_j^{\omega} \}_{j \in \Z}$ induces a measure 
on the space of the complex valued sequences belonging to $l^{2, s'}$.
Moreover, one has $\omega$-almost surely $\| \{ a_j^{\omega} \} \|_{l^{2, s}} = \infty$ (see e.g.  \cite[Lemma~B1]{BT1}). In this sense, we are working  $\omega$-almost surely at the $l^{2, s-}$ level of summability/regularity. 
We define our randomized initial datum for NLS as
\begin{equation}\label{SumOfDeltaSolutiondfnjskdg}
u^{\omega}(1,x)= e^{i \frac{x^2}{4}} \sum_{j\in\mathbb Z} a_j^{\omega}  e^{-i \frac{xj}{2}} ,
\end{equation}
and the related $4 \pi$-periodic random variable
\begin{equation}\label{dfkslkgjdksl}
Z(\omega, x) :=  \sum_{j\in\mathbb Z} a_j^{\omega}  e^{-i \frac{xj}{2}} .
\end{equation}
Since $Z(\omega, x)$ is $4\pi$-periodic, almost surely in $H^{s'}(0,4\pi)$, if $s'>0$ we can use it as a filament function, see \S \ref{sectfil}, and define an event $E_1^{s'}\subset \Omega$ such that  $p(E_1^{s'})=0$ and for every $\omega\in \Omega\backslash E_1^{s'}$  there is a curve  $\chi^{\omega}_1\in \Omega_1^{s'}$ with  filament function $e^{i\frac{x^2}4}Z(\omega, x)$. Therefore we can apply Theorem~\ref{thbf} and define a family of curves 
$$
\chi^{\omega}(t)\in \Omega^{s'}_t , \quad t\in (0,1],\,\, \omega \in E_1^{s'},
$$
which solves  the binormal flow  equation on $(0,1]$ with $\chi^{\omega}(1)=\chi^{\omega}_1$.
For $s'<1/6$, Theorem~\ref{thbf} does not give any information about the behavior of $\chi^{\omega}(t)$ in the limit $t\rightarrow 0$. 
Our  second main result addresses this issue as shown by the next statement. 
 \begin{theorem}\label{thbf2}
 Let $s>s'>0$ and let $\chi^{\omega}(t)$ be defined as above. 
 Then there is an event $E_2^{s'}$ such that  $E_1^{s'}\subset E_2^{s'} $ and $p(E_2^{s'})=0$ such that for every $\omega\in \Omega\backslash E_2^{s'}$ 
 the curve $\chi^{\omega}(t)$ belongs to $C^{2+s'}(\R)$ and  converges to a continuous curve as $t \to 0$, uniformly w.r.t. the  arc-length parametrization. Moreover the trajectories in time  $\chi^\omega(\cdot, x)$ have H\"older regularity $\mathcal C^{\frac 12-\varepsilon}$ for all $x\in\mathbb R$ and $\varepsilon>0$.
\end{theorem}
One may show that the family of curves 
$
\{\chi_1^\omega,\,\omega\in  \Omega\backslash E_2^{s'}\} $ concerned by the statement forms a dense set in $\Omega_1^{s'}$ (cf. \cite{BT2,Tz24} where similar analysis is performed).  
One may also show that for $s<1/6$ the curves at time $1$ satisfy
$$
\{\chi_1^\omega,\,\omega\in  \Omega\backslash E_2^{s'}\}\cap \Omega_1^{\frac 1 6}=\emptyset,
$$
i.e. the result of Theorem  \ref{thbf2} cannot be deduced from Theorem~\ref{thbf}.
\\

The proof of Theorem \ref{thbf2} uses methods developed in the field of probabilistic well-posedness of dispersive PDE's. 
It is worth noticing that  the additional $\mathcal C^{2+s'}$ regularity of the curves and the  $\mathcal C^{\frac 12-\varepsilon}$ regularity of the trajectories 
result from probabilistic considerations.
\\

It is also worth mentioning that different probabilistic approaches to the study of vortex filaments and 
Schr\"odinger map 
can be found in \cite{FG1, FG2} and \cite{GH}, respectively.  
\\ 

We now turn to an intermediate result, used in the proof of Theorem \ref{thbf2}, that we consider of independent interest. This will also give a more precise  idea of which probabilistic techniques are employed in the proof of Theorem \ref{thbf2}.
The starting point for the analysis of the behavior in the limit $t \to 0$
is to notice that 
the pseudo-conformal transformation \eqref{Eq:Pseudoconf} maps a solution
 $u(t)$ of the cubic NLS \eqref{CubicNLS} on $0<t\leq 1$ into a periodic  
 solution $v(t)$ of the following modified 
 cubic NLS on $1\leq \tau <\infty$:
\begin{equation}\label{NLSvIntro}
iv_{\tau}+v_{xx}+\frac {1}{\tau}|v|^2v=0, \qquad x \in  \T. 
\end{equation} 
We shall present only the NLS results for the focusing case (which is the one linked to the binormal flow derived from fluids), but the defocusing case can be considered also with simpler proofs. 
We note that the modified equation \eqref{NLSvIntro} formally converges to the free equation in the asymptotic limit $\tau \to \infty$. Thus, even if the solutions are periodic in space, due to the decaying coefficient $\frac 1\tau$, one may expect to recover some scattering type results in the limit $\tau \to \infty$.
This is the reason why this model is particularly interesting. Indeed, in order to prove Theorem \ref{thbf2}, we will
need to construct a quasi-invariant measure for the modified equation \eqref{NLSvIntro}. This is done in Theorem \ref{MainThm3} below. The interesting feature is that in \eqref{fndjskjnfgjdskjdng} below we are able to identify the asymptotic behavior  
of the transported by the flow measure (as $\tau \to \infty$). 
As recent work on quasi-invariant Gaussian measures for Hamiltonian PDE's on bounded domains revealed, 
it is usually not possible to identify a limit measure when quasi-invariance holds.  
A notable exception is \cite{BT} which also studies quasi-invariant measures after a conformal transformation.  In \cite{BT} one can obtain the limit of the corresponding Radon-Nikodym derivatives because of the closeness to a Gibbs measure. In our context the limit is obtained because of the closeness to the linear flow which leaves invariant the measure $\gamma_s$ which will be defined in the next paragraph.
\\ 

We therefore  introduce the Gaussian measure~$\gamma_s$ induced on $H^{s'}(\T)$, $s' < s$, by the map
\begin{equation}\label{TMIATSIntro} 
\omega \longmapsto   \sum_{k \in \mathbb{Z}} 
\frac{g_k^{\omega}  }{(1 + |k|^{2s+1})^{\frac12} } \, e^{i x k}.
\end{equation}
For details on this construction see the beginning of \S 5. 
The random series \eqref{TMIATSIntro} is $\omega$-almost surely in $C^{s'}(\T)$ for all $s' < s$
while its $H^{s}(\T)$ norm is $\omega$-almost surely infinite (see e.g. \cite{BT1,BT2}).
Then we define, for $M>0$ fixed,
\begin{equation}\label{Def:Rho(1)Intro}
d \rho_s(v) = \mathds{1}_{\left\{ \| v \|^2_{L^{2}(\mathbb{T})} \leq M \right\}} (v)d \gamma_s(v),
\end{equation}
namely we have introduced a rigid cut-off on the size of the $L^2(\T)$ norm of $v$. It is also worthy mentioning that 
the
$L^2(\T)$ norm is conserved under the solution map $\Phi_{1,\tau}$ of equation \eqref{NLSvIntro} 
(with this notation we want to stress that this is the 
solution map between the initial time $\tau=1$ and the final time $\tau >1$).  
Our goal is to study the evolution of~$\rho_s$ under the map $\Phi_{1,\tau}$ and in particular to show its 
quasi-invariance. By this we mean that zero measure sets remains of zero measure under the evolution, namely 
\begin{equation}\label{QuasiInvIntro}
\rho_s(A) =0 \Rightarrow \rho_s(\Phi_{1, \tau} (A)) =0, \qquad \forall \tau >1,
\end{equation}
for all Borel sets $A$ in the $H^{s'}(\T)$ topology, where $s' < s$. We will prove  the following more precise statement. 
\begin{theorem}\label{MainThm3}
Let $s \in (0,1)$. The measure  $\rho_s$ is quasi-invariant under the map $\Phi_{1,\tau}$, in the sense of \eqref{QuasiInvIntro}. 
In addition, for all $\kappa>0$ there exists a constant $C_\kappa$ depending only on $\kappa$ such that
\begin{equation}\label{QuasiInvFin}
\rho_s \left( \Phi_{1,\tau} (A) \right) \leq C_{\kappa} \rho_s  ( A )^{1 - \kappa},
\end{equation}
for all $\tau \geq 1$ and all Borel sets $A \subseteq H^{s'}(\T)$, $s' < s$, namely Borel sets with respect to the $H^{s'}$ topology. 
The transport by $ \Phi_{1,\tau} $ of  $\rho_s$  is absolutely continuous with respect to $\rho_s$ with a Radon-Nikodym derivative  
\begin{equation}\label{QuasiInvFinDensity}
f(\tau, v) 
 := e^{ 2  {\rm Im\,}  \int_1^{\tau} \int_{\T} |\Phi_{1,\lambda}(v)|^2 \Phi_{1,\lambda}(v)  |D|^{2s+1} 
 \overline{\Phi_{1,\lambda}(v)} \, \frac{d \lambda}{\lambda} dx},
 \end{equation}
 where $|D|^{2s+1}$ denotes here the Fourier multiplier\footnote{For a notational simplicity, in this paper we diverge from the usual definition of 
 $|D|$ denoting the  Fourier multiplier operator of symbol $|k| $. } operator of symbol $1 + |k|^{2s+1} $.
Namely
\begin{equation}\label{kklskjkdgnjsd}
\rho_s \left( \Phi_{1,\tau} (A) \right) = \int_A f(\tau,v) \, d \rho_s(v), 
\end{equation}
for all Borel sets $A \subseteq H^{s'}(\T)$, $s' < s$. Moreover, the densities have a limit as $\tau \to \infty$, namely 
 \begin{equation}\label{fndjskjnfgjdskjdng}
\lim_{\tau \to \infty} f(\tau, v)  =
e^{ 2  {\rm Im\,}  \int_1^\infty \int_{\T}  |\Phi_{1,\lambda}(v)|^2 \Phi_{1,\lambda}(v)  |D|^{2s+1} 
 \overline{\Phi_{1,\lambda}(v)} \, \frac{d \lambda}{\lambda} dx}.
\end{equation}
\end{theorem} 
Theorem \ref{MainThm3}  fits in the line of research initiated in \cite{sigma} aiming to study the transport of (infinite dimensional) Gaussian measures under the flow of Hamiltonian PDE's. 
For more details on this line of research we refer to 
\cite{STz2023, FT, K, CT, DT, FT2, GLT22, GLT23, GLT22bis, PTV, PTV2, OT, OT2, OST, OS, GOTW} and the references therein. 
A key novelty in the present work is the use of the global quantitative quasi-invariance estimate \eqref{QuasiInvFin} is order to show that the H\" older norms $C^{s'}$ of the individual trajectories are not growing faster that $t^\varepsilon$ for every $\varepsilon>0$ when $t\gg 1$. This phenomenon, crucially used in  the proof of  Theorem \ref{thbf2}, results from measure propagation global arguments, a large deviation argument in  H\" older norms in  combination with a local in time almost sure smoothing property of the nonlinear part of the solution. 
We believe that it is worth to extend the propagation of H\" older regularity arguments we introduce in this work to other dispersive PDE's.  As mentioned above, an interesting feature of  Theorem \ref{MainThm3} is that we can prove the existence of a limit measure. 
It would be interesting to decide how much this phenomenon may be extended to more general situations. 
\\

It is likely that using the methods of \cite{CO} one can prove a suitable extension of Theorem~\ref{MainThm3}. It is less clear how this extension applies to the binormal flow. We plan to address this issue in a future work. 
\\

We conclude the introduction with a final comment on the deterministic convergence statement in Theorem \ref{thbf}.
Looking at the proof (see the end of Section \ref{thbf}) we see that we could relax the assumption $s >1/6$ to $s>0$ for getting a continuous curve limit 
if we would have at our disposal the following periodic Strichartz estimate
\begin{equation}\label{PErStrInfty}
\| e^{it \partial_{xx}} \phi \|_{L^{4}_t L^{\infty}_x(\T)} 
\leq C_\varepsilon \| \phi \|_{H^{\varepsilon}(\T)}, \qquad \varepsilon >0. 
\end{equation}
Despite many recent progress on the decoupling theory for the paraboloid, generating from~\cite{BD}, the 
problem of establishing \eqref{PErStrInfty}, or disproving it, is still very much open. 
However, even if \eqref{PErStrInfty} is established, the H\"older regularity of the Gaussian curves established in the proof of Theorem \ref{thbf2} is not achieved.
\\

The rest of this paper is organized as follows. In the next section we collect some preliminary facts. 
In Section~3 we construct solutions of the binormal flow at low regularities in a general framework. In \S \ref{sect-thdet} we show that our data $\Omega_1^s$ with $s>0$ enters this general framework and prove the existence part of Theorem \ref{thbf}. In Section \ref{sect-thdet} we also prove the limiting behavior of these solutions under the restriction $s>1/6$, thus completing the proof of Theorem \ref{thbf}. 
In  Section~5 we prove the first part of Theorem~\ref{MainThm3}.  The existence of a limit density is proved in Section~6. Section~7 contains the proof of the control of the growth of H\"older norms of typical solutions. 
Finally, in Section~8 we prove Theorem \ref{thbf2}.\\

{\bf{Acknowledgements:}}  
We warmly thank the referee for the remarks which helped us to improve the quality of the manuscript. 
This research is partially supported as follows. VB is partially supported by the ERC advanced grant GEOEDP, by the French ANR project BOURGEONS. RL is supported by BERC program and by MICINN (Spain) projects Severo Ochoa CEX2021-001142, PID2021-123034NB-I00. NT is partially supported by the ANR project Smooth ANR-22-CE40-0017 and by the ERC advanced grant GEOEDP. LV is funded by MICINN (Spain) projects Severo Ochoa CEX2021-001142, and PID2021-126813NB-I00 (ERDF A way of making Europe), and by Eusko Jaurlaritza project IT1615-22 and BERC program.
\section{Preliminaries}\label{Sec:Setting}
In the first subsection we shall present the notion of filament function of a curve, the sets of curves $\Omega_t^s$ and their topology. 
Then in the following two subsections we briefly recall the construction of the NLS solution in Theorem \ref{wbl}. In \S 2.4 we give some related NLS continuity results.
\subsection{Filament functions of a curve and the sets $\Omega_t^s$.}\label{sectfil}
We call filament function of a curve $\gamma$ a function $u_\gamma$ obtained by the following parallel transport procedure. Denote the tangent vector $T=\partial_x\gamma$. We consider the parallel frames $(T,e_1,e_2)(x)$ obtained by solving the ODEs 
$$\partial_xe_1(x)=-\langle \partial_xT,e_1\rangle T(x),\quad  \partial_xe_2(x)=-\langle \partial_xT,e_2\rangle T(x),$$ 
with data at some $x_0$ say in $[0,1]$ given by an orthonormal frame of $\mathbb R^3$. This is always possible if the curve $\gamma$ is regular enough, for instance if $\partial_x^2\gamma\in L^2_{loc}$. Indeed, this gives global existence in $H^1_{loc}$ for the above ODEs. In addition, the orthonormal frame nature of $(T,e_1,e_2)(x)$ is preserved, since the matrix of the system of evolution in space of $(T,e_1,e_2)(x)$ is antisymmetric.  Then we define: 
$$u_\gamma=\langle \partial_x T,e_1\rangle +i\langle \partial_x T,e_2\rangle.$$ 
The real and imaginary part of $u_\gamma$ are a normal developement of the curve $\gamma$ (see \cite{Bishop}). We called them in this article filament function by refering to the notion introduced by Hasimoto \cite{H} for curves $\gamma$ with non-vanishing  curvature, that is a function determined by the curvature $c$ and torsion $\tau$ of the curve by the following formula: 
$$f_\gamma(x)=c(x)e^{i\int_{x_0}^x \tau(s)ds}.$$ 
For curves $\gamma$ with non-vanishing  curvature $u_\gamma=f_\gamma$ and 
\begin{equation}\label{geo_g}
(e_1+ie_2)(x)=(n+ib)(x)e^{i\int_{x_0}^x \tau(s)ds}
\end{equation}
if in the construction of $u_\gamma$ the initial orthonormal frame $(T,e_1,e_2)(x_0)$ is chosen to be the Frenet frame $(T,n,b)(x_0)$. 
Observe that even if the curvature vanishes, the expression in the right-hand side of \eqref{geo_g} continue to make sense via the parallel frame construction.  
\\

We note that the only degree of freedom in the filament function $u_\gamma$ construction is rotating the initial data $(T,e_1,e_2)(x_0)$, i.e. rotating $(e_1,e_2)$ in the plane orthogonal to $T(x_0)$, which yields by this construction another filament function that is of type $u_\gamma(x)e^{i\theta}$ (and changing $x_0$ boils down to the same argument). \\

Recall that we define $\Omega_t^{s}$ for $t>0$ to be the set of arclength parametrized curves $\gamma$ having a filament function $u_\gamma$ such that $e^{-i\frac{x^2}{4t}}u_\gamma(x)$ is $4\pi t$-periodic and belongs to $H^{s}([0,4\pi t])$.  In view of the previous paragraph the definition of the set $\Omega_t^s$ does not depend on the choice of the filament function. We note that the set $\Omega_t^{s}$ contains the curves 
with $4\pi t$-periodic curvature $c \in H^{s}([0,4\pi t])$ and with torsion~$\tau = \frac{x}{2t} + \eta$ where 
$\eta \in H^{s}([0,4\pi t])$ is also $[0,4\pi t]$-periodic. 
\\

We also note that $u_\gamma$ is constructed exclusively from $T=\partial_x\gamma$ so it does not depend on translations in space of $\gamma$, i.e. it does not depend on $\gamma(x_0)$. Moreover, $u_\gamma$ is uniquely modulo multiplication by $e^{i\theta}$ determined by $T(x_0)=\partial_x\gamma(x_0)$. Therefore if two curves $\gamma,\tilde\gamma\in\Omega_t^s$ satisfy $\gamma(x_0)=\tilde \gamma(x_0), \partial_x\gamma(x_0)=\partial_x\gamma(x_0)$ and $u_\gamma=u_{\tilde\gamma}$ then $\gamma=\tilde\gamma$. 
\\

Conversely, given $u\in L^2_{loc}$ we can construct a curve $\gamma_u$ which has $u$ as filament function, in the following way. First we construct frames $(T,e_1,e_2)(x)$ for all $x\in\mathbb R$ by solving the ODEs 
$$\partial_xT=({\rm Re}\, u) e_1+({\rm Im}\, u)e_2,\quad \partial_x e_1=-({\rm Re}\, u)T,\quad \partial_x e_2=-({\rm Im}\, u)T,$$
with data at say $x=0$ being the $\mathbb R^3$ canonical basis. Then we construct $\gamma_u$ as a curve having $T$ as tangent vector, for instance 
$$\gamma_u(x)=\int_0^x T(s)ds.$$
Modulo rotation and translation this is the unique curve having $u_\gamma$ as filament function.\\

We endow $\Omega_t^s$  with the topology induced by the distance
$$
d_{\Omega_t^s}(\gamma,\tilde \gamma)=\|\gamma-\tilde \gamma\|_{L^\infty([0,1])}+\|\partial_x\gamma-\partial_x\tilde\gamma\|_{L^\infty([0,1])}+\|e^{-i\frac{x^2}{4t}}(u_\gamma-u_{\tilde\gamma})\|_{H^{s}_{per}([0,4\pi t])}\,.
$$
The triangular inequality is straightforward. 
The fact that $d_{\Omega_t^s}(\gamma,\tilde \gamma)=0$ implies $\gamma=\tilde\gamma$ follows from the previous paragraph.
\\ 

Note that  $\gamma(x)-\tilde\gamma(x)$ can get  large when $x$ goes to infinity even if $d_{\Omega_t^s}(\gamma,\tilde\gamma)$ is small. However the topology induced by this distance on $\Omega_t^s\subset H^{s+2}_{loc}$ is equivalent to the one induced by the distance 
$$
\tilde d(\gamma,\tilde\gamma)=\sum_{m\in\mathbb N^*}2^{-m}\frac{\|\gamma-\tilde\gamma\|_{H^{2+s}([-m,m])}}{1+\|\gamma-\tilde\gamma\|_{H^{2+s}([-m,m])}} \,.
$$  
The fact that $\Omega_t^s\subset H^{s+2}_{loc}$ follows from the definition of a filament function $u_\gamma$ which ensures that if $u_\gamma\in H^s_{loc}$ then $\gamma\in H^{2+s}_{loc}$ and vice versa.
\\

Finally, we have that for $s_1\geq s_2$ and $t \in (0,1]$, the space  $ \Omega_t^{s_1}$ is continuously embedded and dense in  $ \Omega_t^{s_2}$.
\subsection{The NLS equations}
We are interested in critical regularity solutions on $t>0$ of the type \eqref{SumOfDeltaSolution}:
$$
u(t,x)=\sum_j \, A_j(t)\frac{e^{i\frac{(x-j)^2}{4t}}}{\sqrt{t}},
$$
where we fix the final condition at time~$t_0 = 1$.  The pseudo-conformal transformation 
\begin{equation}\label{Eq:Pseudoconf}
u(t,x)= pc(v)(t,x) := \frac{e^{i \frac{x^2}{4t}}}{\sqrt{t} }  \bar{v} \left( \frac{1}{t}, \frac{x}{t} \right)
\end{equation}
connects the solution $u(t)$ of \eqref{CubicNLS} on $0<t\leq 1$ and with data at time $1$ to $v(t)$, solution of the following 1-D cubic Schr\"odinger equation  on $1\leq t<\infty$:
\begin{equation}\label{NLSv}
iv_t+v_{xx}+\frac {1}{t}|v|^2v=0,
\end{equation} 
with initial condition at time $1$. 
Moreover, modulo constants, that we avoid for the sake of the
clearness of the presentation, and that do not affect the arguments and results, $v$ is $2\pi-$periodic with 
\begin{equation}\label{vB}
v(t,x)=\sum_j B_j(t) e^{itj^2}e^{ixj},
\end{equation}  
where
$$
B_j(t)=\overline{A_j}(\frac 1t).
$$
In particular,  the Fourier coefficients of $v(t)$ satisfy $|\widehat{v(t, \cdot)}(j)|=|A_j(1/t)|$ thus estimating Sobolev norms of $v(t)$ is equivalent to weighted estimates on the sequence $A_j(1/t)$. Also, we can note that as usually for questions concerning the long-time behavior of cubic NLS, we have the relation:
$$
B_j(t)=\mathcal F(e^{-it\Delta}v(t))(j).
$$
We denote then
$$
B(t,x):=\sum_j B_j(t)e^{ijx}=e^{-it\Delta}v(t)(x),
$$
and noting that
$$v(t)=\sum_j(e^{itj^2}B_k(t))e^{ijx}\Longrightarrow \hat v(\tau,j)=\hat B_j(\tau+j^2),$$
we introduce Bourgain's norms:
\begin{equation}\label{B}
\|v\|^2_{X^{s,b}}:=
\int \sum_j \langle j\rangle^{2s} \langle \lambda+j^2\rangle^{2b} |\widehat{v}(\lambda,j)|^2d\lambda
=
\int \sum_j \langle j\rangle^{2s} \langle \lambda\rangle^{2b} |\widehat{B}_{j}(\lambda)|^2d\lambda
=:\|B\|^2_{H^{s,b}},
\end{equation}
and the localized version to a given time interval $I$ as follows 
\begin{equation}\label{Bloc}
 \|v\|_{X_I^{s,b}}:= \inf_{\tilde v = v _{\vert_{I \times \T}}} \|\tilde v \|_{X^{s,b}}
 = \inf_{\tilde B = B _{\vert_{I \times \Z}}} \Big(\int \sum_j \langle j\rangle^{2s} \langle \lambda\rangle^{2b} |\widehat{\tilde B}_{j}(\lambda)|^2d\lambda\Big)^\frac 12
=:\|B\|_{H^{s,b}_I}
\end{equation}
The sequence $A(t)=\{A_j(t)\}_{j\in\mathbb Z}$ is solution of the system
\begin{equation}\label{Ajsyst}i\partial_t A_k(t)=\frac{1}{ t}\sum_{k-j_1+j_2-j_3=0}e^{-i\frac{k^2-j_1^2+j_2^2-j_3^2}{4t}}A_{j_1}(t)\overline{A_{j_2}(t)}A_{j_3}(t),\end{equation}
and the sequence $B(t)=\{B_j(t)\}_{j\in\mathbb Z}$ is solution of the system (modulo constants)
\begin{equation}\label{Bjsyst}i\partial_t B_k(t)=\frac{1}{ t}\sum_{k-j_1+j_2-j_3=0}e^{-it(k^2-j_1^2+j_2^2-j_3^2)}B_{j_1}(t)\overline{B_{j_2}(t)}B_{j_3}(t),\end{equation}
Let us introduce the nonresonant set
$$NR_k=\{(j_1,j_2,j_3),\, k-j_1+j_2-j_3=0,\, k^2-j_1^2+j_2^2-j_3^2\neq 0\}.$$
First we note that if $k-j_1+j_2-j_3=0$ we have 
$$\omega_{k,j_1,j_2}:=k^2-j_1^2+j_2^2-j_3^2=2(k-j_1)(j_1-j_2).$$ 
Thus $a(k)-a(j_1)+a(j_2)-a(j_3)$ vanishes on the resonant set for any function $a$. In particular, for any real function $a$ we have
\begin{equation*}
\label{Bjcons}
\partial_t \sum_k a(k)|B_k(t)|^2=\frac{1}{2 ti}\sum_{k-j_1+j_2-j_3=0}(a(k)-a(j_1)+a(j_2)-a(j_3))e^{-it\omega_{k,j_1,j_2}}B_{j_1}(t)\overline{B_{j_2}(t)}B_{j_3}(t)\overline{B_k(t)}
\end{equation*}
$$
=\frac{1}{2ti}\sum_{k;NR_k}(a(k)-a(j_1)+a(j_2)-a(j_3))e^{-it\omega_{k,j_1,j_2}}B_{j_1}(t)\overline{B_{j_2}(t)}B_{j_3}(t)\overline{B_k(t)}.
$$
Therefore by taking $a\equiv1$ we obtain that the system conserves the ``mass" :
\begin{equation*}
\sum_k|B_k(t)|^2=\|v(1)\|_{L^2(0,2\pi)}^2=:M.
\end{equation*}
In particular \eqref{Bjsyst} can be written as
\begin{equation*}
i\partial_t B_k(t)=\frac{1}{ t}\sum_{NR_k}e^{-it\omega_{k,j_1,j_2}}B_{j_1}(t)\overline{B_{j_2}(t)}B_{j_3}(t)+\frac 1t\left(2M-|B_k(t)|^2\right)B_k(t).\end{equation*}
\subsection{Construction of the solution $v(t)$ on $[1,\infty)$}\label{ssectconstr}
Let $\{b_k\}\in l^2$ and let $M=\sum_k|b_k|^2$. Let partition $[1, \infty]$ into intervals $I_j$ of length $\delta = \delta(M)$ to be specified. Let then $\eta_j$ a smooth positive compactly supported function such that $\eta_j =1$ on $I_j$.
For all $j \in\mathbb N^*$ we construct $\{B_{j}^\nu\}_{k\in\mathbb Z}$ recursively as follows. First we construct 
as $\{B_{k}^1\}_{k\in\mathbb Z}$ as the unique solution of:
$$i\partial_t B_{k}^1=\frac {1}t \eta_1(t)\sum_{j_1-j_2+j_3-k=0}e^{it\omega_{k,j_1,j_2}}B_{j_1}^1\overline{B_{j_2}^1}B_{j_3}^1(t),\quad B_{k}^1(1)=b_k,$$
over $\mathbb R$. This is possible invoking the Bourgain theory \cite{B94} for $\delta=\delta(M)$.
We note that $\{B_{k}^1\}_{k\in\mathbb Z}$ solves on $[1,1+\delta]$ the equation we are interested in:
$$
i\partial_t B_{k}^1=\frac {1}t \sum_{j_1-j_2+j_3-k=0}e^{it\omega_{k,j_1,j_2}}B_{j_1}^1\overline{B_{j_2}^1}B_{j_3}^1(t).
$$
Now we construct $\{B_{k}^2\}_{k\in\mathbb Z}$ solution of 
$$i\partial_t B_{k}^2=\frac {1}t \eta_2(t)\sum_{j_1-j_2+j_3-k=0}e^{it\omega_{k,j_1,j_2}}B_{j_1}^2\overline{B_{j_2}^2}B_{j_3}^2(t),\quad B_{k}^2(1+\delta)=B_{k}^1(1+\delta).$$
And so on. Then
$$B_k(t):=\left\{\begin{array}{c}B_{k}^1(t),\forall t\in I_1,\\B_{k}^2(t),\forall t\in I_2,\\\mbox{etc}\end{array}\right. $$
solves \eqref{Bjsyst} on $[1,\infty)$ with $B_k(1)=b_k$, and $v$ defined as in \eqref{vB} solves \eqref{NLSv} on $[1,\infty)$ with $v(1,x)=\sum_j b_j e^{ij^2}e^{ixj}$.
\\

Let us also recall that from Bourgain's contruction we have a local control of the Strichartz norm. For instance, for any interval 
$I_j = [j \delta, (j+1)\delta]$, we have a uniform in $j$ bound given by the $L^2$ norm:
\begin{equation}\label{Str}
\|v\|_{L^4([j \delta, (j+1)\delta];L^4[0,2\pi])}\leq C(\|v(1)\|_{L^2[0,2\pi]}).
\end{equation}
Similarly we have, for all $\varepsilon >0$:
\begin{equation}\label{Str2}
\|v\|_{L^6([j \delta, (j+1)\delta];L^6[0,2\pi])}\leq C(\varepsilon, \|v(1)\|_{L^2[0,2\pi]})\|v(j\delta)\|_{H^{\varepsilon}[0,2\pi]}.
\end{equation}
Comparing \eqref{Str} and \eqref{Str2}, we note that a priori in \eqref{Str2}  the right hand-side depends on the solution at time $j\delta$ while for the 
$L^2$ estimate \eqref{Str} we can go back until the initial datum, since the $L^2$ norm is conserved.  
%
However, exploiting the presence of the $\frac 1t$-coefficient in the equation \eqref{NLSv}, we showed in \cite{BLTV} that we actually have the uniform in time control \eqref{est2s}, i.e. 
$$
\sup_{t\geq 1}\|\{B_j(t)\}\|_{l^{2,s}}\leq C(\|\{B_j(1)\}\|_{l^{2}})\|\{B_j(1)\}\|_{l^{2,s}}.
$$
\subsection{Continuity estimates}\label{sect-cont}
Combining the local existence theory of strong solutions developed above and the global control \eqref{est2s} 
of the Sobolev/weighted-$\ell^{2}$ norms given by Theorem ~{\ref{wbl}}, one can prove global continuity of the 
flow map. More precisely, for two solutions $v,\tilde v$ of \eqref{NLSv} then for $T>1, s\geq 0$ we have
\begin{equation}\label{contv}
\sup_{t\in [1,T]}\| v(t) - \tilde v(t) \|_{H^s} \leq C(T,\|v(1)\|_{L^2},\|\tilde v(1)\|_{L^2})
\| v(1) - \tilde v(1) \|_{H^s}.
\end{equation}
  This is standard and can be obtained following for instance the arguments in  \S 3.5.1 in \cite{ErTz} based on estimates in Bourgain spaces. In the same spirit we state a result for the whole family of approximated systems \eqref{NLSTruncated}, that will be of use in \S \ref{Sec:QuasiInvariance}. 
The proof of the local well-posednes and of the continuity estimates is indeed independent of the truncation parameter     
$N \in \N \cup  \{\infty \} $ ($N= \infty$ correspond to the actual solutions introduced in the previous subsection).
\begin{lemma}\label{InverseFlowStab}
Let $T>0$. Let $s \geq 0$ and let $v^N(t)$ and $w^N(t)$ solutions of the truncated system \eqref{NLSTruncated} with initial conditions 
$v^N(\tau)$ and $w^N(\tau)$. Assume that $\| w^N(\tau) \|_{L^2}, \| v^N(\tau) \|_{L^2} \leq M$, $\tau\in[1,T]$.
 One has for all $1\leq t,\tau \leq T$ the regularity estimate
$$
\| v^N(t) - w^N(t) \|_{H^s} \leq C(T,M)
\| v^N(\tau) - w^N(\tau) \|_{H^s}.
$$
In terms of the solution map this writes
\begin{equation}\label{InvFlowPolyTruncated}
 \| \Phi^{N}_{\tau,t} v - \Phi^{N}_{\tau,t} w \|_{H^s} \leq C(T,M)
 \| v -  w  \|_{H^s},  
  \qquad 1\leq t,\tau \leq T.
\end{equation}
\end{lemma}
\section{Construction of the binormal flow at low regularity}\label{Sec:BinormalFlow}
The goal of this section is to make sense of the Schr\"odinger map and of the binormal flow 
starting from solutions of the cubic NLS at low regularity. The main results are Theorems~\ref{ThmWeak1}-\ref{MainThhBinormal}. One can compare it with 
Theorem~3.1 in \cite{NSVZ}, that holds under slightly weaker assumptions (the convergence of NLS smooth solutions $u^n$ to $u$ in $L^2_{t,x,loc}$) but provides 
less information (SM solution $T\in L^2H^1_{loc}\cap \mathcal CL^2_{loc}$). Suitable uniqueness of weak solutions of SM in the periodic setting in $H^s$, $s\geq 1$ was proved in \cite{ChErTz} without this assumption at the NLS level, and the solutions obtained are weakly continuous in $H^s$. Here we suppose a stronger assumption then in \cite{NSVZ}, namely \eqref{AssWeak1} ($\phi u^n$ converges to $\phi u$ in $L^\infty L^2$ for all test functions $\phi$) and get the stronger property $T\in\mathcal C H^1_{loc}$. The proof, as well as the ones in \cite{NSVZ} and \cite{ChErTz} are based on the Hasimoto approach using parallel frames, used first in \cite{Ko} to avoid the issue of vanishing curvature curves for which the use of Frenet frames as in the original Hasimoto transform is not possible. \\

In this section we set the initial time to be~$t_0 \in (0, \infty)$ and we look for weak 
solutions of the Schr\"odinger map and of the binormal flow in an open interval $I$ containing~$t_0$.
Later, in our applications, we will take $t_0 = 1$ and $I =(0, a]$ with $a \geq 1$. 
\subsection{The Hasimoto approach}\label{sectHas}
First we recall the classical Hasimoto approach to construct binormal flow solutions by using sufficiently regular solutions of the cubic NLS equation \eqref{CubicNLS}, as in particular the solutions of Theorem \ref{wbl} for $s>\frac 32$. Let $\mathcal B$ be an orthonormal basis of $\mathbb R^3$, $x_0\in [0,1]$ and $P \in \R^3$. 
Let assume that we have a smooth solution $u$ of \eqref{CubicNLS} on an open time interval~$I$.  
Starting from $u$, the first step will be to construct a frame  
\begin{equation}
\mathcal{T}(t,x) = \left( 
\begin{array}{l}
T \\ 
h \\ 
k
\end{array} 
\right)(t,x).
\end{equation}
such that the first vector $T$ is a solution of the Schr\"odinger map.  
In order to construct $\mathcal{T}$ we solve the ODE
\begin{equation}\label{ODEx0}
\partial_t \mathcal{T}(t,x_0) = \Omega(t,x_0) \mathcal{T}(t,x_0), \qquad t \in I,
\end{equation}
with initial condition $\mathcal{T}(t_0,x_0)=\mathcal B$ at time $t_0 $.
Here 
\begin{equation}\label{OmegaODE}
\Omega(t,x):= \left( 
\begin{array}{lll}
0  & - {\rm Im\,} u_x & {\rm Re\,} u_x
  \\
   {\rm Im\,}  u_x  & 0  & - \frac{|u|^2}{2} 
  \\
 - {\rm Re\,} u_x  &     \frac{|u|^2}{2}  & 0
\end{array}
\right)(t,x) .
\end{equation}
The outcome is a curve with values in $SO(3)$
$$
t \in I \to \mathcal{T}(t,x_0) \in SO(3).
$$
Then, for all $t \in I$, we solve the family of ODEs:
\begin{equation}\label{OdeInX}
  \partial_x \mathcal{T}(t,x) = \Gamma(t,x) \mathcal{T}(t,x) ,  \qquad x \in \R,
\end{equation}
with initial conditions prescribed at the point $(t,x_0)$ by the value of 
$\mathcal{T}(t,x_0)$. 
Here 
\begin{equation}\label{GammaODE}
\Gamma(t,x):=
\left( 
\begin{array}{lll}
0  & {\rm Re\,} u & {\rm Im\,}  u 
  \\
  - {\rm Re\,}  u  & 0  & 0 
  \\
 - {\rm Im\,}  u  &     0 & 0
\end{array}
\right) .
\end{equation}
This defines a map 
$$
(t,x) \in I\times \R   \to  \mathcal{T}( t,x) \in SO(3).
$$

Using the fact that the function 
 $$(t,x) \in I \times \R  \to u(t,x) \in \mathbb{C}$$ 
 is a solution of the cubic NLS equation \eqref{CubicNLS} one can prove by using $\partial_{tx}\mathcal T=\partial_{xt}\mathcal T$, that, at least for regular $u$, 
 the ODE \eqref{ODEx0} is actually valid at any $x \in \R$, namely
  \begin{equation}\label{ODEx}
\partial_t \mathcal{T}(t,x) = \Omega(t,x) \mathcal{T}(t,x), \qquad t \in I.
\end{equation}
In other words, the two linear families of ODEs \eqref{OdeInX} and \eqref{ODEx} are compatible.
\\

It is then easy to see that  
  $T$ solves the Schr\"odinger map equation 
  \begin{equation}\label{eq:SM}
  \partial_t T = T \wedge  T_{xx}.
  \end{equation}
Indeed one can compute 
$\partial_t T = - {\rm Im\,} \partial_x u\, h + {\rm Re\,} \partial_x u\, k $ from $\eqref{ODEx}$
and 
\begin{align*}
T \wedge  T_{xx} 
& = T \wedge\left(  {\rm Re\,} u\, h + {\rm Im\,}  u\, k \right) _x
\\ 
&= T \wedge  \left(  {\rm Re\,} u_x\, h + {\rm Im\,}  u_x\, k +  {\rm Re\,} u\, h_x + {\rm Im\,}  u\, k_x \right)
\\ 
&=     {\rm Re\,} u_x\, k - {\rm Im\,}  u_x\, h,  
\end{align*}
where one used that  $h_x$ and  $k_x$ are parallel to $T$, 
 from \eqref{OdeInX}, and $T \wedge h = k$, $T \wedge k = -h$.
\\

Moreover, one can easily check
 that $\chi$ defined 
as
\begin{equation}\label{ChiDefinitionVeryFirst}
\chi(t,x): = P + \int_{t_0}^{t} (T \wedge T_{x})(s,x_0) ds + \int_{x_0}^x  T(t, y) dy, 
\end{equation}
is a solution of the binormal flow equation \eqref{BinormalVeryFirst}, 
and that $T$ is actually the tangent vector of the arc-length parametrized  curve $\chi$.  
Indeed, simply take the $\partial_t$ derivative of \eqref{ChiDefinitionVeryFirst} and use equation \eqref{eq:SM}. Last but not the least let us note that $u(t)$ is a filament function of $\chi(t)$ and that due to \eqref{OdeInX} we have 
\begin{equation}\label{chit}
\partial_t\chi=T\wedge T_x={\rm Im\,}(\overline{u}(h+ik)).
\end{equation}

We note that if we choose in the definition \eqref{ChiDefinitionVeryFirst} $y_0$ instead of $x_0$, the resulting curve $\chi_{y_0}$ is a space translation of $\chi$:
\begin{equation}\label{translationchi}
\chi_{y_0}(t,x)=\chi(t,x)+\chi_{y_0}(t_0,y_0)-\chi_{y_0}(t_0,y_0).
\end{equation}

\subsection{Construction of BF solutions at low regularity.} 
%
%
%
%
%
%
First we start with a result yielding from a certain type of NLS solutions a frame that solves in a weak way the equations \eqref{OdeInX}-\eqref{ODEx} satisfied by the frames of smooth BF solutions.  
\begin{theorem}\label{ThmWeak1}
Let $I$ an open interval. Let $u\in L^1_{loc}(I\times \mathbb R)$ that can be approximated by smooth solutions $u^n$ of \eqref{CubicNLS}
 in the following sense:
\begin{equation}\label{AssWeak1}
\phi u^{n} \to \phi u \quad \mbox{in} \quad L^{\infty} (I ; L^2(\R)), \qquad \forall \phi \in C^{\infty}_c (  I\times \R  ).
\end{equation}
Then for any $t_0,x_0\in I\times[0,1]$ and any $\mathcal B$ orthonormal frame of $\mathbb R^3$ we can construct a sequence $\mathcal{T}^n := \left( \begin{array}{l} T^n \\ h^n \\ k^n \end{array} \right)$ of smooth frames that are solutions to the 
equations \eqref{OdeInX}-\eqref{ODEx} and converges in
$C (J ; H^{1}_{loc}(\R))$ for all intervals $J \subset \subset I$\footnote{We say that $J \subset \subset I$ if $J \subset K \subset I$ with $K$ compact. }, with $\mathcal T^n(t_0,x_0)=\mathcal B$. In particular, the sequence of frames 
$\mathcal{T}^n$ converges to a continuous frame $\mathcal{T} := \left( \begin{array}{l} T \\ h \\ k \end{array} \right)$ uniformly over compact subsets of $I \times \R$
and we have $|T^n| = |h^n| = |k^n| = 1$ as well as $|T| = |h| = |k| = 1$, and $\mathcal T(t_0,x_0)=\mathcal B$.
 
 We also have 
\begin{enumerate}[label=(\roman*)]
\item
For all $t \in I$   
the equation \eqref{OdeInX}, namely   
\begin{equation}\label{fdjksldkjgdjsn}
   \partial_x \mathcal{T}(t,x) = \Gamma(t,x) \mathcal{T}(t,x), 
 \end{equation}  
 is satisfied for almost every $x \in \R$.
 \item
$\mathcal{T}$ is a 
weak solution of 
the equation \eqref{ODEx}, namely  
\begin{equation}\label{OmegaODEWeak}
 \left\{
\begin{array}{lll}
 \int \Phi^T_t \cdot T \, dtdx & = &  \int  \Phi^T_x  \cdot \left(  - {\rm Im\,}  u\, h +{\rm Re\,}  u\, k     \right) dt dx
  \\ && \\
\int \Phi^h_t  \cdot h \, dtdx & = & 
 \int \left(  \Phi^h_x \cdot {\rm Im\,}  u\, T  + \Phi^h \cdot  {\rm Im\,}  u\, T_x  
+ \Phi^h \cdot \frac{|u|^2}2\, k  \right) dtdx
  \\ && \\
 \int \Phi^k_t  \cdot k \, dtdx &=&  \int \left( - \Phi^k_x \cdot  ({\rm Re\,}  u) T  
- \Phi^k  \cdot {\rm Re\,}  u\, T_x  - \Phi^k \cdot \frac{|u|^2}2 \, h \right)  dtdx
\end{array}
\right. ,
\end{equation}
for all triple of vectors $\Phi^j$, $j=T,h,k$, with components in  $C^{\infty}_c (  \R \times I )$.
\end{enumerate}

The construction of $\mathcal {T}$ is independent of the approximating sequence $u^n$, in the sense that for any other 
sequence $\tilde u^n$ verifying \eqref{AssWeak1}  we have that the relative 
family of solutions $\mathcal{\tilde T}^n$
converges also to $\mathcal{T}$ 
in
$C (J ; H^1_{loc}(\R))$ for all $J \subset \subset I$.
\end{theorem} 

%
%
 \begin{proof}
 Using the procedure described at the beginning of the section, 
 we can construct a frame 
 \begin{equation}
\mathcal{T}^n(t,x) = \left( 
\begin{array}{l}
T^n \\ 
h^n\\ 
k^n
\end{array} 
\right)(t,x)
\end{equation}
 associated with $u^n$ which satisfies equations \eqref{OdeInX}-\eqref{ODEx}, with $\mathcal T^n(t_0,x_0)=\mathcal B$. It will be useful to rewrite 
 \eqref{OdeInX} explicitly 
 \begin{equation}\label{OdeInXWeakN}
\left\{ 
\begin{array}{lll}
 \partial_x T^n  & = &    {\rm Re\,} u^n\, h^n +  {\rm Im\,}  u^n \, k^n,    
\\
 \partial_x h^n & = & -  {\rm Re\,} u^n\, T^n,
\\
\partial_x k^n  & = & -  {\rm Im\,} u^n\, T^n,
\end{array}
\right.
\end{equation}
and to note that from \eqref{ODEx} and integration by parts (using a cancelation coming from the last two equations of \eqref{OdeInXWeakN}) we obtain 
\begin{equation}\label{OmegaODEWeakN}
 \left\{
\begin{array}{lll}
 \int \Phi^T_t \cdot  T^n \, dtdx & = &  \int  \Phi^T_x \cdot \left( - {\rm Im\,}  u^n\, h^n  + {\rm Re\,}  u^n\, k^n       \right) dtdx,
  \\ && \\
\int \Phi^h_t \cdot h^n \, dtdx & = & 
 \int \left(  \Phi^h_x \cdot  {\rm Im\,}  u^n\, T^n  + \Phi^h \cdot   {\rm Im\,}  u^n\, T_x^n  
+ \Phi^h \cdot \frac{ |u^n|^2}2 k^n  \right) dtdx,
  \\ && \\
 \int \Phi^k_t \cdot  k^n \, dtdx &=&  \int \left( - \Phi^k_x \cdot  {\rm Re\,}  u^n\, T^n  
- \Phi^k \cdot  {\rm Re\,}  u^n\, T_x^n  - \Phi^k \cdot \frac{|u^n|^2}2\, h^n \right)  dtdx,
\end{array}
\right. 
\end{equation}
for all triple of vectors $\Phi^j$, $j=T,h,k$, with components in $C^{\infty}_c (  \R \times I )$.

Moreover, since $\mathcal{T}^n$ is a frame, we have in particular
\begin{equation}\label{AllBounded}   
| T^n (t,x) | = | h^n (t,x) | =| k^n (t,x) |  = 1 
\end{equation}
and using this uniform bound and the one given by assumption \eqref{AssWeak1}, we deduce from 
equation \eqref{OdeInXWeakN} also the bound 
\begin{equation}\label{DerivativeBounded}   
\|  \Phi \cdot \mathcal{T}_x^n \|_{L^{\infty}(I; L^2(\R) )} \leq C(\Phi),
\end{equation}
for all matrices $\Phi$ with entries in $C^\infty_c(I \times \R)$.

We will first show that the sequence $\mathcal{T}^n$ is Cauchy in 
$C (J ; L^2_{loc}(\R))$ for all time intervals~$J \subset \subset I$. To do so we define the defect
 $$\mathcal T^{n,m} = \left( 
\begin{array}{l}
\mathcal T^{n,m}_T \\ \\ 
\mathcal T^{n,m}_h \\   \\
\mathcal T^{n,m}_k
\end{array} 
\right) = \left(
 \begin{array}{l}
T^{n} - T^{m} \\ 
h^{n} - h^{m} \\ 
k^{n} - k^{m}
\end{array} 
\right) =
  \mathcal{T}^{n} - \mathcal{T}^{m}$$ and we will study its variation in space-time. We will use several times in the following the notation $f^{n,m}:=f^n-f^m$. 
 
 Using the linearity of \eqref{OdeInXWeakN} 
 we get:  
 \begin{equation}\label{eq:CauchyInX}
\partial_x \mathcal T^{n,m}   = \Gamma^{m} \mathcal T^{n,m}  +  \Gamma^{n,m}  \mathcal{T}^{n},  
 \end{equation}
 where
 \begin{equation}\label{GammaODEJ}
\Gamma^n:=
\left( 
\begin{array}{lll}
0  & {\rm Re\,} u^n & {\rm Im\,}  u^n 
  \\
  - {\rm Re\,}  u^n  & 0  & 0 
  \\
 - {\rm Im\,}  u^n  &     0 & 0
\end{array}
\right).
\end{equation}

Thus, denoting $  | \mathcal T^{n,m} |^2:=|T^{n,m}|^2+|h^{n,m}|^2+|k^{n,m}|^2$, we arrive to
 $$
\frac12     \partial_x (  | \mathcal T^{n,m} |^2)    
 =   Tr\Big(  \Gamma^{m} \mathcal T^{n,m} \,^t\mathcal T^{n,m}  +   \Gamma^{n,m}  \mathcal{T}^{n} \,^t\mathcal T^{n,m} \Big).
 $$   
 Since $\Gamma^{m}$ is antisymmetric
 we have $$ Tr\Big(  \Gamma^{m} \mathcal T^{n,m} \,^t\mathcal T^{n,m} \Big) =0,$$
 thus 
evaluating at $t_0$ and integrating over $[x_0, x]$, recalling \eqref{AllBounded} and $\mathcal T^n(t_0,x_0)=\mathcal T^m(t_0,x_0)=\mathcal B$, one has 
 \begin{equation}\label{R1DomC}
 | \mathcal T^{n,m} (t_0,x)|^2 \lesssim   \| u^{n,m}(t_0,\cdot) \|_{L^1([x_0, x])}.
 \end{equation}

Now we calculate the time variation of $|\mathcal T^{n,m}|^2$. Using the linearity of \eqref{ODEx} 
 we obtain 
 \begin{equation}\label{OmegaODEWeakDeltaMollified2222}
 \left\{
\begin{array}{lll}
 \int \Phi^T \cdot \partial_t \mathcal T^{n,m}_T \, dx & = & r^T[\Phi^T]  
 +
 \int   \Phi^T \cdot \left(
- {\rm Im\,}  u^{m}_x\, \mathcal T^{n,m}_h +   {\rm Re\,}  u^{m}_x\, \mathcal T^{n,m}_k    
\right) dx,
  \\ && \\
\int \Phi^h \cdot  \partial_t \mathcal T^{n,m}_h \, dx & = &   r^h[\Phi^h] 
+
 \int  \Phi^h \cdot   \Big({\rm Im\,}  u^{m}_x\, \mathcal T^{n,m}_T 
-  \frac{|u^{m}|^2}2 \, \mathcal T^{n,m}_k  \Big) dx,
  \\ && \\
 \int \Phi^k \cdot \partial_t \mathcal T^{n,m}_k \, dx &=&  r^k[\Phi^k] +
  \int   - \Phi^k \cdot   \Big({\rm Re\,}  u^{m}_x\, \mathcal T^{n,m}_T  
 +  \frac{|u^{m}|^2}2 \, \mathcal T^{n,m}_h \Big)  dx,
\end{array}
\right. 
\end{equation}
 for all triple of vectors $\Phi^j$, $j=T,h,k$ and
 $$
 r^T[\Phi^T] :=  \int   \Phi^T \cdot \Big(
- {\rm Im\,} u^{n,m}_x\, h^n  +   {\rm Re\,} u^{n,m}_x  \, k^n      \Big) dx,
 $$
 $$
 r^h[\Phi^h] :=
 \int  \Phi^h \cdot  \Big({\rm Im\,} u^{n,m}_x\, T^n 
-  \frac{(|u|^2)^{n,m}}2\,  k^n  \Big) dx,
 $$
 $$
  r^k[\Phi^k] :=
 \int    \Phi^k \cdot  \Big(- {\rm Re\,} u^{n,m}_x\, T^n 
+  \frac{(|u|^2)^{n,m}}2\, h^n  \Big) dx.
 $$
Integrating by parts we can rewrite the remainders as
 \begin{align}\label{Remainder1}
 r^T[\Phi^T] & := - \int   \Phi^T_x \cdot \Big(
- {\rm Im\,}  u^{n,m}  \, h^n +  {\rm Re\,}  u^{n,m} \, k^n\Big)  - \Phi^T \cdot \Big(
- {\rm Im\,}  u^{n,m} \, \partial_x h^n +   {\rm Re\,}u^{n,m} \,\partial_x k^n   \Big) dx,
 \end{align}
 \begin{align}\label{Remainder2}
 r^h[\Phi^h] 
 & := -
 \int  \Phi^h_x \cdot   {\rm Im\,}  u^{n,m} \, T^n 
+ \Phi^h \cdot \frac{(|u|^2)^{n,m}}2\,k^n   - \Phi^h \cdot  {\rm Im\,} u^{n,m} \, \partial_x T^n dx,
 \end{align}
 \begin{align}\label{Remainder3}
 r^k[\Phi^k] 
 & := 
 \int \Phi^k_x \cdot  {\rm Re\,}   u^{n,m} \,T^n 
+ \Phi^k  \cdot \frac{(|u|^2)^{n,m}}2\, h^n +\Phi^k \cdot   {\rm Re\,}   u^{n,m} \, \partial_x T^n  dx.
\end{align}
 
Testing with $\Phi^j(x) = \eta(x) \mathcal T^{n,m}_j(t,x)$, $j \in\{T,h,k\}$, with 
$\eta \in C^{\infty}_c(\R)$ a bump function of an interval containing $x_0$ 
and summing over $j\in\{T,h,k\}$ we arrive to
\begin{equation}\label{BeforePTLL2Uniq}
\frac12  \partial_t  \int \eta|\mathcal T^{n,m}|^2 dx 
=  \sum_{j=T,h,k}  r^j[\eta \mathcal T^{n,m}_j]   +
\int  Tr\Big(\Omega^{m} \mathcal T^{n,m} \,^t\mathcal T^{n,m} \Big)  \eta\, dx,
\end{equation}
where 
\begin{equation}\label{OmegaODEEpsilon}
\Omega^m:= \left( 
\begin{array}{lll}
0  & - {\rm Im\,}  u^m_x & {\rm Re\,} u^m_x 
  \\
   {\rm Im\,} u^m_x  & 0  & - \frac{|u^m|^2}{2} 
  \\
 - {\rm Re\,}  u^m_x  &     \frac{|u^m|^2}{2}  & 0
\end{array}
\right).
\end{equation}
Since $\Omega^{m}$ is antisymmetric we have 
$$ Tr\Big(\Omega^{m} \mathcal T^{n,m} \,^t\mathcal T^{n,m} \Big)=0,$$
thus
we arrive to 
$$
\frac12  \partial_t  \int \eta|\mathcal T^{n,m}|^2 dx 
=  \sum_{j=T,h,k}  r^j[\eta \mathcal T^{n,m}_j], 
$$
and integrating over $[t_0,t]$ we obtain
$$
 \int \eta(x)|\mathcal T^{n,m}(t,x)|^2 dx
= \int \eta(x)|\mathcal T^{n,m}(t_0,x)|^2 dx   + 2 \sum_{j\in\{T,h,k\}} \int_{t_0}^t r^j[\eta\mathcal T^{n,m}_j] dt.
$$
Plugging \eqref{R1DomC} into this estimate we get
\begin{equation}\label{WeCanPassHere1}
 \int \eta(x)|\mathcal T^{n,m}(t,x)|^2 dx 
\leq C(\eta) \| u^{n,m}(t_0,\cdot)  \|_{L^1(\text{supp}(\eta))} + 2 \sum_{j\in\{T,h,k\}} \int_{t_0}^t r^j[\eta \mathcal T^{n,m}_j] dt.
\end{equation}
By \eqref{AssWeak1}-\eqref{AllBounded}-\eqref{DerivativeBounded} and recalling the form of the remainders 
\eqref{Remainder1}-\eqref{Remainder2}-\eqref{Remainder3} it is clear that
$$
\lim_{n,m \to \infty}  \sup_{t \in J}  \int_{t_0}^t | r^j[\eta \mathcal T^{n,m}_j]| dt  =0, \qquad \forall J \subset \subset I.
$$
Thus, using again \eqref{AssWeak1}, we can pass to the limit into \eqref{WeCanPassHere1} 
$$
\lim_{n,m \to \infty} \sup_{t \in J} \int \eta(x)|\mathcal T^{n,m}(t,x)|^2 dx dt =0,  \qquad \forall J \subset \subset I .
$$

Since $\eta$ was a generic bump function, we have obtained that 
$\mathcal{T}^n$ is a Cauchy sequence wit respect to the norm $L^{\infty}(J; L^2_{loc}(\R))$ for any
interval $J \subset \subset I$.  
We denote with $\mathcal{T}$ its limit.

Now, from \eqref{eq:CauchyInX} that states:
$$\partial_x\mathcal T^{n,m}   = \Gamma^{m} \mathcal T^{n,m}  +  \Gamma^{n,m}  \mathcal{T}^{n},  
$$
 using the fact that
$\mathcal{T}^n$ is Cauchy in $L^{\infty}(J; L^2_{loc}(\R))$, the assumption \eqref{AssWeak1}
and recalling the form \eqref{GammaODEJ} of $\Gamma^n$ 
we deduce that $\mathcal{T}^n$ is Cauchy also in $L^{\infty}(J; W^{1,1}_{loc}(\R))$, 
for all $J \subset \subset I$.
In particular, this implies (by the fundamental theorem of calculus) that 
$\mathcal{T}^n$ converges to $\mathcal{T}$ pointwise uniformly over any space-time compact set 
$K \subset I \times \R$  and that $\mathcal{T}$ is continuous on $I \times \R$. 
 Using this information and again the assumption \eqref{AssWeak1} we deduce 
from \eqref{eq:CauchyInX} that~$\mathcal{T}^n$ is Cauchy also in~$L^{\infty}(J; H^{1}_{loc}(\R))$,
for all $J \subset \subset I$.
  
Moreover, it is now 
clear that we can pass to the limit into (the distributional version of) \eqref{OdeInXWeakN}
 and into \eqref{OmegaODEWeakN} proving also the parts $(i)$ and $(ii)$ of the statement. 
 
 The uniqueness part of the statement is proved by considering the defect $  \mathcal{T}^n - \mathcal{\tilde T}^n$ and proceeding in the same way, with straightforward modifications.
 \end{proof}
 In Theorem \ref{thbf2}, we actually work with solutions of the cubic NLS \eqref{CubicNLS} that almost surely satisfy something 
stronger than \eqref{AssWeak1}. In fact our solutions satisfy 
$\rho_s$-almost surely the assumption \eqref{AssWeak1Reinf} of the following Proposition, 
for all $\alpha < s$. This fact will be proved  in Section \ref{Sec:thbf2}.
 \begin{prop}\label{CorWeak1}
Let $\alpha \in [0,1)$. 
If the assumption \eqref{AssWeak1} is reinforced to 
\begin{equation}\label{AssWeak1Reinf}
u^{n} \to  u \quad \mbox{in} \quad L^{\infty} (J ; C^{\alpha}(\R)), \qquad \forall J \subset \subset I,
\end{equation}
then we have that $\mathcal{T}^{n} \to \mathcal{T}$ in 
$L^{\infty} (J ; C^{1+ \alpha}(\R))$ for any interval $J \subset \subset I$. 
For any other 
sequence $\tilde u^n$ verifying \eqref{AssWeak1Reinf}  we have that the relative 
family of solutions~$\mathcal{\tilde T}^n$
converges to $\mathcal{T}$ 
in
$L^{\infty} (J ; C^{1+ \alpha}(\R))$ for any interval $J \subset \subset I$ as well. 
\end{prop}
\begin{proof}
Under the stronger assumption \eqref{AssWeak1Reinf}, we first deduce from~\eqref{eq:CauchyInX} the convergence 
 of $\mathcal{T}^n$ to~$\mathcal{T}$
 in $L^{\infty} (J ; C^{1}(\R))$. Then, using this extra information and \eqref{AssWeak1Reinf} again, 
 we deduce from 
 \eqref{eq:CauchyInX} the convergence 
 of $\mathcal{T}^n$ to~$\mathcal{T}$ in 
 in $L^{\infty} (J ; C^{1 + \alpha}(\R))$, as claimed
 
 The uniqueness part of the statement is proved considering 
 the defect $  \mathcal{T}^n - \mathcal{\tilde T}^n$ and 
 reasoning similarly. 
\end{proof}
 We are now ready to prove the main result of this section. 
 \begin{theorem}\label{MainThhBinormal}
Under the assumptions of Theorem \ref{ThmWeak1} 
the tangent vector 
$T$ is a weak solution of the Schr\"odinger map, namely 
\begin{equation}\label{WeakSM}
\int T \cdot \Phi_t = \int  ( T \wedge  T_x) \cdot \Phi_x , \qquad\forall \Phi\in(C^{\infty}_c (I\times \R  ))^3.
\end{equation}
Moreover, there exists a unique (modulo translation) weak solution $\chi \in C(I; H^2_{loc}(\R))$ of the binormal flow equation,
namely
\begin{equation}\label{eq:WeakBinInTHM}
- \int \chi \cdot \Phi_t   = \int ( \chi_x \wedge \chi_{xx} ) \cdot \Phi, \qquad  
\chi_x = T,\qquad\forall \Phi\in(C^{\infty}_c (I\times \R  ))^3.
\end{equation}
Also, $u(t)$ is a filament function of $\chi(t)$. 

Under the (stronger) assumption \eqref{AssWeak1Reinf} of Proposition \ref{CorWeak1} we have that 
$\chi \in C^{1}(J; C^{\alpha}(\R)) \cap C(J; C^{2+\alpha}(\R))$ for 
all intervals $J \subset \subset I$ and solves classically
 \begin{equation}\label{eq:StrongBinInTHM}
\chi_t   =  \chi_x \wedge \chi_{xx} , \qquad  
\chi_x = T,\qquad\forall (t,x) \in I\times \R .
\end{equation}

\end{theorem}

\begin{proof}

Proceeding as above  (see \eqref{eq:SM}) we know that
$$
\int  T^n \cdot \Phi_t  = \int (T^n \wedge  T^n_{x}) \cdot \Phi_x,
$$
for all of vectors $\Phi$ with components in  $C^{\infty}_c (I\times\R  )$.
Theorem \eqref{ThmWeak1} ensures that 
\begin{equation}\label{RecH1}
T^n \to T \quad 
\mbox{ 
in
$C (J ; H^{1}_{loc}(\R))$ for all intervals $J \subset \subset I$} 
\end{equation} 
so we can pass to the limit in the previous equation and obtain \eqref{WeakSM}.

To construct weak solutions of the binormal flow equation we will need some extra-information on $T^n$.
From \eqref{RecH1} it follows that 
$$
\int_{I_1}  \left( \int_J \left| (T \wedge T_{x})(\tau,x) - ( T^n\wedge T^n_{x})(\tau,x) \right| d\tau \right)  d x\overset{n\rightarrow\infty}{\longrightarrow}0,
$$
where $I_1$ is any (space) bounded interval of $\R$. 
Thus the sequence of functions
$$
x \in I_1 \to   
\int_J \left| (T \wedge T_{x})(\tau,x) - ( T^n\wedge T^n_{x})(\tau,x) \right| d\tau
$$
converge to zero in $L^{1}(I_1)$.  
This implies, by a diagonalization process involving the intervals $I_1+m$ for $m\in\mathbb Z$, that we can extract a subsequence that converges pointwise to zero for almost every $y_0 \in \R$:
\begin{equation}\label{cvx0}
\int_J \left| (T \wedge T_{x})(\tau,y_0) - ( T^{n_k}\wedge T^{n_k}_{x})(\tau,y_0) \right| d\tau \rightarrow 0.
\end{equation}
We choose such $y_0\in[0,1]$.

Now we define for $P \in \R^3$ the curves
\begin{equation}\label{Def:CurveN}
\chi^n_{y_0}( t,x) := P + \int_{t_0}^{t} (T^n \wedge T^n_{x})(\tau,y_0) d\tau + \int_{y_0}^x  T^n(t,y) dy,
\end{equation}
that are smooth binormal flow solutions with filament function $u_n$. 
Using \eqref{cvx0} and \eqref{RecH1} it follows that $\chi^{n_k}$ converges pointwise to 
\begin{equation}\label{Def:Curve}
\chi_{y_0}(t,x) := P + \int_{t_0}^{t} (T \wedge T_{x})(\tau,y_0) d\tau + \int_{y_0}^x  T(t,y) dy,
\end{equation}
and also in $L^{\infty}(J;H^2_{loc}(\R))$ for all intervals $J \subset \subset I$. Eventually we shall see that choosing another $y_0$ induces only a translation in space, so for the moment we shall drop the subindex $y_0$ to lighten the presentation.

Now, taking the $\partial_x$ derivative of \eqref{Def:CurveN} we have 
\begin{equation}\label{CurvePass1}
\partial_x\chi^n = T^n,
\end{equation} 
so \eqref{RecH1} implies that $\chi_x=T$. From \eqref{fdjksldkjgdjsn} it follows that $u(t)$ is the filament function of $\chi(t)$. 

To obtain \eqref{eq:WeakBinInTHM} we use the facts that $\chi^{n_k}$ converges to $\chi$, that $\chi^{n_k}$ is a smooth solution of the binormal flow and the convergence of the tangent vectors \eqref{RecH1}:
$$- \int \chi \cdot \Phi_t  =-\lim_{k\rightarrow\infty}\int \chi^{n_k} \cdot \Phi_t =\lim_{k\rightarrow\infty} \int \chi^{n_k}_t \cdot \Phi$$
$$=\lim_{k\rightarrow\infty} \int \chi^{n_k}_t \cdot \Phi =\lim_{k\rightarrow\infty} \int  (T^{n_k}\wedge T^{n_k}_{x}) \cdot \Phi= \int  (T\wedge T_{x}) \cdot \Phi= \int ( \chi_x \wedge \chi_{xx} ) \cdot \Phi.$$

Under the stronger assumption \eqref{AssWeak1Reinf} we get from Proposition \ref{CorWeak1} that $T^{n} \to   T$ in 
$L^{\infty} (J ; C^{1+ \alpha}(\R))$ for any interval $J \subset \subset I$.
Therefore the weak derivative in time of $\chi$ is at least a continuous function (again even in the case $\alpha =0$)
and
\eqref{eq:StrongBinInTHM} actually holds pointwise. 
Thus we have obtained
$$\chi_t = T \wedge T_{x} \in C(J; C^{\alpha}(\R))$$ 
and $$\chi_x = T \in C(J; C^{1+\alpha}(\R))$$ 
for all compact 
intervals~$J \subset \subset I$. 

Now we get back to the uniqueness modulo translation. If we choose $z_0\neq y_0$ such that \eqref{cvx0} is satisfied then we see that $\chi_{y_0}$ is a translation in space of $\chi_{z_0}$: 
$$\chi_{y_0}(t,x)=\chi_{z_0}(t,x)+\chi_{y_0}(t_0,z_0)-\chi_{y_0}(t_0,y_0).$$
Indeed, in view of \eqref{cvx0}, \eqref{RecH1} and the fact that $T^{n_k}$ is a smooth solution of the Schr\"odinger map, we have
$$\chi_{y_0}(t,x)-\chi_{z_0}(t,x)-\chi_{y_0}(t_0,z_0)+\chi_{y_0}(t_0,y_0)$$
$$=\int_{t_0}^{t} \Big((T \wedge T_{x})(\tau,y_0)-(T \wedge T_{x})(\tau,z_0)\Big) d\tau + \int_{y_0}^{z_0}  T(t,y)-T(t_0,y) dy$$
$$=\lim_{k\rightarrow\infty} \int_{t_0}^{t} \Big((T^{n_k} \wedge T^{n_k}_{x})(\tau,y_0)-(T^{n_k} \wedge T^{n_k}_{x})(\tau,z_0)\Big) d\tau + \int_{y_0}^{z_0}  T^{n_k}(t,y)-T^{n_k}(t_0,y) dy$$
$$=\lim_{k\rightarrow\infty} -\int_{t_0}^{t} \int_{y_0}^{z_0}(T^{n_k} \wedge T^{n_k}_{xx})(\tau,y)dyd\tau + \int_{y_0}^{z_0}  T^{n_k}(t,y)-T^{n_k}(t_0,y) dy$$
$$=\lim_{k\rightarrow\infty} -\int_{t_0}^{t} \int_{y_0}^{z_0}\partial_\tau T^{n_k}(\tau,y)dyd\tau + \int_{y_0}^{z_0}  T^{n_k}(t,y)-T^{n_k}(t_0,y) dy=0.$$
Thus we obtain the same curve modulo translation by choosing any $z_0$ such that \eqref{cvx0} is satisfied, which is valid almost everywhere in $\mathbb R$. Noting that the definition of $\chi_{y_0}$ depends only on $y_0$ and on $T$ which is the limit of $T^n$, this implies also that $\chi_{y_0}$ is unique modulo translation as limit of the whole sequence $\chi_{y_0}^n$. 
This completes the proof.
\end{proof}
\section{Proof of Theorem \ref{thbf}}\label{sect-thdet}
Since $\chi_1\in\Omega_1^s$, its filament function $u_1$ is such that $e^{-i\frac{x^2}{4}}u_1(x)$ is $4\pi$-periodic and belongs to $H^{s}([0,4\pi])$. Let $\{A_j(1)\}$ be the sequence defined by
$u_1=\sum_{j} A_j(1)e^{i\frac{(x-j)^2}{4}}$ and let $u$ be the solution of \eqref{CubicNLS} with $u(1)=u_1$ given by Theorem \ref{wbl}.
\\

Let $\sigma>\frac 32$. We consider an approximating  sequence of curves $\chi^n_1\in\Omega_1^{\sigma}$ for $\chi_1$, that is 
$$d_{\Omega_1^s}( \chi_1^n,\chi_1)=\| \chi_1^n-\chi_1\|_{L^\infty([0,1])}+\| \partial_x\chi_1^n-\partial_x\chi_1\|_{L^\infty([0,1])}+\|e^{-i\frac{x^2}{4}}( u_1^n-u_1)\|_{H^{s}_{per}([0,4\pi ])}\overset{n\rightarrow\infty}{\rightarrow} 0,$$ 
where $u_1^n$ is a filament function of $\chi_1^n$. Since $ \chi_1^n\in\Omega_1^\sigma$ we have that $e^{-i\frac{x^2}{4}} u_1^n(x)$ is $4\pi$-periodic and belongs to $H^{\frac 32^+}([0,4\pi])$. As $e^{-i\frac{x^2}{4}}u^n_1$ converges to $e^{-i\frac{x^2}{4}}u_1$ in $H^s([0,4\pi])$, it follows that  the sequence $\{A_j^n\}\in l^{2,\frac 32^+}$ defined by $u^n_1=\sum_{j} A_j^ne^{i\frac{(x-j)^2}{4}}$ satisfies
\begin{equation}\label{cvcoefs}
\{A_j^n\}\overset{n\rightarrow\infty}{\underset{l^{2,s}_j}{\longrightarrow}}\{A_j(1)\}.
\end{equation}
The existence of such sequences is ensured by the example $A_j^n:=A_j(1)\mathds{1}_{j\leq n}$. 
\\

Let $u^n$ be the solution of \eqref{CubicNLS} with $u^n(1)=u^n_1$ given by Theorem \ref{wbl}, and let 
$$\tilde\chi^n(t):=\Psi_{1,t}(\chi_1^n),$$ 
i.e. the corresponding strong solutions of the binormal flow with $\tilde\chi^n(1)=\chi^n_1$, whose construction is detailed in \S\ref{sectHas}, that have $ u^n(t)$ as filament function at time $t$ and satisfies $ \tilde\chi^n(t_0,x_0)=P, \tilde T^n(t_0,x_0)=\mathcal B$. We denote $v,v^n$ the pseudo-conformal transformations of $u,u^n$ respectively, see formula \eqref{Eq:Pseudoconf}. Using \eqref{contv} and \eqref{cvcoefs} we obtain for any $J\subset\subset (0,1]$
\begin{equation}\label{estdist}
\sup_{t\in J}\| e^{-i\frac{x^2}{4t}}(u^n(t)-u(t))\|_{H^{s}_{per}([0,4\pi t])}=\sup_{t\in 1/J}\|v^n(t)-v(t)\|_{H^{s}}
\end{equation}
$$\leq C(J, \|v(1) \|_{H^{s}}) \|  v^n_1-v_1  \|_{H^{s}}=C(J, \|A_j(1)  \|_{l^{2,s}_j}) \| A^n_j(1)-A_j(1)  \|_{l^{2,s}_j}\overset{n\rightarrow \infty}{\longrightarrow}0,
$$
In particular, this yields 
$$
\int  \phi | u^{n} - u|^2  dxdt
\overset{n\rightarrow \infty}{\longrightarrow}0,
$$
for all $\phi \in C^{\infty}_c( (0,1] \times \R)$. In view of Theorems \ref{ThmWeak1}-\ref{MainThhBinormal} it follows that there exists $\chi$ a weak solution of the binormal flow \eqref{eq:WeakBinInTHM} with $\chi(t)$ having $u(t)$ as filament function. Moreover, in view of the proof of Theorem \ref{MainThhBinormal}, $\chi$ is approximated, pointwise and at the level of the tangent vector, by binormal flow solutions $\chi^n$ with $\chi^n(t)$ having filament function $u^n(t)$. Also from the proof of Theorems \ref{MainThhBinormal} we see that $\tilde\chi^n(t)$ and $\chi^n(t)$ have same tangent vector, thus \eqref{RecH1} yields
$$\|\partial_x\tilde \chi^n(t)-\partial_x\chi(t)\|_{L^\infty(0,1)}\overset{n\to\infty}{\longrightarrow}0.$$
Also, in view of the definitions \eqref{ChiDefinitionVeryFirst} and \eqref{Def:CurveN} of $\tilde\chi^n(t)$ and $\chi^n(t)$ respectively, and of the remark in \eqref{translationchi}, it follows that $\tilde \chi^n(t)$ is a translation $\mathfrak T^n$ of $\chi^n(t)$, independently on $t$. Since $\chi_1^n$ is approximating $\chi_1$ in $\Omega_1^s$, 
$$\|\mathfrak T^n  \chi^n(1)-\chi_1\|_{L^\infty([0,1])}=\| \tilde \chi^n(1)-\chi_1\|_{L^\infty([0,1])}=\| \chi_1^n-\chi_1\|_{L^\infty([0,1])}\overset{n\to\infty}{\longrightarrow}0,$$
and since $\chi^n(1)$ is pointwise converging, it follows that the translation factor of $\mathfrak T^n $ converges to zero. Therefore, using also that $\chi^n(t)$ is pointwise converging to $\chi(t)$ we obtain
$$\| \tilde \chi^n(t)-\chi(t)\|_{L^\infty([0,1])}=\|\mathfrak T^n  \chi^n(t)-\chi(t)\|_{L^\infty([0,1])}\overset{n\to\infty}{\longrightarrow}0.$$
Combining with \eqref{estdist} we have 
$$d_{\Omega_t^s}(\chi^n(t),\chi(t))\overset{n\to\infty}{\longrightarrow}0.$$ 
Moreover, the uniqueness of this extension of the Hasimoto solution map \eqref{HasimotoFlowForUs} is a consequence of fact that Theorem \ref{ThmWeak1} is independent of the approximating sequence $u^n$, see the last paragraph of the statement. 
\\

To conclude the proof of Theorem \ref{thbf} we only must show that the curve $\chi(t)$ converges, 
uniformly w.r.t. $x \in \R$, as $t \to 0$. Let $0<\delta<1$. Form \eqref{RecH1} and definitions \eqref{Def:CurveN}-\eqref{Def:Curve} of $\chi$ we obtain that $\partial_t\chi^n(t)$ converges in the sense of the distributions $\mathcal D'([\delta,1]\times\R)$ to $\partial_t\chi\in L^1_{loc}([\delta, 1],\mathcal D'(\R,\R^3))$. Thus for $\phi \in C^{\infty}_c( [\delta,1] \times \R)$ we can use Fubini to get
$$\int \int \chi(\tau,x) \phi(\tau,x)d\tau dx=-\int \int \partial_\tau\chi(\tau,x) \Big(\int_\delta^\tau \phi(t,x)dt\Big)d\tau dx$$
$$=-\int\int\Big(\int_t^1 \partial_\tau\chi(\tau,x)d\tau\Big)\phi(t,x)dtdx=-\lim _n\int\int\Big(\int_t^1 \partial_\tau\chi^n(\tau,x)d\tau\Big)\phi(t,x)dtdx$$
$$=\lim _n\int \int \partial_\tau \chi^n(\tau,x) \phi(\tau,x)d\tau dx.$$
Thus for $\delta< t_1 < t_2  \leq 1$ we have
\begin{equation}\label{limitchit}
\chi(t_2, x)  -  \chi(t_1, x) =\int_{t_1}^{t_2}\partial_\tau\chi(\tau,x)d\tau=\lim _n\int_{t_1}^{t_2}\partial_\tau\chi^n(\tau,x)d\tau
\end{equation}
in the sense of distributions. Since $\chi^n$ are smooth and $\chi\in\mathcal C([\delta,1],H^2_{loc}(\R,\R^3))\subset \mathcal C([\delta,1]\times\R)$ we have the equality pointwise. 
Using the pseudo-conformal transformation \eqref{Eq:Pseudoconf}, changing variable $s= \frac{1}{t}$
 in the integral and using \eqref{chit} together with the Sobolev embedding $W^{\frac 16^+,6}\subset L^\infty$ we get
\begin{equation}\label{limitchitbis}
\Big | \int_{t_1}^{t_2}\partial_\tau\chi^n(\tau,x)d\tau \Big | \leq \int_{t_1}^{t_2}|u^n(\tau,x)|d\tau =  \int_{t_1}^{t_2}  
  \frac{1}{\sqrt{s}} \left| v^n \left( \frac{1}{s}, \frac{x}{s} \right) \right| ds 
  \end{equation}
$$ = 
 \int_{1/t_2}^{1/t_1}  t^{-\frac{3}{2}}  | v^n \left( t, t x \right) | dt
\leq \int_{1/t_2}^{1/t_1}  t^{-\frac{3}{2}}  \|  v^n(t, \cdot)\|_{W^{\frac 16^+,6}(\T)} dt.
$$
We cover the interval $[1/t_2, 1/t_1]$ by a finite number of disjoint unit subintervals $I_k=[k,k+1]$ and we can then bound, for $s >1/6$, uniformly with respect to $n\in\mathbb N$ and $ x\in\mathbb R$
$$
\Big | \int_{t_1}^{t_2}\partial_\tau\chi^n(\tau,x)d\tau \Big | \leq  \sum_{k\gtrsim 1/t_2} \frac{1}{k^{3/2}} \|  \tilde v^n\|_{L^6( I_k; W^{\frac 16^+,6}(\T))}
$$ $$\leq C(\|  v^n(1)\|_{L^2}) \sum_{k\gtrsim 1/t_2} \frac{1}{(k+1)^{3/2}} \|   v^n(k)\|_{H^s} \leq C(\| e^{-i\frac{\cdot^2}4}u_1\|_{H^s(0,4\pi)}) \sqrt{t_2},$$
where we used the Strichartz inequality \eqref{Str2}, and then \eqref{est2s} to control
$$\|\{A_j(k)\}\|_{l^{2,s}}=\| v^n(k)\|_{H^s} \leq C(\| v^n(1)\|_{L^2}) \|    v^n(1)\|_{H^s}\leq C(\|  e^{-i\frac{\cdot^2}4}u_1\|_{H^s(0,4\pi)}).$$
Summarizing, we have obtained pointwise, due to dominated convergence in \eqref{limitchitbis},
$$|\chi(t_2, x)  -  \chi(t_1, x)|\leq C(\| e^{-i\frac{\cdot^2}4} u_1\|_{H^s(0,4\pi)}) \sqrt{t_2}.$$
Making $t_2$ go to zero we get that the curve $\chi(t)$ converges, 
uniformly w.r.t. $x \in \R$, as $t \to 0$. This completes the proof of Theorem \ref{thbf}.  
\section{Quasi-invariance}\label{Sec:QuasiInvariance}
In this section we prove the first part of Theorem \ref{MainThm3}, in particular the inequality  \eqref{QuasiInvFin}. To do so,  we also introduce the probabilistic tools needed in order to prove Theorem~\ref{thbf2}. 
\subsubsection*{The reference measure}
Let $g^{\omega}_k$ be a sequence of independent standard complex Gaussians, i.e.
\begin{equation}\label{DefComplexGaussians}
g^{\omega}_k = \frac{1}{\sqrt{2}}(h^{\omega}_k+i l^{\omega}_k), \qquad k\in\mathbb{Z},
\end{equation}
with 
$h^{\omega}_k, l^{\omega}_k$ independent   
standard real Gaussians. This means that $h^{\omega}_k, l^{\omega}_k \sim \mathcal{N}(0,1)$ and are independent. 
In particular 
$$\mathbb{E}(h_k^{\omega}) = \mathbb{E}(l_k^{\omega}) =0$$ 
and 
$$
\mathbb{E}(h_j^{\omega}l_k^{\omega}) = 0, \qquad \mathbb{E}(l_k^{\omega} l_j^{\omega}) = \mathbb{E}(h_k^{\omega} h_j^{\omega}) = \delta_{kj},
$$ 
where $\mathbb{E}(\cdot)$ denotes the expectation and $\delta_{kj}$ the Kronecker delta. 
\\

Let  $\rho_s$  be the measure 
\begin{equation}\label{Def:Rho(1)}
d \rho_s (v)= 1_{\big\{ \| v \|^2_{L^{2}(\mathbb{T})} \leq M \big\}} (v)d \gamma_s(v), \qquad M>0,
\end{equation}
where the Gaussian measure~$\gamma_s$ is the one induced on 
$H^{s'}(\T)$, $s' < s$, by the random variable
\begin{equation}\label{TMIATS}
\omega \longmapsto   \sum_{k \in \mathbb{Z}} 
\frac{g_k^{\omega}  }{(1 + |k|^{2s+1})^{\frac12} }  e^{i x k}.
\end{equation}
The measures  $\gamma_s$ is a probability measure on the Sobolev space  $H^{s'}(\T)$ for every $s' < s$. 
\\

Ou aim is  to show the quasi-invariance of $ \rho_s$ under the solution map $\Phi_{1,\tau}$, $\tau \in [1, \infty)$, associated with the equation \eqref{NLSv}, where the initial condition is taken at the time $\tau=1$. This means to show 
 \begin{equation}\label{Rho*AbsCont}
 (\Phi_{1,\tau}) _* \rho_s \ll  \rho_s, \qquad   \tau \in [1, \infty).
\end{equation}  
In fact, we will prove the (stronger) quantitative estimate \eqref{QuasiInvFin} 
that implies \eqref{Rho*AbsCont}.
\subsubsection*{Finite dimensional approximation}
In order to replace the associated (formal) infinite dimensional Lebesgue measure with a well defined object we 
will work with the truncated system\footnote{The operator $P_{\leq N}$ the projection of the first $|k| \leq N$ Fourier modes of a function on~$\T$ and
 $P_{>N}  := \text{Id} - P_{\leq N}$. }, for $N \in \mathbb{N}$:
\begin{equation}\label{NLSTruncated}
\left\{\begin{array}{c}i \partial_t P_{\leq N} v^N + \partial_{xx} P_{\leq N} v^N +
 \frac{1}{t} P_{\leq N} \left( | P_{\leq N} v^N|^2P_{\leq N} v^N \right)=0, \quad t \in [1, \infty),\\
 \,\\
 P_{> N} v^N(t) = e^{i(t-\tau) \Delta} P_{> N} v^N(\tau), \qquad t \in [1, \infty),
 \end{array}\right.
\end{equation} 
where the initial data is prescribed at some initial time $\tau\in  [1, \infty)$.  
 In terms of the evolution of the coefficients $B_k^N(t)$ defined as in \eqref{vB} as follows
     $$v^N(t,x)=\sum_j B^N_j(t) e^{itj^2}e^{ixj}, \quad
B^N_j(t)=\mathcal F(e^{-it\Delta}v^N(t))(j),
$$
 this becomes in view of \eqref{Bjsyst}:
\begin{equation}\label{Bjsyst-Truncated}
\left\{\begin{array}{c}\partial_t B_k^N(t)=F(\{B^N_j(t)\}):=-\frac{i }{ t}
\sum_{\substack{k-j_1+j_2-j_3=0 \\ |j_1|, |j_2|, |j_3| \leq N }}
e^{-it(k^2-j_1^2+j_2^2-j_3^2)}B^N_{j_1}(t)\overline{B^N_{j_2}(t)}B^N_{j_3}(t), 
\quad
|k| \leq N,\\
\,\\
B^N_k(t) = B^N_k(\tau) \quad \mbox{for $|k| > N$},
 \end{array}\right.
\end{equation}
with initial datum $B^N_k(\tau)$, $k \in \mathbb{Z}$. With a little abuse of notation we denote by $\Phi^{N}_{\tau,t}$ the solution map associated with  both \eqref{NLSTruncated} and~\eqref{Bjsyst-Truncated}. 
Also, we will use the same notation $\Phi_{\tau,t}$ for the solution map  of \eqref{NLSv} or \eqref{Bjsyst}, and clearly $\Phi^{N=\infty}_{\tau,t }=\Phi_{\tau,t}$.  We will show that $\Phi^{N}_{\tau,t}$ preserves  the Lebesgue measure 
$$
dL_N (v) := \prod_{|k| \leq N} d ({\rm Re\,} B^N_k) d ({\rm Im\,} B^N_k),
$$
for all $N \in \mathbb{N}$.  This is a consequence of the Hamiltonian structure, however we give a direct proof in the following lemma.
\begin{lemma}\label{Lemma:LCN}
Let $N \in \mathbb{N}$, $ t,\tau \in [1, \infty)$. Then for any Borel set $A \subset H^{s'}(\T)$ we have
$$
L_N (P_{\leq N} \Phi^{N}_{\tau,t} (A) ) = L_N (P_{\leq N} A ). 
$$
\end{lemma}
\begin{proof}
For $|k|\leq N$, in view of \eqref{Bjsyst-Truncated}, it is easy to check that the statement follows by Liouville theorem (coming back to the real and imaginary parts of $B^N_k(t)$)
once we have proved the divergence free condition  
\begin{equation}\label{eq:DivFree}
\frac{\partial F(\{B^N_k(\cdot)\})}{\partial B^N_k(\cdot) } + \frac{\overline{\partial F(\{B^N_k(\cdot)\})}}{\partial \overline{B^N_k(\cdot)} }  = 0. 
\end{equation}
Now when we compute $\frac{\partial F(\{B^N_k(\cdot)\})}{\partial B^N_k(\cdot)}$ we 
have a nonzero contribution only from the terms of the sum such that $k=j_1$ or $k=j_3$. Since in this way the terms 
such that $k=j_1=j_3$ are counted two times, and noting that the sum in $F$ is invariant under interchanging $j_1$ and $ j_3$, we arrive to
$$
\frac{\partial F(\{B^N_k(\cdot)\})}{\partial B^N_k(\cdot)}  = 
 -\frac{i }{t} \sum_{|j_{2}| \leq N}  |B^N_{j_2}(\cdot)|^2 +  \frac{i}{2t} |B^N_k|^2
$$
Similarly we can compute
$$
\frac{\overline{\partial F(\{B^N_k(\cdot)\})}}{\partial \overline{B^N_k(\cdot)} }  = 
 \frac{i}{t} \sum_{|j_{2}| \leq N}  |B^N_{j_2}(\cdot)|^2 -  \frac{i }{2t} |B^N_k|^2
$$
so that \eqref{eq:DivFree} follows.
\end{proof}
We will also need a finite dimensional analogous of the mass conservation.
\begin{lemma}\label{Lemma:MCN}
Let $N \in \mathbb{N} \cup \{ \infty \}$, $t, \tau \in [1, \infty)$. Then 
$\|  \Phi^{N}_{\tau,t} v \|_{L^{2}} =  \|   v \|_{L^{2}}.$
\end{lemma}
\begin{proof}
We multiply the equation \eqref{NLSTruncated} agains $iP_{\leq N} \bar v_N$ then integrate over $\mathbb{T}$ 
and we take the real part. After integration by parts the lemma follows, noting that
$$
{\rm Re\,} \, \frac{i}{t} \int_{\T} (P_{\leq N} \overline v^N) P_{\leq N}  \left( |P_{\leq N} v^N|^2 P_{\leq N}  v^N \right) 
=  
{\rm Re\,} \, \frac{i}{t} \int_{\T} |P_{\leq N} \overline v^N|^4  = 0.
$$
\end{proof}
The following smoothing lemma is the main estimate that is needed to show the quasi-invariance of $\rho_s$.  The analysis is very close to Lemma 4.2 of \cite{FoSe}, but with the important difference that we have dependence on $t$ on the right hand side. This improvement, that will be fundamental for our purposes, 
is due to the fact that we have an extra factor $t^{-1}$ in front of our cubic nonlinearity. 
\begin{lemma}\label{ExpIntBound}
Let $s \in (0,1)$, $M>0$ and $v$ such that $\|v\|_{L^2(\T)} \leq M$, and $t, \tau \in [1, \infty)$. 
The following holds for $\varepsilon >0$ sufficiently small
$$
\left| \|  |\partial_x|^{s+1/2} P_{\leq N}\Phi^{N}_{\tau,t}  v\|_{L^2}^2  -  \| |\partial_x|^{s+1/2}  P_{\leq N} v\|_{L^2}^2 \right| 
\leq C(M) \|  |D|^{s- \varepsilon} P_{\leq N}v \|_{L^2}^3,
$$
where for $\sigma\geq 0$, we denote by  $|\partial_x|^{\sigma}$ the Fourier multiplier of symbol $|k|^{\sigma}$.
\end{lemma}
\begin{proof}
The proof is the same of that of Lemma 4.2 of \cite{FoSe} after the modification that we are going to explain. We first simplify the notation letting 
$$
v_N(t) = P_{\leq N} \Phi^{N}_{\tau,t}  v,
$$
that solves the first equation of \eqref{NLSTruncated}, i.e.
$$i\partial_t v_N+\partial_{xx}v_N+\frac 1t
P_{\leq N}(|v_N|^2v_N)=0.$$
Since
$$
\|  |\partial_x|^{s+1/2} v_N(t) \|_{L^2}^2  -  
\|  |\partial_x|^{s+1/2}  v_N(\tau) \|_{L^2}^2= \int_{\tau}^{t} \frac{d}{d\theta} \|  |\partial_x|^{s+1/2} v_N(\theta) \|_{L^2}^2 \, d\theta,
$$
using  that $ P_{\leq N}$ is symmetric and the identity $P_{\leq N} v_N=v_N$, it follows that
\begin{align*}
\int_{\tau}^{t}\frac{d}{d\theta} \|  |\partial_x|^{s+1/2} v_N(\theta) \|_{L^2}^2 \, d\theta   &= 
- 2  {\rm Im\,}  \int_\tau^t \int_{\T} \ \left( |\partial_x|^{s+1/2} ( |v_N(\theta)|^2 v_N(\theta)) \right) |\partial_x|^{s+1/2} \overline{v_N(\theta)} \, \frac{d \theta}{\theta}
\\ \nonumber
&= 
- 2  {\rm Im\,}  \int_\tau^t \int_{\T} |v_N(\theta)|^2 v_N(\theta)  |\partial_x|^{2s+1} \overline{v_N(\theta)} \, \frac{d \theta}{\theta}.
\end{align*}
Thus once we have shown
\begin{equation}\label{fmdksldkngklskdng}
\left| {\rm Im\,}  \int_\tau^t \int_{\T} |v_N(\theta)|^2 v_N(\theta)  |\partial_x|^{2s+1} \overline{v_N(\theta)} \, \frac{d \theta}{\theta}dx \right| \leq 
C(M) \|  |D|^{s- \varepsilon} v_N(\tau) \|_{L^2}^3 
\end{equation}
the proof is concluded.  Recalling \eqref{Bjsyst-Truncated}, we have reduced the matters to bound the modulus of  
\begin{align}
 {\rm Im\,}  \sum_{\substack{ k_j : |k_j| \leq N \\ k_1 - k_2 + k_3 -k_4 =0 \\  k_4 \neq k_1,  k_3  } }  \int_{\tau}^t 
e^{i(k_1^2 - k_2^2 + k_3^2 - k_4^2) \theta} |k_4|^{2s+1}
(B^N_{k_1} \overline{B^N_{k_2}} B^N_{k_3} \overline{B^N_{k_4}})(\theta) \, \frac{d\theta}{\theta} \, ,
\end{align}
where the restriction $k_4 \neq k_1,  k_3$ comes from the fact that if $k_4 = k_1$ or $k_4 = k_3$ then the integral   
is real. This restriction removes the resonances, that in one dimension are only these trivial ones, 
so hereafter we will always have 
\begin{equation}\label{ResCond}
k_1^2 - k_2^2 + k_3^2 - k_4^2 \neq 0.
\end{equation}
By a symmetrization argument we are reduced to bound the modulus of  
\begin{align*}
 {\rm Im\,}  \sum_{\substack{ k_j : |k_j| \leq N \\ k_1 - k_2 + k_3 -k_4 =0 \\  k_4 \neq k_1,  k_3  } }  \int_{\tau}^t 
e^{i(k_1^2 - k_2^2 + k_3^2 - k_4^2)\theta} (  |k_1|^{2s+1} - |k_2|^{2s+1} + |k_3|^{2s+1} - |k_4|^{2s+1}) 
(B^N_{k_1} \overline{B^N_{k_2}} B^N_{k_3} \overline{B^N_{k_4}})(\theta) \, \frac{d\theta}{\theta} .
\end{align*}
Now letting 
$$
\phi(k_1, k_2, k_3, k_4) := 
\frac{ |k_1|^{2s+1} - |k_2|^{2s+1} + |k_3|^{2s+1} - |k_4|^{2s+1}}{|k_1|^{2} - |k_2|^{2} + |k_3|^{2} - |k_4|^{2}}
$$
and  integrating by parts in time, as allowed by the non resonance condition \eqref{ResCond},  
we can reduce, taking also advantage of the symmetry of the indexes, to estimate the modulus of each of the following terms
\begin{align}\label{fmdjksldjngjksdjng}
& i \sum_{\substack{ k_j : |k_j| \leq N \\ k_1 - k_2 + k_3 -k_4 =0 \\  k_4 \neq k_1,  k_3  } }  \int_{\tau}^t 
e^{i(k_1^2 - k_2^2 + k_3^2 - k_4^2)\theta} \phi(k_1, k_2, k_3, k_4) 
\left( (\partial_{\theta} B^N_{k_1}) \overline{B^N_{k_2}} B^N_{k_3} \overline{B_{k_4}}\right)(\theta) \, \frac{d \theta}{\theta}
\\ \nonumber
&
- i \sum_{\substack{ k_j : |k_j| \leq N \\ k_1 - k_2 + k_3 -k_4 =0 \\  k_4 \neq k_1,  k_3  } }   \int_{\tau}^t 
e^{i(k_1^2 - k_2^2 + k_3^2 - k_4^2) \theta} \phi(k_1, k_2, k_3, k_4)
\left(B^N_{k_1} \overline{B^N_{k_2}} B^N_{k_3} \overline{B_{k_4}} \right)(\theta) \, \frac{d \theta}{\theta^2}
\\ \nonumber
&
- i  \sum_{\substack{ k_j : |k_j| \leq N \\ k_1 - k_2 + k_3 -k_4 =0 \\  k_4 \neq k_1,  k_3  } }    
e^{i(k_1^2 - k_2^2 + k_3^2 - k_4^2) \theta} \phi(k_1, k_2, k_3, k_4)
\left[ \frac{1}{\theta} \left(B^N_{k_1} \overline{B^N_{k_2}} B^N_{k_3} \overline{B^N_{k_4}} \right)(\theta)  \right]_{\theta = \tau}^{\theta = t}.
\end{align}
Now using equation \eqref{Bjsyst-Truncated}
 we see that the 
  first term of \eqref{fmdjksldjngjksdjng}  becomes, after renaming the indexes,
 $$ 
 \!\!\!\!\!\!\!\!\!\!\!\!\!\!\!\!\!\!  \sum_{\substack{ k_j : |k_j| \leq N \\ k_1-k_2 + k_3-k_4+k_5-k_6=0 \\ |k_4-k_5+k_6|^{2} - |k_2|^{2} + |k_3|^{2} - |k_4|^{2} \neq 0 } } 
  \!\!\!\!\!\!\!\!\!\!\!\!\!\!\!\!\!\!
   \int_{\tau}^t 
e^{i(k_1^2 - k_2^2 + k_3^2 - k_4^2 + k_5^2 - k_6^2)\theta} \phi(k_4-k_5+k_6, k_1, k_2, k_3) 
\!\!\!\! \prod_{k=1,3,5} \!\!\!\! B^N_{k_\ell}(\theta) \!\!\!\! \prod_{k=2,4,6} \!\!\!\! \overline{B^N_{k_\ell}  (\theta)}  \, \frac{d \theta}{\theta^2}.
$$
Now we are in the exact same situation as in the proof of Lemma 4.2 of \cite{FoSe}, but we have an extra factor $\theta^{-1}$ 
that, combined with the uniform control of the $l^{2,s}$ norms of $\{ B^N_{k}(\theta)\}$ from \S \ref{ssectconstr},
allows to obtain an estimate which is uniform over $t >1$, as claimed in \eqref{fmdksldkngklskdng}, proceeding exactly as in \cite{FoSe}.
 \end{proof}
\subsubsection*{Proof of \eqref{QuasiInvFin}}
We are now ready to show the quasi-invariance of $\rho_s$ under $\Phi_{1,t}$, that is the content of the first part of Theorem \ref{MainThm3}. In fact we will prove the quantitative estimate~\eqref{QuasiInvFin}, and even the more general estimate
\begin{equation}\label{QuasiInvFin_general}
\rho_s \left( \Phi_{\tau,t} (A) \right) \leq C_{\kappa} \,\rho_s  ( A )^{1 - \kappa}, \quad \forall\kappa >0,\quad \forall \, \tau, t\geq 1.
\end{equation}
We denote 
 with  $\gamma_{s,N}^{\perp}(v)$ the Gaussian measure induced by the random variable
\begin{equation}\label{Def:GaussmEasurePerp}
\omega \longmapsto  \sum_{|k| >N} 
\frac{g_k^{\omega} }{
(1 + |k|^{2s +1})^{\frac 1 2}
}   e^{i x k}\, .
\end{equation}
The measure $ \gamma_s$ factorizes as
\begin{align*}
 d \gamma_s (v) 
& = e^{- \left(\| P_{\leq N}  v\|_{L^2}^2 + \| P_{\leq N}  v\|_{\dot{H}^{s + 1/2}}^2 \right) } d L_N (P_{\leq N} v) \,  
d \gamma_{s,N}^{\perp}(P_{>N} v).
\end{align*}
Thus given any Borel set $A \subset H^{s'}$, $s'<s < 1/2$, we have
\begin{align*}
 \rho_s \big( \Phi_{\tau,t}^{N} (A) \big) 
 & =
 \int_{\Phi_{\tau,t}^{N} (A)}   d \rho_s (v)
\\ \nonumber
&= \int_{\Phi_{\tau,t}^{N} (A)}  1_{\big\{ \| v \|_{L^{2}} \leq M \big\} } d \gamma_s(v) 
 \\ \nonumber
& = 
\int_{\Phi_{\tau,t}^{N} (A)}  
1_{\big\{ \| v \|_{L^{2}} \leq M \big\} } 
e^{- \big( \| P_{\leq N}  v\|_{L^2}^2 + \| P_{\leq N}  v\|_{\dot{H}^{s+1/2}}^2 \big) } d L_N (P_{\leq N} v ) \,  d \gamma_{s,N}^{\perp}(P_{>N} v)\\ \nonumber
& =
\int_{A}  
1_{\big\{ \| v \|_{L^{2}} \leq M \big\} } 
e^{- \big( \| P_{\leq N}  v\|_{L^2}^2 + \| P_{\leq N}  \Phi^{N}_{\tau,t}  v\|_{\dot{H}^{s+1/2}}^2 \big) } d L_N (P_{\leq N} v) \,  d \gamma_{s,N}^{\perp}(P_{>N} v)\\ \nonumber
& =
\int_{A} 1_{\big\{ \| v \|_{L^{2}} \leq M \big\} } 
e^{- \big(  \| P_{\leq N}  \Phi^{N}_{\tau,t}  v\|_{\dot{H}^{s + 1/2}}^2  -  \| P_{\leq N}  v\|_{\dot{H}^{s + 1/2}}^2  \big)} 
 \,  d \gamma_s(v)\,
 \nonumber
 \\ & =
 \int_{A} 
e^{- \big( \| P_{\leq N}  \Phi^{N}_{\tau,t}  v\|_{\dot{H}^{s + 1/2}}^2  -  \| P_{\leq N}  v\|_{\dot{H}^{s + 1/2}}^2 \big)} 
 \,  d \rho_s(v)\,, \nonumber
\end{align*}
where in the fourth identity we used Lemmas \ref{Lemma:LCN}-\ref{Lemma:MCN}, and the fact that in view of \eqref{NLSTruncated} we have $P_{>N}\Phi_{\tau,t}^{N} v=e^{i(t-\tau)\Delta}P_{>N}v$, 
to change variables in the indicator function and in the Lebesgue measure. Now using 
Lemma \ref{ExpIntBound} we arrive to
$$
 \rho_s \left( \Phi_{\tau,t}^{N} (A) \right)   \leq
 \int_{A} 
e^{C(M) \| P_{\leq N} |D|^{s- \varepsilon} v \|_{L^2}^3} 
 \,  d \rho_s(v)\,.
$$
Thus using H\"older inequality
$$
\rho_s \big( \Phi_{\tau,t}^{N} (A) \big)  \leq
 \rho_s(A) ^{1- \frac1p} \left( \int_{A} 
e^{p \, C(M) \| P_{\leq N} |D|^{s- \varepsilon} v \|_{L^2}^3} 
 \,  d \rho_s(v) \right)^{\frac1p}.
$$
On the other hand by Lemma 5.1 of \cite{FoSe} we have, for all $p < \infty$ and $N\in\mathbb N$,
$$\int_A e^{p \, C(M) \| P_{\leq N} |D|^{s- \varepsilon} v \|_{L^2}^3} 
 \,  d \rho_s(v) \leq C(p,M),$$
so, letting $\kappa = 1/p$ we have arrived at 
\begin{equation}\label{QuasiInvQuant}
\rho_s \left( \Phi_{\tau,t}^{N} (A) \right) \leq C_{\kappa} \rho_s  ( A )^{1 - \kappa},
\qquad \kappa >0.
\end{equation}
Now we would  like to pass to the limit $N \to \infty$ in this inequality to obtain \eqref{QuasiInvFin}, using the strategy introduced in \cite{sigma}.
First we do this for compact sets $A$ in the chosen $H^{s'}$ topology. Since $A$ is compact it is in particular bounded in $H^{s'}$, thus we can use the stability estimate recalled at the 
end of this section  to show that for all $\varepsilon >0$ we can find $N_{\varepsilon}$ sufficiently large that
\begin{equation}\label{FlowStability}
\Phi_{\tau,t} (A) \subset 
\Phi^{N}_{\tau,t} (A + B_{\varepsilon}), \qquad \forall N>N_\varepsilon,
\end{equation}
where $B_{\varepsilon}$ is a ball of radius $\varepsilon >0$ in the $H^{s''}$ topology, where   $s''<s'$ is close to $s'$.
We will show how to do it at the end of the section.  Once we have proved \eqref{FlowStability} we can then proceed as follows.
Since $A$ is compact in $H^{s'}$ it is also a closed set in $H^{s''}$ and therefore by the dominated convergence 
for all $\delta >0$ we have 
$\rho_s  ( A + B_{\varepsilon}) - \rho_s  ( A )  \leq \delta $ for all $\varepsilon$ sufficiently small (depending on $\delta$).
Thus 
$$
\rho_s( \Phi_{\tau,t} (A)) \leq  \rho_s( \Phi^{N}_{\tau,t} (A + B_{\varepsilon})) \leq 
C_{\kappa} \rho_s  ( A + B_{\varepsilon} )^{1 - \kappa}  < 
C_{\kappa} ( \rho  ( A ) + \delta )^{1 - \kappa}  
$$ 
and since $\delta >0$ is arbitrary we have indeed proved \eqref{QuasiInvFin_general} (and in particular \eqref{QuasiInvFin})
for $A$ compact. The argument then extends to Borel sets $A \subset H^{s'}$ 
using the inner regularity of the 
Gaussian measure $\gamma_s$ and the continuity of the solution map (see Lemma \ref{InverseFlowStab}). 
\subsubsection*{Proof of \eqref{FlowStability}}
Letting 
$$
\Psi_N(v) = (\Phi^{N}_{\tau,t})^{-1}  \Phi_{\tau,t}  (v)
$$
where $(\Phi^{N}_{\tau,t})^{-1}$ is the backward truncated solution map from time $t$ to time $\tau$, the 
inclusion \eqref{FlowStability} follows once we prove for $A$ compact in the $H^{s'}$ topology and $s'<s$ the existence of $N_\varepsilon$ such that 
\begin{equation}\label{laterBettExpl}
\| v - \Psi_N v\|_{H^{s''}} < \varepsilon, \qquad \forall N>N_\varepsilon, \qquad \forall v\in A.
\end{equation}
This is simply because $\Phi^{N}_{\tau, t} \Psi_N(v) =  \Phi_{\tau, t}  (v)$, which means that given any 
$v \in A$ we can write $\Phi_{\tau, t}  (v)$ as the $\Phi^{N}_{\tau, t}$ evolution of an initial datum
$\Psi_N(v)$
contained in $v + B_{\varepsilon}$ (by \eqref{laterBettExpl}), that is exactly what  \eqref{FlowStability} means.   To prove \eqref{laterBettExpl} we estimate 
\begin{align}\nonumber
\| v - \Psi_N(v) \|_{H^{s''}} 
& =
\| (\Phi^{N}_{\tau,t})^{-1}  \Phi_{\tau,t}  (v) - (\Phi^{N}_{\tau,t})^{-1}  \Phi^{N}_{\tau,t}(v) \|_{H^{s''}}
\\ \nonumber
&
\leq C(A, M) \|  \Phi_{\tau,t}  (v) -   \Phi^{N}_{\tau,t}(v) \|_{H^{s''}} , 
\end{align}
where we used  that $(\Phi^{N}_{\tau,t})^{-1} =\Phi^{N}_{t,\tau}$, estimate \eqref{InvFlowPolyTruncated} and the fact that
$A$ is compact, so that in particular $\sup_{v \in A} \| v \|_{H^{s'}} < \infty$. 
Finally, Lemma 2.27 of \cite{B94_pak}  implies that
$$
\sup_{v\in A}\|  \Phi_{\tau,t}  (v) -   \Phi^{N}_{\tau,t}(v) \|_{H^{s''}} \leq CN^{s''-s'}
$$
which yields \eqref{laterBettExpl}.
\section{Limit density}
In this section we prove that the transported measures $(\Phi_{1,t})_* \rho_s$ that we have constructed in the previous section have densities  
and that these densities with respect to $\rho_s$ attain a limit as $t \to \infty$. 
This shows the validity of the representation formula \eqref{QuasiInvFinDensity} and so it concludes the proof 
of Theorem \ref{MainThm3}.
At fixed time $t >1$ we define the family of  approximated densities
 $$
 f_N(t, v) := 
 e^{2 {\rm Im\,}  \int_1^t \int_{\T}  |\Phi^N_{1,\tau}|^2 \Phi^N_{1,\tau}(v)  |D|^{2s+1} \overline{\Phi^N_{1,\tau}(v)} \, \frac{d \tau}{\tau} dx } 
 $$
 where we recall that $\Phi^{N}_{1,t}$ is the solution map of the truncated forward evolution equation \eqref{NLSTruncated}. Let $q >1$. By the analysis of Section \ref{Sec:QuasiInvariance}  we have the uniform bound 
$$
\sup_{N} \|  f_N(t, v) \|_{L^q(\rho_s)} < \infty.
$$
Moreover, invoking Lemma \ref{ExpIntBound} one can prove that the sequence $f_N(t, v)$ converges 
in measure (as $N \to \infty$) to 
$f(t, v)$ 
where 
\begin{equation}\label{fmdksldkngklskdngBis}
 f(t, v) := e^{ 2  {\rm Im\,}  \int_1^t \int_{\T}  |\Phi_{1,\tau}(v)|^2 \Phi_{1,\tau}(v)  |D|^{2s+1} 
 \overline{\Phi_{1,\tau}(v)} \, \frac{d \tau}{\tau} dx}.
\end{equation}
These two facts imply that $f_N(t, v)$ converges to $f(t, v)$ in $L^{q'}$ for all $q' < q$ (see Lemma~3.7 of \cite{Tzv08}).
Recalling then the estimate \eqref{fmdksldkngklskdng} we see that the density satisfies 
\begin{equation}\label{fmdksldkngklskdngTris}
|f(t, v)| \leq e^{ C(M) \|  |D|^{s- \varepsilon} v \|_{L^2}^3 },  
\end{equation}
for all sufficiently small $\varepsilon >0$. Moreover, we have good approximation properties for the solution map $\Phi_{1,t}$
with the approximated solution map
$\Phi^{N}_{1,t}$ in the $H^{s'}(\T)$ topology for all $s' < s$ if we restrict to compact sets (again, this is 
a consequence of Lemma 2.27 in \cite{B94_pak} taking into account the local well-posedness theory that we 
developed in Section \ref{ssectconstr}). Once we have these three ingredients
we can proceed exactly as in the proof of Proposition~7.2 of \cite{GLT23} (that is written in an abstract setting) 
to show  
that the density $f(t, v)$ coincides with that of the transported measure $(\Phi_{1,t})_* \rho_s$ , namely
that for any Borel set $A$ (in the $H^{s'}(\T)$, $s' < s$, topology)  we have 
$$
\rho_s(\Phi_{1,t} (A))=  \int_{A} f(t, v) d \rho_s(v).
$$
Now we want to pass to the limit $t \to \infty$ into \eqref{fmdksldkngklskdngBis}, that is the novelty of this model. 
We can do it using \eqref{fmdksldkngklskdngTris} and the dominated convergence Theorem, getting
 $$
\lim_{t \to \infty} f(t, v) =  f(\infty, v) :=  
e^{ 2 {\rm Im\,}  \int_1^\infty \int_{\T} |\Phi_{1,\tau}(v)|^2 \Phi_{1,\tau}(v)  |D|^{2s+1} 
 \overline{\Phi_{1,\tau}(v)} \, \frac{d \tau}{\tau} dx},
 $$
 that concludes the proof of Theorem \ref{MainThm3}. 
%
\section{Growth of H\"older norms of typical solutions of NLS}\label{Sec:thbf2}
In this section we aim to control the growth as $t \to \infty$ of the  $\mathcal C^{s'}$-norm of $\rho_s$-typical solutions, for $s' < s$, in a quantitative way. This will be crucial, in the next section, in order to prove Theorem~\ref{thbf2}.  More precisely, the main goal of this section is to establish  Theorem~\ref{QuantControlCsBis} below. We will need a number of preliminary propositions. The 
proof of Theorem~\ref{QuantControlCsBis}  as a consequence of these propositions will be presented at the end of the section.
Here is the main result of this section. 
\begin{theorem}\label{QuantControlCsBis}
Let $s \in (0,1)$, $0 \leq s' < s$ and $\varepsilon \in (0,1/2)$. For $\rho_s$-almost every $f$ there exists a
 diverging sequence of natural numbers $\{ N_n(f) \}_{n \in \N}$ and $C(f)>0$, $T(f) \geq1$ 
such that for all  $t > T(f)$:
\begin{equation}\label{QuantControlCsBisFinalBound}
\limsup_{n \to \infty}   \|  \Phi_{1, t} P_{\leq N_n(f)} f \|_{\mathcal C^{s'}} \lesssim C(f) +  t^{\varepsilon}.
\end{equation}
\end{theorem}
In the proof of Theorem~\ref{QuantControlCsBis}  we first introduce a gauge transform which allows to use a well known local analysis in the Bourgain spaces. This analysis leads to local in time bounds of the  H\"older norms of typical solutions.
The passage to the global in time bound \eqref{QuantControlCsBisFinalBound} is much less standard. It is based on measure  propagation arguments, using in a crucial way the quantitative quasi-invariance of the measure $\rho_s$ established in Theorem~\ref{MainThm3}.
\subsection{A gauge transform}
In order to attack the problem, it will be useful to introduce the gauge transformation 
\begin{equation}\label{GaugeMap}
w^N(t) := e^{-2i \mu_N (v^N) \log t} v^N(t),
\end{equation}
where $v^N$ solves \eqref{NLSTruncated} 
and 
\begin{equation}
\mu_N (f) := \int_{\T} |P_{\leq N} f(x)|^2 dx.
\end{equation}
We recall that the operator $P_{\leq N}$ is the projection of the first $|k| \leq N$ Fourier modes of a function on~$\T$ and
 $P_{>N}  := \text{Id} - P_{\leq N}$.
We also recall that, from Lemma \ref{Lemma:MCN}
\begin{equation}\label{MuNGaugeMap}
\mu_N(v^N(t, \cdot)) := \int_{\T} |P_{\leq N} v^N(t, x)|^2 dx = \int_{\T} |P_{\leq N} v^N(1, x)|^2 dx .
\end{equation}
When $N= \infty$ we abbreviate 
$w := w^\infty$ (we similarly abbreviated $v := v^\infty$). 
When there is no ambiguity we abbreviate $\mu_N = \mu_N(v^N(t, \cdot)) =  \mu_N(w^N(t, \cdot))$ 
and $\mu_N = \mu_N(f)$ as well.
\\

In view of \eqref{GaugeMap} and \eqref{NLSTruncated}
The function $P_{\leq N} w^N$ is a solution of the following gauged truncated
cubic NLS:
\begin{equation}\label{NLSvIntrogauged}
i \partial_t P_{\leq N} w^N + \partial_{xx} P_{\leq N} w^N +
\frac{1}{t}  P_{\leq N} \mathcal{N} (P_{\leq N}w^N) 
=0, \quad \mathcal{N}(f) := (|f|^2 - 2 \mu_N) f, 
\end{equation}
defined for $ 1 \leq t <\infty$ and $x \in  \T$. 
Note that the nonlinearity $\mathcal{N}$ also depends on $N$ through~$\mu_N = \mu_N(f)$. We omit this dependence in order to simplify the notations (the value of $\mu_N$ will be always clear by the context).
On the high Fourier modes we simply have a gauged linear evolution
\begin{align}\label{HighModesLinearEvolutionGauged} 
P_{> N} w^N(t) = e^{- 2 i \mu_N (\ln t - \ln \tau) +  i(t-\tau) \Delta} P_{> N} w^N(\tau), \qquad \tau, t \in [1, \infty).
\end{align}
We will denote by $\tilde{\Phi}^{N}_{\tau,t}$ the solution map associated with \eqref{NLSvIntrogauged}-\eqref{HighModesLinearEvolutionGauged} and we abbreviate 
as usual $\tilde{\Phi}_{\tau,t} := \tilde{\Phi}^{N=\infty}_{\tau,t }$. From \eqref{GaugeMap} and \eqref{MuNGaugeMap} we have for all $N \in \N \cup \{ \infty\}$ and for all $\tau, t \in [1, \infty)$
\begin{equation}\label{ComparisonFlows}
\tilde{\Phi}^N_{\tau,t} f = e^{-2i \mu_{N} \ln t} \Phi^N_{\tau, t}  e^{2i \mu_{N} \ln \tau} f
= e^{-2i \mu_{N} (\ln t - \ln \tau)} \Phi^N_{\tau, t}   f,
\end{equation}
where we recall that $\Phi^{N}_{\tau,t}$ is the solution map associated with \eqref{NLSTruncated}.
The interest of the gauge transform \eqref{GaugeMap} is that it removes resonant interactions which is of importance in the local in time smoothing properties discussed in the next section.
\subsection{Local in time smoothing properties}
We will need some preliminary probabilistic local in time well-posedness results in which we focus  
on solutions of \eqref{NLSvIntrogauged} evolving from the initial data
prescribed at a generic initial time $T \in  [1, \infty)$.
Roughly speaking, we will show that the Duhamel contribution $w^N(t) - e^{i (t - T) \Delta} w(T)$ 
is almost surely $1/2 - \varepsilon $ smoothing, for all $\varepsilon >0$. We note that we 
use in a fundamental way the random oscillations of the initial data to get $1/2 - \varepsilon $ smoothing for $s$ close to zero, compared to the $\min\{2s,\frac 12\}$ smoothing in the deterministic case in \cite[Theorem 4.1, Lemma 4.2]{ETsmoothing}.  
\\

We denote by
$
B^{r}(R) := \left\{ f: \| f \|_{H^{r}} \leq R \right\}
$
the ball centred in zero of radius $R$, in the $H^{r}$ topology. 
Let $\beta>0$ and $s' < s$. We define the following exceptional set of functions 
\begin{equation}\label{DEf:OmegaEx}
\Omega_{\delta} :=  \Big\{ f : \| f \|_{H^{s'}}  >  \delta^{-\beta/3} \ \mbox{or} \ 
|\hat{f}(k)| (1 + |k|^{2s+1})^{\frac12}  > \delta^{- \frac13 \beta} \langle k \rangle^\varepsilon \, \mbox{for some $k \in \Z$}
 \Big\}  
 \end{equation}
 outside which we have local well-posedness and smoothing of the Duhamel contribution to the solution,
  in the sense of the following proposition. 
  The set $\Omega_{\delta}$ is exceptional in the sense that his measure is exponentially small with $\delta$. More precisely
 \begin{equation}\label{MeasureOmegaExceptPert}
\gamma_s(\Omega_{\delta}) \leq e^{-\delta^{-\beta}}
\end{equation} 
for all $\delta >0$ sufficiently small; this last assertion can be proved using some basic properties of the Gaussian random variables used to define the measure $\gamma_s$ in \eqref{TMIATSIntro}.
\begin{prop}\label{1/2SmoothingNLSPert} 
Let $s \in (0,1)$, $\sigma \in (1/4, 1/2)$ and $T \geq 1$. 
There exist $\gamma > \beta >0$, $\bar{\delta} \in (0,1)$ such that the following holds. 
For all $\delta \in (0, \bar{\delta})$,
all $$w(T) \in \Omega_{\delta}^\complement + B^{s + \sigma}(1)$$ 
and all $N \in 2^{\N} \cup \{ \infty \}$,
the solution to the (renormalized, truncated) cubic NLS \eqref{NLSvIntrogauged} 
on $t \in [T, T + \delta]$ has the form 
$$
P_{\leq N} w^N(t) = e^{i (t - T)  \Delta} P_{\leq N} w(T)  + \tilde{w}^N, 
$$
where the Duhamel contribution $\tilde{w}^N$ satisfies 
\begin{equation}\label{BoundWNSmoothing}
\| \tilde{w}^N \|_{X^{s + \sigma, \frac12 + \gamma}_{[T, T + \delta]}} \leq 1,
\end{equation}
where the restricted Bourgain norms are defined in \eqref{Bloc}.
\end{prop}
\begin{remark}
It is worth to stress out that the set $\Omega_{\delta}$ (and thus $\Omega_{\delta}^{\complement}$) 
is independent of $N$ and $T$ (recall definition \eqref{DEf:OmegaEx}). 
\end{remark}
\begin{remark}\label{OmegaDeltaStability}
By the definition \eqref{DEf:OmegaEx} of the set $\Omega_{\delta}$ one can  see that 
$f \in \Omega_{\delta}$ if and only if $\bar{f} \in \Omega_{\delta}$; in fact $e^{i \alpha} \Omega_{\delta} = \Omega_{\delta}$ for all  $\alpha \in \R$. Moreover  
$e^{i t \Delta} \Omega_{\delta}= \Omega_{\delta}$ for all $t \in \R$. The same properties are then true for $\Omega_{\delta}^{\complement}$.
Regarding $\Omega_{\delta}^{\complement}$, one can see that 
$P_{\leq N} \Omega_{\delta}^{\complement} \subset \Omega_{\delta}^{\complement}$ and 
$P_{> N} \Omega_{\delta}^{\complement} \subset \Omega_{\delta}^{\complement}$ 
for all $N \in \N$ as well.  
\end{remark}
\begin{remark}\label{SigmaRem}
The parameters $\gamma > \beta >0$ in the statement have to be taken sufficeintly small.
\end{remark}
\begin{proof}
We follow the method introduced in \cite{Bo96} in the (more difficult) 2d case.  
Since in our case $s>0$, we have that the resonant contribution is a part of the smoother term $\tilde{w}^N$. Hereafter $N$, $N_j$ will denote dyadic scales ($N$ can be taken also $= \infty$). The nonlinearity \eqref{NLSvIntrogauged} can be decomposed as
$$
\mathcal{N}(P_{\leq N} w^N(t, x))= \mathcal{N}_1(P_{\leq N} w^N(t, x)) + \mathcal{N}_2(P_{\leq N} w^N(t, x))
$$
where 
$$
\mathcal{N}_1(P_{\leq N} w^N(t, x)) := 
\sum_{\substack{n_1, n_2, n_3 \\ n_2 \neq n_1, n_3 \\ |n_1|, |n_2|, |n_3| \leq N}} 
\widehat{w^N(t,x)}(n_1) \overline{\widehat{w^N(t,x)}}(n_2) \widehat{w^N(t,x)}(n_3) e^{i (n_1 - n_2 + n_3) \cdot x}
$$
$$
\mathcal{N}_2(P_{\leq N} w^N(t, x)) :=  - \sum_{|n| \leq N} \widehat{w^N(t, x)}(n)  |\widehat{w^N(t, x)}(n)|^2 e^{i n \cdot x} .
$$

Since $w(T) \in \Omega_{\delta}^\complement + B^{s + \sigma}(1)$ we have 
$w(T) = f + g$ 
where $f \in \Omega_{\delta}^{\complement}$ and $\| g \|_{H^{s+\sigma}} \leq 1$. Thus  
$$e^{i (t - T)  \Delta} P_{\leq N} w(T) = e^{i (t - T)  \Delta} P_{\leq N} f +  G , \qquad G:= e^{i (t - T)  \Delta} P_{\leq N} g $$
and the function $G$ satisfies 
$$\|\eta(t - T) G\|_{X^{s+\sigma, \frac12 + \gamma} } \lesssim 1,$$ 
where $\eta : \R \to \R^{+}$ is a smooth compactly supported function such that $\eta = 1$ on $[-1,1]$. 

Thus, taking $\gamma > \beta$ and $\bar{\delta}$ sufficiently small (recall $\delta \in (0, \bar{\delta}$))
we can easily check that the statement follows 
once we have proved the 
multilinear estimate
\begin{equation}\label{fmdksldkngksl}
\| \mathcal{N}_\ell (P_{\leq N} w_1(J_1), P_{\leq N} w_2(J_2), P_{\leq N} w_3(J_3) ) \|_{X^{s + \sigma, -\frac12 + \gamma}_{[T, T+\delta]}} \lesssim \delta^{\gamma}\prod_{j = 1}^3 a(J_j) ,  
  \end{equation}
for all $f$ outside $ \Omega_{\delta}$, 
where\footnote{Here $R$ and $D$ stands for random and deterministic, respectively.} 
$\ell \in {1,2}$, $j \in \{1,2,3\}$, $J_j \in \{ R, D \}$, 
 $$
w_j(R) :=  \eta( t - T )
e^{i(t- T)\Delta} f,
\qquad a(R) :=  \delta^{^{-\beta/3}} \, ,
$$
$$
 w_j(D) \in X^{s + \sigma, \frac12 + \gamma},  \qquad a(D) 
 := \| w_j \|_{X^{s + \sigma, \frac12 + \gamma}}\, .
$$
We stress out that the exceptional set $\Omega_{\delta}$ and the implicit constant in \eqref{fmdksldkngksl} are independent of $N \in 2^{\N} \cup \{ \infty \} $ and $T \geq 1$. We will refer to $w_j(D)$ and $w_j(R)$ as deterministic and random contributions, respectively. This is 
because the size of $w_j(R)$ will be always estimated outside the exceptional set $\Omega_{\delta}$. Using a basic property of the restriction spaces, we see that the estimate \eqref{fmdksldkngksl} follows 
from the validity of
\begin{equation}\label{fmdksldkngkslBis}
\| \eta(t-T) \mathcal{N}_\ell(P_{\leq N} w_1(J_1), P_{\leq N} w_2(J_2), P_{\leq N} w_3(J_3) ) \|_{X^{s + \sigma, -\frac12 + 2 \gamma}} \lesssim \prod_{j = 1}^3  a(J_j) ,
\end{equation}
$\ell \in {1,2}$,  $j \in \{1,2,3\}$, $J_j \in \{ R, D \}$,
for $f$ outside $\Omega_{\delta}$.
Again, the exceptional set $\Omega_{\delta}$ and the implicit constant in \eqref{fmdksldkngkslBis} are 
independent of~$N \in 2^{\N} \cup \{ \infty \} $ and $T \geq 1$. We will first prove \eqref{fmdksldkngkslBis} for $\ell=1$, that is the hardest case.
We thus have to prove
\begin{multline}\label{mfkdlskdjgkldskjg}
 \Big\| \eta(t - T) \sum_{\substack{ N_1, N_2, N_3 \\ N_1, N_2,   N_3 \leq N }} \mathcal{N}_1(P_{N_1} w_1(J_1), P_{N_2} w_2(J_2), P_{N_3}w_3(J_3) ) \Big\|_{X^{s+ \sigma, - \frac12 + 2 \gamma}}
 \\
 \lesssim  \prod_{j = 1}^3 a(J_j) ,  
\end{multline}
where $N_j \in 2^{\N_0}$ are dyadic scales and $P_{1} := P_{\leq 1}$ while
$$
P_{N_j} := P_{\leq N_j} - P_{\leq \frac12 N_j}, \quad N_{j} \geq 2.
$$
First we focus on the contribution to \eqref{mfkdlskdjgkldskjg} of the terms of the sum for which
the deterministic part of the solution acts at the  highest frequency. The argument will be independent 
on whether the highest frequency is 
$N_1, N_2$ or $N_3$, so we assume w.l.g. $N_1 \geq N_2, N_3$ and that $J_1 = D$.
By duality and Fourier support considerations  we can bound this contribution as
\begin{multline}\label{fdklsòdlkfsòl}
\lesssim 
\sup_{\| \phi \|_{X^{-s-\sigma, \frac12 - 2 \gamma}} \leq 1} 
\\
\sum_{\substack{N_1, N_2, N_3, N_4 \\ N_4 \lesssim N_1  , \,\,\, N_2, N_3  \leq N_1} } 
\left| \int N_1^{s + \sigma} \left( \mathcal{N}_1(P_{N_1} w_1(D), P_{N_2} w_2(J_2), P_{N_3}w_3(J_3) ) \right) N_4^{-s - \sigma } P_{N_4} \phi  \right|.
\end{multline}
We first restrict to the contributions for which $N_2 \simeq N_1$ (the case $N_3 \simeq N_1$ will be the same switching the role of $N_2$ and $N_3$).
Using the embedding 
 $X^{0, \frac38} \hookrightarrow L^4_{t,x}$ 
we estimate this as
\begin{align*}
 \lesssim 
& \sup_{\| \phi \|_{X^{-s-\sigma, \frac12 - 2 \gamma}} \leq 1}
\\ 
 &\sum_{\substack{N_1, N_2, N_3, N_4 \\ N_4 \lesssim N_1  , \,\,\, N_2, N_3  \leq N_1} } 
 \| P_{N_1} w_1(D) \|_{X^{s+\sigma, \frac38}}  \| P_{N_2} w_2(J_2) \|_{X^{0, \frac38}}
 \| P_{N_3} w_3(J_3) \|_{X^{0, \frac38}}  \| P_{N_4} \phi \|_{X^{-s -\sigma, \frac38}} .
 \end{align*}
Using
 Cauchy-Schwarz in $N_1, N_2,  N_3, N_4$ and recalling that $N_2 \simeq N_1  \gtrsim N_3, N_4$ we have, for all~$\varepsilon >0$
\begin{align*}
& \lesssim_{\varepsilon} \sup_{\| \phi \|_{X^{-s-\sigma, \frac12 - 2 \gamma}} \leq 1} 
\\
&
 \!\! \Bigg( \sum_{\substack{N_1, N_2, N_3, N_4 \\ N_4 \lesssim N_1  , \,\,\, N_2, N_3  \leq N_1} } \!\!\!\!\!\!\!\!\!
 N_2^{4\varepsilon} \| P_{N_1} w_1(D) \|^2_{X^{s + \sigma, \frac38}}   \| P_{N_2} w_2(J_2) \|^2_{X^{0, \frac38}}
  \| P_{N_3} w_3(J_3) \|^2_{X^{0, \frac38}}   \| P_{N_4} \phi \|^2_{X^{- s -\sigma, \frac38}} \Bigg)^{\frac12}
\\ &
 \lesssim_{\varepsilon}
 \sup_{\| \phi \|_{X^{-s-\sigma, \frac12 - 2 \gamma}} \leq 1} 
  \|  w_1(D) \|_{X^{s + \sigma, \frac38}}   \|  w_2(J_2) \|_{X^{4 \varepsilon, \frac38}}
  \| w_3(J_3) \|_{X^{0, \frac38}}  \|  \phi \|_{X^{- s -\sigma, \frac38}} ,
 \end{align*}
 that for $\gamma, \varepsilon >0$ sufficiently small that $\frac38 < \frac 12 - 2 \gamma$ and 
 $4 \varepsilon < s$
is bounded by the r.h.s. of \eqref{fmdksldkngkslBis}, for $f$ outside the exceptional set $\Omega_{\delta}$. 
This last fact is clear when we look at the deterministic contributions 
$J_2 = D$ and/or $J_3 = D$, while for the random ones follows by the fact that
$$
\| w_j(R)\|_{X^{s', b}}  = \|\eta( t - T )
e^{i(t- T)\Delta} f \|_{X^{s', b}} \lesssim \| f\|_{H^{s'}} \lesssim 
\delta^{-\beta/3}, \qquad s' < s,
$$
where the last identity follows by  $f \in \Omega_{\delta}^\complement$.
\\

Coming back to \eqref{fdklsòdlkfsòl}, we must then consider the contributions for which $N_2,  N_3 \ll N_1$. By Fourier support consideration this forces $N_1 \simeq N_4$, thus we bound these contributions to \eqref{fdklsòdlkfsòl} with
\begin{align*}
\lesssim 
& \sup_{\| \phi \|_{X^{-s-\sigma, \frac12 - 2 \gamma}} \leq 1} 
\\ &
 \sum_{\substack{N_1, N_2, N_3, N_4 \\ N_4 \simeq N_1  , \,\,\, N_2, N_3  \leq N_1} } 
 \| P_{N_1} w_1(D) \|_{X^{s+\sigma, \frac38}}  \| P_{N_2} w_2(J_2) \|_{X^{0, \frac38}}
 \| P_{N_3} w_3(J_3) \|_{X^{0, \frac38}}  \| P_{N_4} \phi \|_{X^{-s -\sigma, \frac38}} 
 \end{align*}
Since the sum is restricted to $N_1 \simeq N_4$, 
 using Cauchy-Schwarz only in $N_1$ we get the estimate   
\begin{equation*} 
 \sup_{\| \phi \|_{X^{-s-\sigma, \frac12 - 2 \gamma}} \leq 1} 
 \| w_1(D) \|_{X^{s + \sigma, \frac38}}  \| \phi \|_{X^{- s -\sigma, \frac38}} 
  \sum_{\substack{N_2, N_3 } }
  \| P_{N_2} w_2(J_2) \|_{X^{0, \frac38}}
 \| P_{N_3} w_3(J_3) \|_{X^{0, \frac38}}    
\end{equation*}
then by Cauchy-Schwarz in $N_2, N_3$ we get, for all $\varepsilon >0$ 
 \begin{equation*}
 \lesssim_{\varepsilon}
 \sup_{\| \phi \|_{X^{-s-\sigma, \frac12 - 2 \gamma}} \leq 1} 
  \|  w_1(D) \|_{X^{s + \sigma, \frac38}}  \|  \phi \|_{X^{- s -\sigma, \frac38}}  \|  w_2(J_2) \|_{X^{\varepsilon, \frac38}}
  \| w_3(J_3) \|_{X^{\varepsilon, \frac38}}  ,
 \end{equation*}
 that, as before, for $\gamma, \varepsilon >0$ sufficiently small that $\frac38 < \frac 12 - 2 \gamma$ and 
 $ \varepsilon < s$,
is bounded by the r.h.s. of \eqref{fmdksldkngkslBis}, for $f$ outside $\Omega_{\delta}$. 
\\

To handle the remaining contributions to \eqref{mfkdlskdjgkldskjg},  namely when the random part of the solution decomposition acts as  
largest frequency, we recall a basic fact about
 restriction spaces. If $w \in X^{s + \sigma,b}$ we can represent it as 
 $$
 w(x,t) = \int \phi(\lambda) \sum_{n} b_{\lambda}(n) e^{i n x + i (n^2+\lambda) t } d \lambda,
 $$ 
 where 
 $$
 b_{\lambda}(n) := \frac{\widehat{w}(n, \lambda + n^2)}{\left( \sum_m \langle m \rangle^{2s + 2\sigma} |\widehat{w}(m, \lambda + n^2)|^2 \right)^{1/2}}
 $$
 and 
 $$
 \phi(\lambda):= \frac{1}{\langle \lambda \rangle^b} \big(  \sum_m \langle \lambda \rangle^{2b} \langle m \rangle^{2s +2 \sigma} |\widehat{w}(m, \lambda + n^2)|^2 \big)^{1/2}\,.
 $$
 Note that
 \begin{equation}\label{fjdksldkjngklsdkgnm}
 \sum_{n} \langle n \rangle^{2s + 2 \sigma} |b_{\lambda}(n)|^2   = 1
 \end{equation}
 and (by Cauchy-Schwarz) for $b>1/2$, 
 \begin{equation}\label{fmdksldkjngjnjskjgn}
 \int |\phi(\lambda)| d \lambda \lesssim \| w \|_{X^{s+ \sigma, b}} \, .
 \end{equation}
We are now ready to handle the remaining contributions to \eqref{mfkdlskdjgkldskjg}, meaning when the random part of the solution decomposition acts as  largest frequency. The argument will be again independent of which one is the largest frequency and on their order, thus we assume without loss of generality that $N_1 \geq N_2 \geq N_3$ and $J_1 = R$.
We introduce a slightly larger cutoff 
\begin{align}\label{def:kappa}
\kappa(t) := \eta(N_1^{-\varepsilon} t) , \qquad \varepsilon >0.
\end{align}
Recalling that $\eta = 1$ in $[0,1]$ and is compactly supported and taking $N_1$ large enough 
we have  $\kappa \eta = \eta$. Thus by standard properties of the restriction norm we can replace $\eta$ by $\kappa$
when we estimate the remaining contributions to \eqref{mfkdlskdjgkldskjg}.
\\
 
For $f \in \Omega_{\delta}^{\complement}$ we define
$$
g_{n} :=  \widehat{f}(n) (1 + |n|^{2s+1})^{\frac12}, \qquad n \in \Z.
$$
Note that, by the definition \eqref{DEf:OmegaEx}, we have 
 \begin{equation}\label{fmjdksldjgnjskdjngjdskjgn}
|g_{n} | \leq \delta^{- \frac13 \beta} \langle n \rangle^\varepsilon .
\end{equation} 
Case $J_2 = J_3 = R$. Letting $\mathcal{F}$ the space-time Fourier transform,
this contribution is bounded as 
\begin{align*}
 &\lesssim 
\sum_{\substack{N_1 \\ N_4 \lesssim N_1 \\ N_1 \geq N_2 \geq N_3 }} N_4^{s +  \sigma}   
\\
& \qquad 
\Bigg( \int \sum_{|n_4| \simeq N_4}  \frac{1_{\{\langle \tau - n_4^2 \rangle \leq N_1^C\} } }{\langle \tau - n_4^2 \rangle^{1 - 4 \gamma}}  
   \Bigg| \mathcal{F} \Big( \kappa(t-T) Z(x,t)
   \Big)(n_4,\tau) \Bigg|^2
d \tau
\Bigg)^{1/2} 
\\
\end{align*}
where
$$
Z(x,t) := 
\sum_{\substack{|n_j| \simeq N_j 
    \\ n_1, n_3 \neq n_2 }}   \frac{g_{n_1} e^{i n_1 x + i n_1^2 (t-T) }}{\langle n_1 \rangle^{s + 1/2} } 
  \frac{\overline{g_{n_2} } e^{- i n_2 x -  i n_2^2 (t-T) }}{\langle n_2 \rangle^{s + 1/2} }
   \frac{g_{n_3} e^{i n_3 x + i n_3^2 (t-T) }}{\langle n_3 \rangle^{s + 1/2} } 
$$
The presence of the cut-off  $1_{\{\langle \tau - n_4^2 \rangle \leq N_1^C\}}$ is justified by the fact that the 
complementary part of the integration region (very far from 
the paraboloid) gives an acceptable contribution as long as the constant $C$ is sufficiently large. Computing the space-time Fourier transform we get (here the hat is the Fourier transform in time)
\begin{align*}
&=
\sum_{\substack{N_1 \\ N_4 \lesssim N_1 \\ N_1 \geq N_2 \geq N_3 }} N_4^{s +  \sigma}   
  \Bigg( \int \sum_{|n_4| \simeq N_4}  \frac{1_{\{\langle \tau - n_4^2 \rangle \leq N_1^C\} }}{\langle \tau - n_4^2 \rangle^{1 - 4 \gamma}}  
\\
& \qquad 
   \Bigg| \!\!\!\! \sum_{\substack{|n_j| \simeq N_j \\ n_1 - n_2 + n_3 = n_4 
    \\ n_1, n_3 \neq n_2 }} \!\!\!\!
\widehat{\kappa}(\tau - n_1^2  + n_2^2 - n_3^2 )\frac{g_{n_1}}{\langle n_1 \rangle^{s + 1/2} }
\frac{\overline{g_{n_2} } }{\langle n_2 \rangle^{s + 1/2} }
   \frac{g_{n_3}}{\langle n_3 \rangle^{s + 1/2} }  \Bigg|^2
d \tau \Bigg)^{1/2}.
\end{align*}
Case $J_2 = D, J_3 = R$. This contribution is bounded as
\begin{align*}
 &\lesssim 
\sum_{\substack{N_1 \\ N_4 \lesssim N_1 \\ N_1 \geq N_2 \geq N_3 }} N_4^{s +  \sigma}   
\int  \phi(\lambda_2) \Bigg( \int \sum_{|n_4| \simeq N_4}  \frac{1_{\{\langle \tau - n_4^2 \rangle \leq N_1^C\} } }{\langle \tau - n_4^2 \rangle^{1 - 4 \gamma}}   
   \Bigg| \mathcal{F} \Big( 
\kappa(t - T) Z(x,t)
   \Big)(n_4,\tau) \Bigg|^2
d \tau
\Bigg)^{1/2} d \lambda_2,
\end{align*}
where 
$$
Z(x,t) :=  \sum_{\substack{|n_j| \simeq N_j 
    \\ n_1, n_3 \neq n_2 }}  \frac{g_{n_1} e^{i n_1 x + i n_1^2 (t-T) }  }{\langle n_1 \rangle^{s + 1/2} } 
  \overline{b_{\lambda_2}(n_2)} e^{- i n_2 x - i (n_2^2+\lambda_2) (t-T) } 
   \frac{g_{n_3} e^{i n_3 x + i n_3^2 (t-T) }}{\langle n_3 \rangle^{s + 1/2} } .
$$
Using the Minkowski integral inequality to switch the $\ell^2_n L^{2}_{\tau}$ norm and the integration over 
$d \lambda_2 $ and then  
\eqref{fmdksldkjngjnjskjgn}
we bound this by
\begin{align*}
\\
&\lesssim 
 \| w_2 \|_{X^{s + \sigma, \frac12 + \gamma}}
\sum_{\substack{N_1 \\ N_4 \lesssim N_1 \\ N_1 \geq N_2 \geq N_3 }} N_4^{s +  \sigma}   
  \Bigg( \int \sum_{|n_4| \simeq N_4}  \frac{1_{\{\langle \tau - n_4^2 \rangle \leq N_1^C\} }}{\langle \tau - n_4^2 \rangle^{1 - 4 \gamma}}  
\\
& \qquad 
 \sup_{\lambda_2}  \Bigg| \!\!\!\! \sum_{\substack{|n_j| \simeq N_j \\ n_1 - n_2 + n_3 = n_4 
   \\ n_1, n_3 \neq n_2 }} \!\!\!\! \widehat{\kappa}(t-n_1^2  + n_2^2 + \lambda_2 - n_3^2)
\frac{g_{n_1} }{\langle n_1 \rangle^{s + 1/2} } \overline{b_{\lambda_2}(n_2)} 
 \frac{g_{n_3} }{\langle n_3 \rangle^{s + 1/2} }  \Bigg|^2
d \tau \Bigg)^{1/2}.
\end{align*}
Case $J_2 = R, J_3 = D$. Similarly, this contribution is bounded as
\begin{align*}
 &\lesssim 
\sum_{\substack{N_1 \\ N_4 \lesssim N_1 \\ N_1 \geq N_2 \geq N_3 }} N_4^{s +  \sigma}   
\int  \phi(\lambda_3) \Big( \int \sum_{|n_4| \simeq N_4}  \frac{1_{\{\langle \tau - n_4^2 \rangle \leq N_1^C\} } }{\langle \tau - n_4^2 \rangle^{1 - 4 \gamma}}   
   \Big| \mathcal{F} \Big(  \kappa(t-T)
Z(x,t)
   \Big)(n_4,\tau) \Big|^2
d \tau
\Big)^{1/2} d \lambda_3
\end{align*}
where
$$
Z(x,t) : = \sum_{\substack{|n_j| \simeq N_j  
    \\ n_1, n_3 \neq n_2 }}  \frac{g_{n_1} e^{i n_1 x + i n_1^2 (t-T) }}{\langle n_1 \rangle^{s + 1/2} } 
  \frac{ \overline{ g_{n_2}} e^{- i n_2 x - i n_2^2 (t-T) }}{\langle n_2 \rangle^{s + 1/2} } 
  b_{\lambda_3}(n_3) e^{i n_3 x + i (n_3^2+\lambda_3) (t-T) } 
$$
and then 
\begin{align*}
&\lesssim 
 \| w_3 \|_{X^{s + \sigma, \frac12 + \gamma}}
\sum_{\substack{N_1 \\ N_4 \lesssim N_1 \\ N_1 \geq N_2 \geq N_3 }} N_4^{s +  \sigma}   
  \Bigg( \int \sum_{|n_4| \simeq N_4}  \frac{1_{\{\langle \tau - n_4^2 \rangle \leq N_1^C\} }}{\langle \tau - n_4^2 \rangle^{1 - 4 \gamma}}  
\\
& \qquad 
 \sup_{\lambda_3}  \Bigg| \!\!\!\! \sum_{\substack{|n_j| \simeq N_j \\ n_1 - n_2 + n_3 = n_4 
   \\ n_1, n_3 \neq n_2 }} \!\!\!\! \widehat{\kappa}( \tau -n_1^2  + n_2^2  - n_3^2 - \lambda_3  )
\frac{g_{n_1} }{\langle n_1 \rangle^{s + 1/2} }  
 \frac{\overline{g_{n_2}} }{\langle n_2 \rangle^{s + 1/2} } b_{\lambda_3}(n_3)  \Bigg|^2
d \tau \Bigg)^{1/2} .
\end{align*}
Case $J_2 = D, J_3 = D$. This contribution is bounded as 
\begin{multline*}
 \lesssim 
\sum_{\substack{N_1 \\ N_4 \lesssim N_1 \\ N_1 \geq N_2 \geq N_3 }} N_4^{s +  \sigma}   
\int \phi(\lambda_2) \phi(\lambda_3)
\\
 \Big( \int \sum_{|n_4| \simeq N_4}  \frac{1_{\{\langle \tau - n_4^2 \rangle \leq N_1^C\} } }{\langle \tau - n_4^2 \rangle^{1 - 4 \gamma}}  
  \Big| \mathcal{F} \Big( \kappa(t-T) Z(x,t) \Big)(n_4,\tau) \Big|^2
d \tau
\Big)^{1/2} d \lambda_2 d \lambda_3
\end{multline*}
where 
$$
Z(x,t) := \sum_{\substack{|n_j| \simeq N_j  
    \\ n_1, n_3 \neq n_2 }}  \frac{g_{n_1} }{\langle n_1 \rangle^{s + 1/2} } 
     \overline{b_{\lambda_2}(n_2)} e^{- i n_2 x  - i (n_2^2+\lambda_2) (t-T) }
     b_{\lambda_3}(n_3) e^{i n_3 x + i (n_3^2+\lambda_3) (t-T) }
$$
and then 
\begin{align*}
&\lesssim 
 \| w_2 \|_{X^{s + \sigma, \frac12 + \gamma}} 
  \| w_3 \|_{X^{s + \sigma, \frac12 + \gamma}} \sum_{\substack{N_1 \\ N_4 \lesssim N_1 \\ N_1 \geq N_2 \geq N_3 }} N_4^{s +  \sigma}   
  \Big( \int \sum_{|n_4| \simeq N_4}  \frac{1_{\{\langle \tau - n_4^2 \rangle \leq N_1^C\} }}{\langle \tau - n_4^2 \rangle^{1 - 4 \gamma}}  
\\
& \qquad 
  \sup_{\lambda_2, \lambda_3} \Big| \!\!\!\!\! \sum_{\substack{|n_j| \simeq N_j \\ n_1 - n_2 + n_3 = n_4 
    \\ n_1, n_3 \neq n_2 }} \!\!\!\!\!\! \widehat{\kappa}(\tau - n_1^2  + n_2^2 +  \lambda_2 - n_3^2 -  \lambda_3)
\frac{g_{n_1}}{\langle n_1 \rangle^{s + 1/2} } \overline{b_{\lambda_2}(n_2)} b_{\lambda_3}(n_3)  \Big|^2
d \tau \Big)^{1/2},
\end{align*}
here we used the Minkowski integral inequality to switch the $\ell^2_n L^{2}_{\tau}$ norm and the integration over 
$d \lambda_2 d\lambda_3 $ then  \eqref{fmdksldkjngjnjskjgn}. 
\\

Thus, if we  define, for $j = 1, 2 ,3$
 \begin{equation}\label{Def:Clambda}
 c_{\lambda_j, J_j}(n_j) := \left\{ 
 \begin{array}{ll}
 b_{\lambda_j}(n_j)  
 & \mbox{if $J_j = D$}   \\ & \\
\frac{1}{\langle n_j \rangle^{s + 1/2} } g_{n_j}    & \mbox{if $J_j = R$} ,
 \end{array}
 \right.
 \end{equation}
\begin{equation*}
\ell_j  =
\left\{ 
 \begin{array}{ll}
 1 
 & \mbox{if $J_j = D$}   \\ & \\
0    & \mbox{if $J_j = R$} 
 \end{array}
 \right.
 \end{equation*} 
\begin{equation*}
\mu_j  =
\left\{ 
 \begin{array}{ll}
 \| w_j \|_{X^{s + \sigma, \frac12 + \gamma}}
 & \mbox{if $J_j = D$}   \\ & \\
1   & \mbox{if $J_j = R$} 
 \end{array}
 \right.
 \end{equation*} 
we can rewrite the  four previous estimates compactly as
\begin{align}\label{kdlsdkknjfnjnjn}
&\lesssim 
\mu_2 \mu_3 \sum_{\substack{N_1 \\ N_4 \lesssim N_1 \\ N_1 \geq N_2 \geq N_3 }} N_4^{s +  \sigma}   
  \Big( \int \sum_{|n_4| \simeq N_4}  \frac{1_{\{\langle \tau - n_4^2 \rangle \leq N_1^C\} }}{\langle \tau - n_4^2 \rangle^{1 - 4 \gamma}}  \sup_{\lambda_2, \lambda_3} 
\\
& \nonumber \qquad 
  \Big|\!\!\!\!\!\!\!\! \sum_{\substack{|n_j| \simeq N_j \\ n_1 - n_2 + n_3 = n_4 
     \\ n_1, n_3 \neq n_2 }} \!\!\!\!\!\!\!\!
\widehat{\kappa}(\tau-n_1^2  + n_2^2 + \ell_2 \lambda_2 - n_3^2 - \ell_3 \lambda_3)) 
\frac{g_{n_1}}{\langle n_1 \rangle^{s + 1/2} } \overline{c_{\lambda_2, J_2}(n_2)} c_{\lambda_3, J_3}(n_3)  \Big|^2
d \tau \Big)^{1/2} .
\end{align}
Then, changing 
variables $\tau' = \tau - n_4^2$ and using that, if we restrict to 
$n_1 - n_2 + n_3 = n_4 $ we have
$$
n_1^2  - n_2^2  + n_3^2  - n_4^2 =   - 2 (n_1 - n_2)(n_3 - n_2), 
$$ 
this is further bounded by
\begin{align*}
&\lesssim \mu_2 \mu_3
\sum_{\substack{N_1 \\ N_4 \lesssim N_1 \\ N_1 \geq N_2 \geq N_3 }}  N_4^{s +  \sigma}  
\Bigg( \int \sum_{|n_4| \simeq N_4}  \frac{1_{\{\langle \tau' \rangle \leq N_1^C\} }}{\langle \tau' \rangle^{1 - 4 \gamma}}  
 \sup_{\lambda_2, \lambda_3}
\\
&
   \Bigg| \!\!\!\!\!\!  \sum_{\substack{|n_j| \simeq N_j \\ n_1 - n_2 + n_3 = n_4 
   \\ n_1, n_3 \neq n_2 }} \!\!\!\!\!\!\!\!\!
   \widehat{\kappa}(\tau' + 2 (n_1 - n_2)(n_3 - n_2) + \ell_2 \lambda_2 - \ell_3 \lambda_3)) 
\frac{g_{n_1}}{\langle n_1 \rangle^{s + 1/2} } \overline{c_{\lambda_2, J_2}(n_2)} c_{\lambda_3, J_3}(n_3)  \Bigg|^2
d \tau' \Bigg)^{1/2} .
\end{align*}
By H\"older inequality in $d \tau'$ we then get
\begin{align*}
& \lesssim \mu_2 \mu_3 
\sum_{\substack{N_1 \\ N_4 \lesssim N_1 \\ N_1 \geq N_2 \geq N_3 }} 
 N_4^{s +  \sigma} N_1^{4C \gamma}   
 \Bigg(  \sum_{|n_4| \simeq N_4} \sup_{\tau', \lambda_2, \lambda_3 }  
\\ &
   \Bigg| \!\!\!\!\!\! \sum_{\substack{|n_j| \simeq N_j \\ n_1 - n_2 + n_3 = n_4 
  \\ n_1, n_3 \neq n_2 }} \!\!\!\!\!\!\!\!\!
   \widehat{\kappa}(\tau' + 2 (n_1 - n_2)(n_3 - n_2) + \ell_2 \lambda_2 - \ell_3 \lambda_3)) 
\frac{g_{n_1}}{\langle n_1 \rangle^{s + 1/2} } \overline{c_{\lambda_2, J_2}(n_2)} c_{\lambda_3, J_3}(n_3)  \Bigg|^2 
\Bigg)^{1/2} .
\end{align*}
Recalling the definition \eqref{def:kappa} of $\kappa$ we see that the $n_1, n_2, n_3$ that satisfy 
$$2 (n_1 - n_2)(n_3 - n_2) \notin  - \tau' - \ell_2 \lambda_2 + \ell_3 \lambda_3  + (-\frac12 , \frac12)$$ 
give an acceptable contribution (namely $\lesssim  \mu_2\mu_3 \delta^{-\frac{\beta}{3} }    a(J_2)  a(J_3)$). 
Thus we can restrict the sum to the $n_1, n_2, n_3$ such that  
$$
2 (n_1 - n_2)(n_3 - n_2) \in  - \tau' - \ell_2 \lambda_2 + \ell_3 \lambda_3  + (-\frac12 , \frac12),
$$
defining the sum as zero if no triple $n_1, n_2, n_3$ is contained in this interval. 
For $\tau', \lambda_1, \lambda_2$ given, this interval contains  only one integer that we denote with
$\mu=\mu (\tau', \lambda_1, \lambda_2)$. 
Thus,
 letting 
\begin{align*}
\Omega  := & \sum_{\substack{N_1 \\ N_4 \lesssim N_1 \\ N_1 \geq N_2 \geq N_3 }}    N_4^{s +  \sigma } N_1^{4C \gamma + \varepsilon}   
\\ 
&
\Bigg( \sum_{|n_4| \simeq N_4} \sup_{\tau', \lambda_2, \lambda_3 }  
   \Bigg| \!\!\!\!\!\!\!\!\!\!\!\!\!\!\!\!\!\!\!\!\!\!\!\! \sum_{\substack{|n_j| \simeq N_j \\ n_1 - n_2 + n_3 = n_4 
   \\  (n_1 - n_2)(n_3 - n_2) = \frac12 \mu(\tau', \lambda_2, \lambda_3) \\ n_1, n_3 \neq n_2 }} \!\!\!\!\!\!\!\!\!\!\!\!\!\!\!\!\!\!\!\!\!\!\!\!
\frac{g_{n_1}}{\langle n_1 \rangle^{s + 1/2} } \overline{c_{\lambda_2, J_2}(n_2)} c_{\lambda_3, J_3}(n_3)  \Bigg|^2
 \Bigg)^{1/2}
\end{align*}
we have reduced to prove
$$
\Omega \lesssim  \delta^{-\frac{\beta}{3} }   \mu_2\mu_3  a(J_2)  a(J_3) .
$$
 Letting 
 $$ A_{n_4} := \{ (n_1, n_2, n_3) : |n_j| \simeq N_j, \,\,  n_1 - n_2 + n_3 = n_4, \,\, 2(n_1 - n_2)(n_3 - n_2) = \mu, \,\, n_1, n_3 \neq n_2 \}, $$
 \begin{equation} A := \{ (n_1, n_2, n_3) : |n_j| \simeq N_j,  \,\, 2(n_1 - n_2)(n_3 - n_2) = \mu, \,\, n_1, n_3 \neq n_2 \}, 
 \end{equation}
 we have 
 \begin{equation}\label{mfklsdknjgfksl}
 \bigcup_{n_4} A_{n_4} = A, \qquad A_{n_4} \cap A_{(n_{4})'} = \emptyset \quad  \forall \, n_4 \neq (n_4)' . 
 \end{equation}
 Recall that by the divisor bound and $N_4 \lesssim N_1$ one has 
 \begin{equation}\label{DivBound1}
 \# A_{n_4} \lesssim N_{1}^\varepsilon ,
 \end{equation}
 \begin{equation}\label{DivBound2}
\# A \lesssim N_{1}^\varepsilon N_{3}. \end{equation}
 Thus using Cauchy-Schwarz and \eqref{DivBound1} we get
\begin{align*}
& \Big|  \sum_{\substack{|n_j| \simeq N_j \\ n_1 - n_2 + n_3 = n_4 \\ (n_1 - n_2)(n_3 - n_2) = \frac12 \mu(\tau', \lambda_2, \lambda_3) \\ n_1, n_3 \neq n_2 }}
 \frac{g_{n_1}}{\langle n_1 \rangle^{s + 1/2} }  c_{\lambda_2, J_2}(n_2)  c_{\lambda_3, J_3}(n_3)  \Big|
 \\
 & \leq
  \sum_{(n_1, n_2, n_3) \in A_{n_4}}
\frac{|g_{n_1}|}{\langle n_1 \rangle^{s + 1/2} } | c_{\lambda_2, J_2}(n_2)| | c_{\lambda_3, J_3}(n_3)| 
\\ \nonumber 
 & \lesssim 
N_1^{\frac{\varepsilon}{2}} \Big( \sum_{(n_1, n_2, n_3) \in A_{n_4}}
\frac{|g_{n_1}|^2}{\langle n_1 \rangle^{2s + 1} } | c_{\lambda_2, J_2}(n_2)|^2 | c_{\lambda_3, J_3}(n_3)|^2  \Big)^{1/2} .
\end{align*}

We restricted to $\sigma \in (1/4, 1/2)$, thus in particular we have 
$s + \sigma < s + \frac{1}{2}  - \varepsilon$, taking $\varepsilon>0$ sufficiently small. Recalling 
 \eqref{fjdksldkjngklsdkgnm}, \ref{fmjdksldjgnjskdjngjdskjgn}, \eqref{Def:Clambda} 
 we thus have 
 $$
 | c_{\lambda_j, J_j}(n_j)| \leq  \frac{a(J_j)}{\langle n_j \rangle^{s + \sigma}},
 $$ 
 for all $f$ outside $\Omega_\delta$, 
 that plugged in the above estimate, together with \eqref{fmjdksldjgnjskdjngjdskjgn}, gives
\begin{align*}
& \Big|  \sum_{\substack{|n_j| \simeq N_j \\ n_1 - n_2 + n_3 = n_4 \\ (n_1 - n_2)(n_3 - n_2) = \frac12 \mu(\tau', \lambda_2, \lambda_3) \\ n_1, n_3 \neq n_2 }}
 \frac{g_{n_1}}{\langle n_1 \rangle^{s + 1/2} }  c_{\lambda_2, J_2}(n_2)  c_{\lambda_3, J_3}(n_3)  \Big|
 \\
 &\lesssim 
N_1^{\frac{3\varepsilon}{2}}  \delta^{-\frac13 \beta} \mu_1 \mu_2 a(J_2)   a(J_3)  \Big( \sum_{(n_1, n_2, n_3) \in A_{n_4}}
\frac{1}{\langle n_1 \rangle^{2s + 1} } \frac{1}{\langle n_2 \rangle^{2s + 2\sigma} }
\frac{1}{\langle n_3 \rangle^{2s + 2\sigma} }   \Big)^{1/2} .
\end{align*}
Recalling the definition of $\Omega$ and using this bound we arrive to
\begin{align*}
\Omega  \lesssim & \sum_{\substack{N_1 \\ N_4 \lesssim N_1 \\ N_1 \geq N_2 \geq N_3 }}    
N_4^{s +  \sigma } N_1^{4C \gamma + \frac52 \varepsilon}  
  \delta^{-\frac13 \beta}  a(J_2)   a(J_3) 
\\ 
&
\Bigg( \sum_{|n_4| \simeq N_4}  
 \Bigg( \sum_{(n_1, n_2, n_3) \in A_{n_4}}
\frac{1}{\langle n_1 \rangle^{2s + 1} } \frac{1}{\langle n_2 \rangle^{2s + 2\sigma} }
\frac{1}{\langle n_3 \rangle^{2s + 2\sigma} }  \Bigg)
 \Bigg)^{1/2}\,.
\end{align*}
 We first consider the case $N_4 \simeq N_1$. In this case
we can estimate:
 \begin{align*} 
 \Omega 
 & \lesssim  \delta^{-\frac13 \beta}  \mu_1  \mu_2 a(J_2)   a(J_3)
\\ 
&
\qquad
\sum_{\substack{N_1 \\ N_4 \simeq N_1 \\ N_1 \geq N_2 \geq N_3 }}   
N_4^{s +  \sigma } N_1^{4C \gamma + \frac52 \varepsilon}   \left( \sum_{|n_4| \simeq N_4}  
      \sum_{(n_1, n_2, n_3) \in A_{n_4}}
\frac{1}{\langle n_1 \rangle^{2s + 1} } \frac{1}{\langle n_2 \rangle^{2s + 2\sigma} }
\frac{1}{\langle n_3 \rangle^{2s + 2\sigma} }  \right)^{\frac12}  
\\ \nonumber
& \lesssim 
  \delta^{-\frac13 \beta}  \mu_1  \mu_2  a(J_2)   a(J_3) 
  \\ 
&
\qquad
 \sum_{\substack{N_1 \\ N_4 \simeq N_1 \\ N_1 \geq N_2 \geq N_3 }}  
N_4^{s +  \sigma } N_1^{4C \gamma + \frac52 \varepsilon}   \left(     \sum_{(n_1, n_2, n_3) \in A}
N_1^{-2s -1} N_2^{-2s - 2 \sigma} N_3^{-2s - 2 \sigma} \right)^{\frac12}  
\\ \nonumber
& \lesssim 
  \delta^{-\frac13 \beta}  \mu_1  \mu_2  a(J_2)   a(J_3)  \sum_{\substack{N_1 \\ N_4 \simeq N_1 \\ N_1 \geq N_2 \geq N_3 }}  N_4^{ \sigma - \frac{1}{2}  } 
N_1^{4C \gamma + \frac52 \varepsilon}  N_2^{-s -\sigma } N_3^{-s - \sigma}
 \left(     \sum_{(n_1, n_2, n_3) \in A} 1 \right)^{\frac12}
\\ \nonumber
& \lesssim 
  \delta^{-\frac13 \beta}   \mu_1  \mu_2  a(J_2)   a(J_3) 
  \sum_{\substack{N_4, N_1, N_2, N_3 \\ N_4 \simeq N_1 \\ N_1 \geq N_2 \geq N_3 }}  N_4^{ \sigma - \frac{1}{2} } 
  N_1^{ 4C \gamma + 3 \varepsilon} N_2^{-s -\sigma } N_3^{-s - \sigma + \frac12}\,.
 \end{align*}
Note that in the second inequality we used \eqref{mfklsdknjgfksl}, in the third inequality 
we used $N_1 \gtrsim N_4$ and the last inequality we used  \eqref{DivBound2}. 
\\

Using again $N_4 \simeq N_1$, the sum above is bounded by  
$$
\sum_{N_4, N_1, N_2, N_3}  N_4^{ \sigma - \frac{1}{2} + 4C \gamma+ 4 \varepsilon} N_1^{-\varepsilon} N_2^{-s -\sigma } N_3^{-s - \sigma + \frac12}.
$$
We claim that this is convergent, for all $s >0$ and $\sigma \in (1/4, 1/2)$, taking  $\gamma$ and $\varepsilon$ sufficiently small.
Indeed, using $N_2 \geq N_3  $ we estimate it
with
$$
\leq \sum_{N_4, N_1, N_2, N_3}  N_4^{ \sigma - \frac{1}{2} + 4C \gamma+ 3 \varepsilon} N_1^{-\varepsilon} N_2^{-s -\sigma + \frac14} N_3^{-s - \sigma + \frac14}
$$
and we note that the exponents are all negative as long as we take $\gamma, \varepsilon >0$ sufficiently small.
\\

We now consider the (only other possible) case $N_{4} \ll N_1$. This forces $N_{2} \simeq N_1$. Thus proceeding 
as above we reduce to prove the convergence of 
$$
  \sum_{\substack{N_4, N_1, N_2, N_3 \\ N_2 \simeq N_1 \\ N_2, N_3 \leq N_1 }}  N_4^{ \sigma - \frac{1}{2} } N_1^{3\varepsilon + 4C \gamma} N_2^{-s -\sigma } N_3^{-s - \sigma + \frac12} .
$$
Using $N_{2} \simeq {N_1} \geq N_3$ the sum is bounded by 
$$
  \sum_{N_4, N_1, N_2, N_3 }  N_4^{ \sigma - \frac{1}{2} } N_1^{- \varepsilon } 
  N_2^{-s -\sigma + 4\varepsilon + 4C \gamma +  \frac14 } N_3^{-s - \sigma + \frac14}
$$
that converges for all $s >0$, $\sigma \in (1/4, 1/2)$, as long as we take and $\varepsilon, \gamma >0$ sufficiently small.
\\

Thus, in order to conclude the proof, it remains to prove \eqref{fmdksldkngkslBis} for $j=2$. Proceeding as above we reduce to show the estimate  
$$
 \sum_{N_4} N_4^{2s + 2 \sigma + c \gamma}  \sum_{n \simeq N_4} 
|c_{\lambda, J_1}(n)|^2 | c_{\lambda_2, J_2}(n)|^2 | c_{\lambda_3, J_3}(n)|^2  \lesssim 
\mu_1 \mu_2 \mu_3 a(J_1)  a(J_2)  a(J_3) ,
$$
that follows by
\begin{align*}
& \sum_{N_4} N_4^{2s + 2 \sigma + c \gamma}  \sum_{n \simeq N_4} 
|c_{\lambda, J_1}(n)|^2 | c_{\lambda_2, J_2}(n)|^2 | c_{\lambda_3, J_3}(n)|^2 
\\ \nonumber
&
\qquad \qquad\qquad\qquad \lesssim_{\varepsilon} 
\mu_1 \mu_2 \mu_3 a(J_1) a(J_2)  a( J_3) 
\sum_{N_4} N_4^{-4s - 4 \sigma + c \gamma + \varepsilon} , \qquad \varepsilon >0,
\end{align*}
where, recalling the definition \eqref{Def:Clambda}, we have 
used \eqref{fjdksldkjngklsdkgnm} and \eqref{fmjdksldjgnjskdjngjdskjgn}. 
Taking $\gamma >0$ and $\varepsilon >0$ sufficiently small this is indeed finite for all 
$s >0$ and $\sigma >0$. The proof is then concluded.   
\end{proof}
%
In the following we will prove some additional results on the decomposition of the solution~$w$,  
which proof is similar to that of Proposition \ref{1/2SmoothingNLSPert}.
We recall that    
$$
B^{r}(R) := \left\{ f: \| f \|_{H^{r}} \leq R \right\}.
$$
\begin{prop}\label{1/2SmoothingNLSPertTris}
Let $s \in (0,1)$, $ 0 < s' <s$, $\sigma \in (1/4,1/2)$ and $T\geq 1$. 
There exist $\gamma > \beta >0$,  $\bar{\delta} \in (0,1)$ 
and $\bar{N} \geq 1$ such that the following holds.
Let $\delta \in (0, \bar{\delta})$, $N \in 2^{\N}$, $N > \bar N$ and $\rho \in [0,1]$.
Let $w(T) \in \Omega_{\delta}^\complement$
and $e_N \in P_{> N} \Omega_{\delta}^\complement$. Let
$$
w'(T) \in w(T) + e_N  +  B^{s + \sigma}(\rho) .
 $$
We have 
$$
w'(t)  = w(t) + e^{i(t-T) \Delta} e_N + r(t), \qquad t \in [T, T + \delta ], 
$$
with
\begin{equation}\label{DeltaBound}
\| r \|_{X^{s' + \sigma, \frac12 + \gamma}_{[T, T + \delta ]}} \leq C  ( \rho + N^{s'-s}),
\end{equation}
for some $C >0$.
Here
 $w$ and $w'$ are solutions of the (renormalized) cubic NLS \eqref{NLSvIntrogauged} with $N = \infty$ and with 
 datum $w(T)$ and $w'(T)$, respectively.
\end{prop}
\begin{proof}
By assumption
\begin{equation}\label{fndjskdjgjkdsjrngf}
w'(T) = w(T) +  e_N + h,
\end{equation}
where $e_N \in P_{> N} \Omega_{\delta}^{\complement}$ and $h \in B^{s + \sigma}(\rho)$. 
Let $t \in [T, T + \delta]$. By Proposition \ref{1/2SmoothingNLSPert} we have 
\begin{equation}\label{fndjsk1}
w(t) = e^{i (t - T)  \Delta} w(T) + \tilde{w}(t), \\ \quad  
\| \tilde{w} \|_{X^{s + \sigma, \frac12 + \gamma}_{[T, T + \delta]}} \leq 1.
\end{equation}
By Remark \ref{OmegaDeltaStability} we also have 
$e_N \in \Omega_{\delta}^{\complement}$. Thus again using Proposition \ref{1/2SmoothingNLSPert} 
we have as well
\begin{equation}\label{fndjsk2}
w'(t) = e^{i (t - T)  \Delta} w'(T) + \tilde{w'}(t), \\ \quad  
\| \tilde{w'} \|_{X^{s + \sigma, \frac12 + \gamma}_{[T, T + \delta]}} \leq 1.
\end{equation}
We can represent the solutions via the Duhamel formula
\begin{align}\label{w'Eq}
w'(t) & = e^{i (t - T)  \Delta}  w'(T)  + \int_{0}^{t} e^{i (\tau - T)  \Delta} \mathcal{N}(w'(\tau)) \, d \tau  
\\ \nonumber
&= e^{i (t - T)  \Delta}  (w(T) +  e_N + h)  + \int_{0}^{t} e^{i (\tau - T)  \Delta} \mathcal{N}( w'(\tau)) \, d \tau,
\end{align}
\begin{equation}\label{wEq}
w(t)  = e^{i (t - T)  \Delta}  w(T)  + \int_{0}^{t} e^{i (\tau - T)  \Delta} \mathcal{N}(w(\tau)) \, d \tau.
\end{equation}
Letting 
\begin{equation}\label{Deltadsaegnjds}
r(t) := w'(t)  -  w(t) - e^{i(t-T) \Delta} e_N 
\end{equation}
it suffices to prove \eqref{DeltaBound}. 
Taking 
 the difference of \eqref{w'Eq}-\eqref{wEq} we obtain the following equation
for $r$:
\begin{equation}\label{mkdfsldfkngjdkls}
r(t)  = e^{i (t - T)  \Delta}   h  +
 \int_{0}^{t} e^{i (\tau - T)  \Delta} \big( \mathcal{N}(w'(\tau)) - \mathcal{N}(w(\tau)) \big) \, d \tau  .
\end{equation}
Looking at the nonlinearity as a trilinear operator (as in the proof of Proposition \ref{1/2SmoothingNLSPert}), and recalling \eqref{fndjskdjgjkdsjrngf}-\eqref{Deltadsaegnjds} we see that  
$\mathcal{N}(w', w', w') - \mathcal{N}(w, w,w)$ can be rewritten as a linear combination of terms of the form
$$
\mathcal{N}(f_1, f_2, f_3),
$$
where one of the $f_j$  belongs to the set  
$$
F:= \{ r, e^{i ((\cdot) - T)\Delta} e_N, \bar{r}, \overline{e^{i ((\cdot) - T)\Delta} e_N} \}
$$
while the other two $f_j$ belong to the set 
\begin{align*}
E
 =  \{ & e^{i ((\cdot) - T)}  e_N , e^{i ((\cdot) - T)} h ,  e^{i ((\cdot) - T)  \Delta} w(T),  \tilde{w'}(\cdot) , \tilde{w}(\cdot),  \\ 
 & \overline{e^{i ((\cdot) - T)}  e_N} , \overline{e^{i ((\cdot) - T)} h} , \overline{e^{i ((\cdot) - T)  \Delta} w(T)}, 
 \overline{\tilde{w'}(\cdot)} , \overline{\tilde{w}(\cdot)}
 \}.
\end{align*}
We recall $w(T) \in \Omega_{\delta}^{C}$. We also note $e^{i ((\cdot) - T)}  e_N \in P_{> N} \Omega_{\delta}^{\complement} $ (see Remark \ref{OmegaDeltaStability}) and
\begin{equation}\label{fdjalkdnfjdklsjkfgn}
\|e^{i ((\cdot) - T)} h \|_{X^{s + \sigma, \frac12 + \gamma}_{[T, T + \delta]}}  \lesssim  \| h \|_{H^{s + \sigma}} \lesssim \rho.
\end{equation}
Recalling also the estimates  \eqref{fndjsk1}-\eqref{fndjsk2} for $\tilde{w}$, $\tilde{w'}$,  we can use  the multilinear estimates~\eqref{fmdksldkngksl} to show that 
$$
\|  \int_{0}^{t} e^{i (\tau - T)  \Delta} \big( \mathcal{N}(w(\tau)) 
- \mathcal{N}(w'(\tau)) \big) \, d \tau \|_{X^{s' + \sigma, \frac12 + \gamma}_{[T, T + \delta]}} \leq 
\frac{1}{2} \| r \|_{X^{s' + \sigma, \frac12 + \gamma}_{[T, T + \delta]}} +  CN^{s'-s }.
$$
Combining with \eqref{mkdfsldfkngjdkls}-\eqref{fdjalkdnfjdklsjkfgn} we get
$$
\| r(t) \|_{X^{s' + \sigma, \frac12 + \gamma}_{[T, T + \delta]}}  \leq 2 \Big(
\| e^{i (t - T)  \Delta}  h \|_{X^{s' + \sigma, \frac12 + \gamma}_{[T, T + \delta]}} +
C N^{s'-s} \Big) 
\leq 2 \Big(
\rho +
C N^{s'-s} \Big)  
$$
from which we deduce \eqref{DeltaBound}.
\end{proof}
\begin{prop}\label{1/2SmoothingNLSPert2}
Let $s \in (0,1)$, $0< s' <s $, $\sigma \in (1/4, 1/2)$ and $T \geq 1$. 
There exist $\gamma > \beta >0$, $\bar{\delta} \in (0,1)$ and $\bar{N} \geq1$ such that the following holds. 
Let  $\delta \in (0, \bar{\delta})$ and $N \in 2^{\N}$, $N > \bar{N}$. 
Let $w(T) \in P_{\leq N} \Omega_{\delta}$. 
We then have
$$
w(t)  = w^N(t) + r(t), \qquad t \in [T, T + \delta ], 
$$
with 
$$
\| r  \|_{X^{s' + \sigma, \frac12 + \gamma}_{[T, T+ \delta]}} 
\leq   C N^{s' -s},
$$
for some $C>0$.
Here
 $w^N$ is a solution of the (renormalized, truncated) cubic NLS \eqref{NLSvIntrogauged} with 
 datum $w(T) = P_{\leq N} w(T)$
and $w = w^{\infty}$.
\end{prop}
\begin{proof}
We have $w(T) = P_{\leq N} w(T)$ (since $w(T) \in P_{\leq N} \Omega_{\delta}$). 
This also implies  
$w^N(t) = P_{\leq N} w^N (t)$. 
Using these facts, the Duhamel representations of $w$ and $w^N$ rewrite as
\begin{align*}
w(t)  
& 
= e^{i (t - T)  \Delta}  P_{\leq N} w(T)  
+   \tilde{w}
\\
&
= e^{i (t - T)  \Delta}  P_{\leq N} w(T)  
+   \int_{0}^{t} e^{i (\tau - T)  \Delta} \mathcal{N}(w(\tau)) \, d \tau , 
\end{align*}
\begin{align*}
w^{N}(t)  
& 
= e^{i (t - T)  \Delta}  P_{\leq N} w(T)  
+ \tilde{w}^N 
\\&
= e^{i (t - T)  \Delta}  P_{\leq N} w(T)  
+  P_{\leq N}  \int_{0}^{t} e^{i (\tau - T)  \Delta} \mathcal{N}( w^{N}(\tau)) \, d \tau  ,
\end{align*}
we get
\begin{align}\label{CombWT}
 w(t) - w^N(t)
& =   
 P_{\leq N}  \int_{0}^{t} e^{i (\tau - T)  \Delta} \big( \mathcal{N}(w(\tau))  -  \mathcal{N}(w^N(\tau)) \big)\, d \tau    
\\ \nonumber
& +   
 P_{> N}  \int_{0}^{t} e^{i (\tau - T)  \Delta} \mathcal{N}(w(\tau) ) \, d \tau.
\end{align}
Looking at the nonlinearity as a trilinear operator (as in the proof of Proposition \ref{1/2SmoothingNLSPert}) 
we see that  
$\mathcal{N}(w, w, w) - \mathcal{N}(w^N, w^N, w^N)$ can 
be rewritten as a linear combination of terms of the form
$$
\mathcal{N}(f_1, f_2, f_3),
$$
where one of the $f_j = w(\cdot) -  w^N(\cdot)$ or its conjugate,
while the other two $f_j$ belong to the set 
\begin{equation*}
F  :=  \{ e^{i ((\cdot) - T)} P_{\leq N} w(T), \tilde{w}, \tilde{w}^N, \overline{e^{i ((\cdot) - T)} P_{\leq N} w(T)}, 
\overline{\tilde{w}}, \overline{\tilde{w}^N }   \} .
\end{equation*}
Similarly,
we see that  
$\mathcal{N}(w, w, w)$ can 
be rewritten as a linear combination of terms of the form
$$
\mathcal{N}(f_1, f_2, f_3),
$$
where any of the $f_j$ belongs to 
 $$
G:= \{ e^{i ((\cdot) - T)  \Delta} P_{\leq N} w(T),  \tilde{w}(t), \overline{e^{i ((\cdot) - T)  \Delta} P_{\leq N} w(T)}, 
\overline{ \tilde{w}(t)} \}. 
$$    
Since $P_{\leq N}  w(T)  \in P_{\leq N} \Omega_{\delta}^{\complement} \subset \Omega_{\delta}^{\complement}$
we get by Remark \eqref{OmegaDeltaStability} and Proposition \ref{1/2SmoothingNLSPert} that the 
Duhamel contributions $\tilde{w}$ and $\tilde{w}^N$ satisfy the smoothing estimate  \eqref{BoundWNSmoothing}.
We can thus use the multilinear estimates~\eqref{fmdksldkngksl} to show that 
$$
\| P_{\leq N}  \int_{0}^{t} e^{i (\tau - T)  \Delta} \big( \mathcal{N}(w(\tau))  
-  \mathcal{N}(w^N(\tau)) \big)\, d \tau    \|_{X^{s' + \sigma, \frac12 + \gamma}_{[T, T + \delta]}} \leq 
\frac{1}{2} \| w(\cdot) -  w^N(\cdot) \|_{X^{s' + \sigma, \frac12 + \gamma}_{[T, T + \delta]}} +  CN^{s'-s }
$$
as well as
$$
\| P_{> N}  \int_{0}^{t} e^{i (\tau - T)  \Delta}  \mathcal{N}(w(\tau))  
 \, d \tau    \|_{X^{s' + \sigma, \frac12 + \gamma}_{[T, T + \delta]}} \leq 
  CN^{s'-s }.
$$
Combining these bounds with \eqref{CombWT} we get 
the desired statement.
\end{proof}
\subsection{Global in time analysis}
We will need the following large deviation estimate.   
\begin{lemma}\label{LemmaDelta}
Let $\delta >0$ and $0 \leq s' < s$. We have 
 $$
 \gamma_s \left( \big\{ f  : \sup_{t \in \R} \| e^{i t\Delta} f \|_{\mathcal C^{s'}(\T)}  > \delta^{-1}  \big\} \right) \lesssim e^{- C \delta^{ - 2 }} ,
  $$  
  where the constants here depend only on $s,s'$.
\end{lemma} 
\begin{proof} 
Let $N \in 2^{\N}$. 
We have
for all $p \geq 1$ and $\varepsilon >0$:
\begin{multline*}
\sup_{t \in \R} \| e^{i t\Delta} f \|_{\mathcal C^{s'}(\T)} = \sup_{t \in \T} \| e^{i t\Delta} f \|_{ \mathcal C^{s'}(\T)} 
\\
\lesssim \sum_N  N^{s'} \|  e^{i t\Delta} P_N f \|_{L^\infty_{x,t}}
 \lesssim \sum_{N \geq 1} N^{\frac3p + s' + \varepsilon} \|  e^{i t\Delta} P_N f \|_{L^p_{x,t}(\T^2)},
\end{multline*}
where we have used 
the Bernstein inequality and the fact that the space-time Fourier transform of 
$e^{i t\Delta} P_N f$ is supported in a spacetime rectangle of dimensions $N \times N^2$.
Letting 
$$
\sigma_N := c_0 N^{-\varepsilon},
$$ 
with $c_0 >0$ sufficiently small that $\sum_{N \geq 1} \sigma_N < 1$, we have
\begin{multline}
\label{fkalsdkjfalksd}
 \gamma_s \left(  \sum_{N \geq 1} \|  e^{i t\Delta} P_N f \|_{L^p_{x,t}} >  N^{-\frac3p-s' - \varepsilon} \delta^{-1} \right)
 \\
 \lesssim
 \sum_{N \geq 1} 
 \gamma_s \left(   \|  e^{i t\Delta} P_N f \|_{L^p_{x,t}} >  \sigma_N N^{-\frac3p-s' - \varepsilon } \delta^{-1} \right).
\end{multline}
On the other hand, by Gaussian hypercontractivity we have, for all $q \geq 1$
$$
\| e^{i t\Delta} P_N f \|_{L^{q}(\gamma_s)} \lesssim_{\varepsilon} \sqrt{q} N^{-s + \varepsilon}
$$
thus by Minkowski's inequality we obtain for all  $q \geq p\geq2$
the following bound (hereafter we omit the dependence on $\varepsilon$ of the constants) 
\begin{align*}
\left\|  \| P_N e^{it \Delta} f \|_{L^p_{x,t}(\T^2)} \right\|_{L^{q}(\gamma_s)}
 \leq \Bigl\|  \| P_N  e^{it \Delta} f \|_{L^q(\gamma_s)} \Bigr\|_{L^p_{x,t}(\T^2)} 
\lesssim \sqrt{q} N^{-s + \varepsilon} 
\end{align*}
that implies 
$$
 \gamma_s \left(  \|  e^{i t\Delta} P_N f \|_{L^p_{x,t}(\T^2)} >  \lambda \right) \lesssim e^{-C \lambda^2 N^{2s - 2 \varepsilon} }.
$$
Thus letting $\lambda = N^{-\frac3p-s' - \varepsilon} \delta^{-1}$ 
we can bound the right hand side of \eqref{fkalsdkjfalksd} as
\begin{align*}
& \lesssim \sum_{N \geq 1} \gamma_s \left(  \|  e^{i t\Delta} P_N f \|_{L^p_{x,t}(\T^2)} >  \sigma _N N^{-\frac3p-s' - \varepsilon} \delta^{-1} \right)
 \\ &
 \lesssim  \sum_{N \geq 1}  e^{-C  \delta^{-2} \sigma_N N^{2s - 4\varepsilon -\frac6p-2s'} } \lesssim
  \sum_{N \geq 1} 
 e^{-C c_0  \delta^{-2}  N^{2s  - 2s' - 5 \varepsilon - \frac6p } } \lesssim
 e^{-C c_0  \delta^{-2}  } ,
\end{align*}
where the last inequality holds as long as 
$$
2s - 2s' - 5 \varepsilon - \frac6p  >0,
$$
that for $s' < s$ is true as long as we take $p$ sufficiently large and $\varepsilon >0$ sufficiently small.
%
\end{proof}
We will now define the exceptional set $E$ outside which the statement of Theorem \ref{QuantControlCsBis} is valid. We will then prove that this set has measure zero. The definition involves two sequences of time scales that will be used in the proof of  
the upper bounds \eqref{GlobSmooth2Preq}-\eqref{GlobSmooth2PreqTWU} and \eqref{QuantControlCstrisFinalBound}, respectively.
\\

We first define the first sequence of time scales $T_j$, $j \in \N^*$ by 
\begin{equation}\label{TjDef}
 T_1 = 1, \qquad  T_{j+1} - T_j = c j^{-\varepsilon},
 \end{equation}
where $\varepsilon \in (0,1) $ and $c >0$. In particular 
$T_j \to \infty$ as $j \to \infty$ (in fact  $T_j\simeq j^{1-\varepsilon}$). 
Then, for any $m \in \N^*$  we define a family of refinements of the sequence obtained by dividing each interval
$[T_{j}, T_{j+1}]$ into $m$ equal subintervals, namely we set 
$$
T^m_{j, k} = T_{j} + \frac{k}{m} (T_{j+1} - T_j), \qquad k = 0, \ldots, m;
$$
note that we have, in particular  
$$T^m_{j, 0} = T_{j},  \qquad T^m_{j, m} = T_{j+1} $$
and
$$
T^{m}_{j, k} - T^{m}_{j, k-1}  = \frac{c}{j^{\varepsilon} m}, \qquad k=1, \ldots, m.
$$
We then choose $c>0$ sufficiently small that
 we can apply Proposition \ref{1/2SmoothingNLSPert} 
on each interval $[T^{m}_{j, k-1}, T^{m}_{j, k}]$.  
Let $s' < s$ and $s'' \in (\frac{s' + s}{2}, s)$. We define
  \begin{equation}\label{Def:BadSets}
  E_{j,m} := F_{j,m}  
  \cup \Omega_{\frac{c}{j^{\varepsilon} m}}, \qquad 
  F_{j,m} := \big\{ f  : \sup_{t \in \mathbb{R}} \| e^{i t\Delta} f \|_{\mathcal C^{s''}}  > m + T_j^{\varepsilon}  \big\}.
  \end{equation}
Now we define the second sequence of time scales. Let $m \in \N^*$. We consider a partition of the time interval $[1, m]$ into 
$\simeq m^2$ intervals 
$$[\mathcal{T}_{k-1}, \mathcal{T}_k] := \left[1 +  \frac{k-1}{m} ,  1 +  \frac{k}{m} \right] , \qquad k = 1, \ldots, m(m-1)$$
of length $1/m$. We define
\begin{equation}\label{FnDef}
E_m := \Omega_{\frac1m} \cup F_m, \qquad 
 F_{m} := \big\{ f  : \sup_{t \in \mathbb{R}} \| e^{i t\Delta} f \|_{\mathcal C^{s''}}  > m^{\varepsilon}  \big\}.
\end{equation}
The exceptional set is 
\begin{equation}\label{Def:GammaJ} 
E := \bigcap_{\ell \geq 1}  \bigcup_{M \geq 1} \bigcap_{N \geq M}
  \bigcup_{m \geq \ell} (K_m \cup L_m ), 
   \end{equation}
where
\begin{equation}\label{DEf:KmLm}
K_m :=    \bigcup_{j \geq 1} \bigcup_{k=1}^m
 (\tilde{\Phi}_{1, T^{m}_{j, k-1}}^N)^{-1} E_{j,m}, \qquad 
L_m :=  \bigcup_{k= 1 , \ldots, m(m-1)} 
 (\tilde{\Phi}_{1, \mathcal{T}_{k - 1}}^N)^{-1} E_m.
\end{equation}


\begin{lemma}\label{LemmaMesureE}
Let $s \in (0,1)$, $0 \leq s' < s$ and $\varepsilon \in (0,\frac12)$. We have $$\rho_s(E) =0, $$
where $E$ is the exceptional set defined in \eqref{Def:GammaJ}.  
\end{lemma}
\begin{proof}
We claim that for all $\varepsilon \in (0,\frac12)$ we have
\begin{equation}\label{fdjskdjhgfdsgm}
\rho_s (K_m \cup L_m) \lesssim_{\varepsilon, \beta} 
   e^{- C_{\beta}  m^{\min(\varepsilon, \beta)}} ,
\end{equation}
where $\beta >0$ is the small parameter from Proposition~\ref{1/2SmoothingNLSPert} (we assume here $\beta \in(0,1)$) and 
$C_{\beta} >0$. 
We will prove the inequality \eqref{fdjskdjhgfdsgm} in a second time. Let $\kappa = \min (\varepsilon, \beta)$. We have $\kappa >0$. From \eqref{fdjskdjhgfdsgm} we deduce 
   $$
\rho_s \left(   \bigcup_{  m \geq \ell} ( K_m \cup L_m) \right) \lesssim 
\sum_{m \geq \ell}  m e^{- C m^{\kappa }} 
\lesssim_{\varepsilon} 
\sum_{m \geq \ell}   e^{- \frac{C}{2} m^{\kappa} }  \lesssim_{\varepsilon}
e^{- \frac{C}{4} \ell^{\kappa}} .
  $$
Thus 
\begin{align*}
\rho_s  \left(   \bigcup_{M \geq 1} \bigcap_{N \geq M}
  \bigcup_{   m \geq \ell} ( K_m \cup L_m) \right)
&  \lesssim \liminf_{N \to \infty} 
 \rho_s \left(    \bigcup_{   m \geq \ell} ( K_m \cup L_m) \right)  
 \\ &
 \lesssim_{\varepsilon} 
 \liminf_{N \to \infty}  
e^{ - \frac{C}{4} \ell^{\kappa}} =  e^{ - \frac{C}{4} \ell^{\kappa}} .
\end{align*}  
Thus
\begin{align*}
\rho_s(E) & := \rho_s \left( \bigcap_{\ell \geq 1}  \bigcup_{M \geq 1} \bigcap_{N \geq M}   \bigcup_{   m \geq \ell} 
( K_m \cup L_m)  \right) 
 \\ &
\leq \inf_{\ell \geq 1} 
\rho_s \left(   \bigcup_{M \geq 1} \bigcap_{N \geq M}  \bigcup_{   m \geq \ell} 
( K_m \cup L_m)  \right) \leq \inf_{\ell \geq 1} C_{\varepsilon} e^{ - \frac{C}{4} \ell^{\kappa}}  = 0.
\end{align*}
%
We now prove the claim \eqref{fdjskdjhgfdsgm}
 We first show 
 $$
\rho_s (K_m) \lesssim_{\varepsilon} 
   e^{- C m^{\varepsilon}} .
$$
Recall that 
$ K_m := \bigcup_{j \geq 1} \bigcup_{k=1}^m
 (\tilde{\Phi}_{1, T^{m}_{j, k-1}}^N)^{-1} E_{j,m}$ and 
$$
 E_{j,m} := F_{j,m}  
  \cup \Omega_{\frac{c}{j^{\varepsilon} m}}, \qquad 
  F_{j,m} := \big\{ f  : \sup_{t \in \mathbb{R}} \| e^{i t\Delta} f \|_{\mathcal C^{s''}}  > m + T_j^{\varepsilon}  \big\}.
$$
 Using Lemma \ref{LemmaDelta} with $\delta = (m + T_j^{\varepsilon})^{-1}$ and at H\"older regularity $ s''\in (\frac{s' + s}{2}, s)$  we have
 \begin{equation}\label{djslkgjnjkdsjnjgksdg}
  \rho_s (F_{j,m}) \leq  \gamma_s (F_{j,m}) \lesssim e^{- C( m + T_j^{  \varepsilon})^2} \lesssim 
  e^{- C m^2} e^{ -C T_j^{ 2 \varepsilon}} \lesssim e^{- C m^2} e^{ -C j^{2-2\varepsilon}} 
  \end{equation}
where we have used 
   $T_j\simeq j^{1-\varepsilon}$.
  On the other hand, by 
 \eqref{MeasureOmegaExceptPert} we have (recall that $m, j \geq 1$, $c \in (0,1)$, $\beta >0$)
 \begin{equation}\label{djslkgjnjkdsjnjgksdg2}
 \rho_s( \Omega_{\frac{c}{j^{\varepsilon} m}}) \leq 
 \gamma_s (\Omega_{\frac{c}{j^{\varepsilon} m}}) \lesssim 
 e^{- c^{-\beta} j^{\varepsilon \beta} m^{\beta}}
  \lesssim e^{- \frac12 c^{-\beta} m^{\beta}} e^{- \frac12 j^{\varepsilon \beta}}.
\end{equation}
We thus get
  $$
  \rho_s (E_{j,m}) \leq \rho_s (F_{j,m}) +    \rho_s( \Omega_{\frac{c}{j^{\varepsilon} m}})  \lesssim
  e^{- C m^2} e^{ -C j^{2-2\varepsilon}} + e^{- \frac12  c^{-\beta} m^{\beta}} e^{- \frac12  j^{\varepsilon \beta}}.
  $$  
 Now for all $f \in E_{j,m}$ we have  
 \begin{equation}\label{fjdisldkjngioe}
  (\tilde{\Phi}_{1, T^{m}_{j, k-1}}^N)^{-1} f = \tilde{\Phi}_{T^{m}_{j, k-1}, 1}^N f = 
  \Phi^N_{T^{m}_{j, k-1}, 1}  e^{2i \mu_{N}(f) \ln T^{m}_{j, k-1}} f ,
 \end{equation}
 where in the second identity we used \eqref{ComparisonFlows}. On the other hand we have 
 $e^{i \alpha} E_{j,m} = E_{j,m}$ for all $\alpha \in \R$. This is indeed clear for the sets $F_{j,m}$ 
 from their definition and for the sets $\Omega_{\frac{c}{j^{\varepsilon} m}}$ has been observed in Remark \ref{OmegaDeltaStability}. Thus, if  $f \in E_{j,m}$ we have
$ e^{2i \mu_{N}(f) \ln T^{m}_{j, k-1}} f \in E_{j,m}$ as well. Plugging this into \eqref{fjdisldkjngioe} we see that
if $f \in E_{j,m}$ then
$$
  (\tilde{\Phi}_{1, T^{m}_{j, k-1}}^N)^{-1} f \in
  \Phi^N_{T^{m}_{j, k-1}, 1}   E_{j,m} = 
  ( \Phi^N_{1, T^{m}_{j, k-1}})^{-1}   E_{j,m},
 $$
 meaning
$$
  (\tilde{\Phi}_{1, T^{m}_{j, k-1}}^N)^{-1} E_{j,m}  \subseteq 
    ( \Phi^N_{1, T^{m}_{j, k-1}})^{-1}   E_{j,m}.
$$
On the other hand, exchanging the role of the flows in the previous argument we deduce 
$$
  (\Phi_{1, T^{m}_{j, k-1}}^N)^{-1} E_{j,m}  \subseteq 
    ( \tilde{\Phi}^N_{1, T^{m}_{j, k-1}})^{-1}   E_{j,m}
$$
so that 
$$
  (\Phi_{1, T^{m}_{j, k-1}}^N)^{-1} E_{j,m}  = 
    ( \tilde{\Phi}^N_{1, T^{m}_{j, k-1}})^{-1}   E_{j,m}.
$$
Thus,
using the (quantitative) quasi-invariance of the measure $\rho$ proved 
  in~\eqref{QuasiInvQuant} (with say 
  $\kappa=1/2$), we obtain 
  \begin{align}\label{HereInverseFlow}
\rho_s \left( ( \tilde{\Phi}_{1, T^{m}_{j, k-1}}^N)^{-1} E_{j,m} \right) 
& = \rho_s \left( ( \Phi_{1, T^{m}_{j, k-1}}^N)^{-1} E_{j,m} \right) 
\\ \nonumber 
&
\lesssim  
\left( \rho_s (  E_{j,m} ) \right)^{1/2} 
\lesssim   e^{- \frac12 C m^2} e^{ - \frac12 C j^{2-2\varepsilon}} + e^{- \frac14 c^{-\beta} m^{\beta}} e^{- \frac14 j^{\varepsilon \beta}}. 
  \end{align}
From this we deduce, for some $C, C_{\beta} >0$
$$
\rho_s (K_m) = \rho_s \left(    \bigcup_{j \geq 1} \bigcup_{k=1}^{m}
 (\tilde{\Phi}_{1, T^{m}_{j, k-1}}^N)^{-1} E_{j, m} \right) \lesssim 
 \sum_{j \geq 1} m e^{- 2 C_{\beta} m^{\beta}} e^{- C j^{\varepsilon \beta}}
\lesssim_{\varepsilon} 
   e^{- C_{\beta} m^{\beta}} \,.
$$
Now we show 
$$\rho_s (L_m)   \lesssim_{\varepsilon} 
  e^{-  \frac14 m^{\min(\beta, 2 \varepsilon)}} \,.
  $$
We recall that 
$$ L_m :=  \bigcup_{k= 1 , \ldots, m(m-1)} 
 (\tilde{\Phi}_{1, \mathcal{T}_{k - 1}}^N)^{-1} E_m
 $$
 and 
$$E_m := \Omega_{\frac1m} \cup F_m, \qquad 
 F_{m} := \big\{ f  : \sup_{t \in \mathbb{R}} \| e^{i t\Delta} f \|_{\mathcal C^{s''}}  > m^{\varepsilon}  \big\}.
$$
From \eqref{MeasureOmegaExceptPert} we obtain 
  \begin{equation}\label{HereInverseFlowbis}
 \rho_s (  \Omega_{\frac1m} )
\lesssim e^{-   m^\beta}.
  \end{equation}
 Using Lemma \ref{LemmaDelta} with $\delta = m^{-\varepsilon}$ and at H\"older regularity $s''  \in (\frac{s' + s}{2}, s)$  we have
 \begin{equation}
  \rho_s (F_{m}) \leq  \gamma_s (F_{m}) \lesssim e^{- m^{2\varepsilon} }.
  \end{equation}
Thus
$$
\rho_s  ( E_{m} ) \lesssim  e^{-  m^{\beta}} + e^{- m^{2\varepsilon} }.
$$
Exactly as before we can prove
$$
  (\Phi_{1, T_{k-1}}^N)^{-1} E_{m}  = 
    ( \tilde{\Phi}^N_{1, T_{k-1}})^{-1}   E_{m},
$$
thus using the quantitative quasi-invariance of the measure $\rho_s$ from~\eqref{QuasiInvQuant} with $\kappa =1/2$ we 
obtain
$$
\rho_s \left( \tilde{\Phi}_{1, \mathcal{T}_{k - 1}}^N)^{-1} E_m  \right) 
\lesssim 
\rho_s( E_m )^{1/2} \lesssim    
e^{-  \frac12 m^{\beta}} + e^{- \frac12 m^{2\varepsilon} }.
$$
Thus
\begin{align*}
\rho_s (L_m) & = \rho_s \left(    \bigcup_{k= 1 , \ldots, m(m-1)} 
 (\tilde{\Phi}_{1, T_{k - 1}}^N)^{-1} \Omega_{\frac1m} \right)
 \\ & 
  \lesssim  \sum_{k= 1 , \ldots, m(m-1)}   
 e^{-  \frac14 m^{\beta}} + e^{- \frac12 m^{2\varepsilon} }
  \lesssim 
  m^2 e^{-  \frac12 m^{\beta}} + m^2e^{- \frac12 m^{2\varepsilon} }
  \lesssim_{\varepsilon, \beta} 
  e^{-  \frac14 m^{\min(\beta, 2 \varepsilon)}} .
\end{align*} 
This completes the proof of the claim \eqref{LemmaMesureE} and thus of the lemma.
\end{proof}

It will be useful to the keep in mind that
\begin{equation}\label{fdnjskjhbhjjbhghgyhgfhdjs}
E^\complement = 
\bigcup_{\ell \geq 1}  \bigcap_{M \geq 1} \bigcup_{N \geq M} \bigcap_{   m \geq \ell}
(K_m^{\complement} \cap L_m^{\complement} )
 \end{equation}
 and 
 $$
  K_m^{\complement} = 
   \bigcap_{j \geq 1} \bigcap_{k=1}^m
  (\tilde{\Phi}_{1, T^m_{j, k}}^N)^{-1} E_{j, m}^{\complement} ,
  \qquad 
  L_m^{\complement} :=  \bigcap_{k= 1 , \ldots, m(m-1)} 
 (\tilde{\Phi}_{1, \mathcal{T}_{k - 1}}^N)^{-1} E_m^{\complement}.
 $$
From Lemma \ref{LemmaMesureE} it follows that 
$$
\rho_s(E^{\complement}) = 1.
$$
Let $0 \leq s' < s$ with $s \in (0,1)$. We recall that we
aim to control the growth as $t \to \infty$ of the 
$\mathcal C^{s'}$-norm of $\rho_s$-typical solutions in a quantitative way.
To do so we first control the norm of solutions of the truncated equation, taking advantage of the quantitative
quasi-invariance of the measure $\rho_s$ proved in 
\eqref{QuasiInvQuant}. We obtain the following statement. 
\begin{prop}\label{QuantControlCs}
Let $s \in (0,1)$, $0 \leq s' < s$ and $\varepsilon \in (0,1/2)$. Let $f \in E^{\complement}$. There exists 
 $T(f)\geq 1$, $C(f) >0$ and a diverging sequence of natural numbers $\{ N_p(f) \}_{p \in \N}$
such that 
\begin{equation}\label{GlobSmooth2Preq}
   \sup_{n \in \N}\|  \tilde{\Phi}_{1, t}^{N_p(f)} f \|_{\mathcal C^{s'}} \lesssim C(f) +  t^{\varepsilon}.
\end{equation}
as well as
\begin{equation}\label{GlobSmooth2PreqTWU}
 \sup_{n \in \N}\|  \tilde{\Phi}_{1, t}^{ N_p(f)} P_{\leq  N_p(f)} f \|_{\mathcal C^{s'}} \lesssim C(f) +  t^{\varepsilon}.
\end{equation}
Moreover, 
given any $t>T(f)$ the estimate 
\begin{equation}\label{QuantControlCstrisFinalBound}
\left\| \tilde{\Phi}_{1, t}(P_{\leq N_p(f)} f) - \tilde{\Phi}_{1, t}^{N_p(f)}(P_{\leq N_p(f)}f)  \right\|_{\mathcal C^{s'}} \lesssim 
C^{t^2}  N_p(f)^{\frac{s' -s}{2}}
\end{equation}
holds for all $p$ sufficiently large.
\end{prop}
\begin{remark}
Using \eqref{ComparisonFlows} we see that under the same assumptions we have 
\begin{equation}\label{GlobSmooth2Preqngjfkejngjfdkjgn}
   \sup_{n \in \N}\|  \Phi_{1, t}^{ N_p(f)} f \|_{\mathcal C^{s'}} \lesssim C(f) +  t^{\varepsilon}
\end{equation}
as well as
\begin{equation}\label{GlobSmooth2PreqngjfkejngjfdkjgnTWU}
 \sup_{n \in \N}\|  \Phi_{1, t}^{\tilde N_n(f)} P_{\leq \tilde N_n(f)} f \|_{\mathcal C^{s'}} \lesssim C(f) +  t^{\varepsilon}.
\end{equation}
\end{remark}
\begin{proof}
Let $f \in E^{\complement}$. Remember $\rho_s (E^{\complement}) = 1$ and  
$$E^\complement = 
\bigcup_{\ell \geq 1}  \bigcap_{M \geq 1} \bigcup_{N \geq M} \bigcap_{   m \geq \ell}
(K_m^{\complement} \cap L_m^{\complement} ),
$$
where
 $$
  K_m^{\complement} = 
   \bigcap_{j \geq 1} \bigcap_{k=1}^m
  (\tilde{\Phi}_{1, T^m_{j, k}}^N)^{-1} E_{j, m}^{\complement} ,
  \qquad 
  L_m^{\complement} :=  \bigcap_{k= 1 , \ldots, m(m-1)} 
 (\tilde{\Phi}_{1, \mathcal{T}_{k - 1}}^N)^{-1} E_m^{\complement}.
 $$
Thus if $f \in E^\complement$ then it exists 
an integer $\bar \ell (f)$ and a subsequence $N_p(f) \overset{p \to \infty}{ \to } \infty$ 
such that 
\begin{equation}\label{ndjskdkjsngjdskdjgnjskPreq}
f \in   (\tilde{\Phi}_{1, T^{m}_{j, k-1}}^{N_p(f)})^{-1} E_{j, m}^{\complement} \qquad  \forall m \geq \bar \ell (f), \, \forall j \geq 1, \, \forall k=1, \ldots, m 
\end{equation}
and
\begin{equation}\label{ndjskdkjsngjdskdjgnjsk}
f \in   (\tilde{\Phi}_{1, T_{k - 1}}^{N_p(f)})^{-1} E_m^{\complement} 
\qquad  \forall m \geq \bar \ell (f), \, \forall k=1, \ldots, m(m-1) .
\end{equation}
We will use \eqref{ndjskdkjsngjdskdjgnjskPreq} to deduce \eqref{GlobSmooth2Preq}-\eqref{GlobSmooth2PreqTWU} and 
\eqref{ndjskdkjsngjdskdjgnjsk} to deduce \eqref{QuantControlCstrisFinalBound}.
We start proving \eqref{GlobSmooth2Preq}-\eqref{GlobSmooth2PreqTWU}.
Recalling \eqref{Def:BadSets} we see that if $f \in  (\tilde{\Phi}_{1, T^{m}_{j, k-1}}^N)^{-1} E_{j, m}^{\complement}$ 
then 
 \begin{equation}\label{WeSee1}
 \sup_{t \in \R} \| e^{it\Delta} \tilde{\Phi}_{1, T^{m}_{j, k-1}}^{N} f \|_{\mathcal C^{s' + \kappa}} \leq m + T_j^{\varepsilon}.  \qquad 
\end{equation}
On the other hand, 
the solution 
$$
w^N(t) := \tilde{\Phi}_{1, t}^N f  
$$
can 
 be represented,
over any interval $[T^{m}_{j, k-1}, T^{m}_{j, k}]$, 
as in Proposition \ref{1/2SmoothingNLSPert}:
\begin{equation}\label{nfdlksjdngjskjdngjskdjng}
P_{\leq N} \tilde{\Phi}_{1, t}^N f  = e^{i (t - T^{m}_{j, k-1})  \Delta} P_{\leq N} \tilde{\Phi}_{1, T^{m}_{j, k-1}}^N f
  + \tilde{w}^N(t),   \qquad t \in [T^{m}_{j, k-1}, T^{m}_{j, k}].
\end{equation}
If we  take $\sigma < 1/2$ sufficiently close to $1/2$ and $s''\in (s',s)$, in such a way that 
$$s' + 1/2 < s'' + 1/2<s + \sigma,$$ 
then  we get  by  Sobolev embedding and \eqref{BoundWNSmoothing} 
\begin{equation}\label{ErdLocPreqPreq}
\sup_{t \in [T^{m}_{j, k-1}, T^{m}_{j, k}] }  
\| \tilde{w}^N(t) \|_{\mathcal C^{s''}(\T)} 
 \lesssim
\sup_{t \in [T^{m}_{j, k-1}, T^{m}_{j, k}] } \| \tilde{w}^N(t) \|_{H^{s + \sigma}(\T)} 
\lesssim 1,
\end{equation}
as long as $\tilde{\Phi}_{1, T^{m}_{j, k-1}}^N f \in \Omega_{\frac{c}{j^{\varepsilon} m}}^\complement$.
\\

We note that if $f \in  (\tilde{\Phi}_{1, T^{m}_{j, k-1}}^N)^{-1} E_{j, m}^{\complement}$ then 
$\tilde{\Phi}_{1, T^{m}_{j, k-1}}^N f \in E_{j, m} ^{\complement}$. Since 
$E_{j, m} ^{\complement} \subseteq \Omega_{\frac{c}{j^{\varepsilon} m}}^\complement$ we are 
 allowed to use \eqref{ErdLocPreqPreq}.
Combining it with \eqref{WeSee1} and with a boundedness property of $P_{\leq N}$ over H\"older spaces 
we deduce 
\begin{equation}\label{WeSee2}
\sup_{t \in [T^{m}_{j, k -1}, T^{m}_{j, k}]} \|  P_{\leq N} \tilde{\Phi}_{1, t}^{N} f \|_{\mathcal C^{s'}} \lesssim m + T_j^{\varepsilon}, 
\end{equation}
for all $f \in  (\tilde{\Phi}_{1, T^{m}_{j, k-1}}^N)^{-1} E_{j, m}^{\complement}$. 
We now look at high Fourier modes. We have (recall \eqref{HighModesLinearEvolutionGauged})
\begin{equation}\label{nfdlksjdngjskjdngjskdjng2}
P_{> N} \tilde{\Phi}_{1, t}^N f  =  
e^{- 2 i \mu_N (\ln t - \ln T^{m}_{j, k-1}) + i (t - T^{m}_{j, k-1})  \Delta} P_{> N} \tilde{\Phi}_{1, T^{m}_{j, k-1}}^N f.
\end{equation}
Thus if $f \in  (\tilde{\Phi}_{1, T^{m}_{j, k-1}}^N)^{-1} E_{j, m}^{\complement}$
we can invoke again
\eqref{WeSee1} to prove
\begin{equation}\label{WeSee2HighModes}
\sup_{t \in [T^{m}_{j, k -1}, T^{m}_{j, k}]} \|  P_{> N} \tilde{\Phi}_{1, t}^{N} f \|_{\mathcal C^{s'}} \lesssim m + T_j^{\varepsilon}.
\end{equation}
In conclusion, using \eqref{ndjskdkjsngjdskdjgnjskPreq} we can deduce that there exists an integer 
$\bar \ell (f)$ and a subsequence $N_p(f) \overset{p \to \infty}{ \to } \infty$ 
such that  
\begin{equation}\label{RightOne}
\sup_{k = 1, \ldots, m} \, \sup_{t \in [  T^{m}_{j, k-1}, T^{m}_{j, k} ]  }  \|  \tilde{\Phi}_{1, t}^{N_p(f)} f \|_{\mathcal C^{s'}} \lesssim m +  T_j^{\varepsilon} 
\qquad \forall  m \geq \bar \ell (f), \, \forall j \geq 1.   
\end{equation}
Since 
$$[1, \infty) = \bigcup_{j \geq 1} \bigcup_{k = 1, \ldots, m}  [ T^{m}_{j, k-1}, T^{m}_{j, k} ] , \qquad \forall m \in \N^*$$
we can use  \eqref{RightOne} to deduce that for $f \in E^\complement$ there exists an integer 
$\bar \ell (f)$ and a subsequence~$N_p(f) \overset{p \to \infty}{ \to } \infty$ 
such that  
\begin{equation*}
\sup_{p \in \N}
 \|  \tilde{\Phi}_{1, t}^{N_p(f)} f \|_{\mathcal C^{s'}} \lesssim    t^{\varepsilon} + \bar \ell (f) 
\qquad \forall  m \geq \bar \ell (f),    
\end{equation*}
from which we can deduce 
 \eqref{GlobSmooth2Preq}. The estimate \eqref{GlobSmooth2PreqTWU} can be deduced in the same way, 
noting that $P_{\leq N} E^\complement \subset E^\complement$ (see definitions \ref{Def:BadSets}-\eqref{FnDef} and use 
\eqref{fdnjskjhbhjjbhghgyhgfhdjs} and Remark \eqref{OmegaDeltaStability}) and 
\begin{equation}\label{PNFlowIdentity}
\tilde{\Phi}_{1, t}^{\tilde N_n(f)} P_{\leq \tilde N_n(f)} f  = P_{\leq \tilde N_n(f)}  \tilde{\Phi}_{1, t}^{\tilde N_n(f)}  f  ;
\end{equation}
and invoking the boundedness of $P_{\leq N}$ from $C^s$ to $C^{s - \varepsilon}$, $\varepsilon >0$.
The identity \eqref{PNFlowIdentity} easily follows from the definition of $\tilde{\Phi}_{1, t}^{\tilde N_n(f)}$
as flow associated with the equations \eqref{NLSvIntrogauged}-\eqref{HighModesLinearEvolutionGauged}.  
\\

We will now use \eqref{ndjskdkjsngjdskdjgnjsk} to deduce \eqref{QuantControlCstrisFinalBound}. 
We will in fact prove that for all $f$ that satisfy \eqref{ndjskdkjsngjdskdjgnjsk} we have the following: given any natural number $m \geq \bar \ell(f) $
and $t \in [\mathcal{T}_{k-1}, \mathcal{T}_k ]$, $k = 1, \ldots, m(m-1)$, we have 
%
%
\begin{equation}\label{PertGrowthBis}
\tilde{\Phi}_{1, t}(P_{\leq N_p} f) \in \tilde{\Phi}_{1, t}^{N_p}( P_{\leq N_p} f)  +  P_{> N_p} E_m^\complement + 
B^{s' + \sigma}\left((4C)^{k-1} N_p^{s' -s}\right)
\end{equation}\label{djskhhgfjkdsdjhg}
for all $p$ sufficiently large that 
$$(4C)^{m^2} N_p^{s' -s} < 1.$$
Here $C >1$ is the constant from Propositions \ref{1/2SmoothingNLSPertTris}-\ref{1/2SmoothingNLSPert2} and  we used the abbreviation $N_p = N_p(f)$. 
We prove \eqref{PertGrowthBis}  by induction of $k$.
Since 
$$
f \in   (\tilde{\Phi}_{1, \mathcal{T}_{0}}^{N_p(f)})^{-1} E_m^{\complement} = (\tilde{\Phi}_{1, 1}^{N_p(f)})^{-1} E_m^{\complement} = 
E_m^{\complement} ,
$$
 the case $k=1$ is covered by 
Proposition \ref{1/2SmoothingNLSPert2}. 
Thus we assume that for some $k \geq 2$ we have 
\begin{equation}\label{Growth1}
\tilde{\Phi}_{1, t}(P_{\leq N_p} f) \in \tilde{\Phi}_{1, t}^{N_p}(P_{\leq N_p}  f)  + P_{> N_p}  E_m^\complement +
B^{s' + \sigma}\left((4C)^{k-2} N_p^{s' -s}\right), \qquad t \in [\mathcal{T}_{k-2}, \mathcal{T}_{k-1} ]
\end{equation} 
and we want to show 
\begin{equation}\label{Growth1Thesis}
\tilde{\Phi}_{1, t}(P_{\leq N_p} f)  \in \tilde{\Phi}_{1, t}^{N_p}(P_{\leq N_p} f) + P_{> N_p}   E_m^\complement +
B^{s' + \sigma}\left( (4C)^{k-1} N_p^{s' -s} \right), \qquad t \in  [\mathcal{T}_{k-1}, \mathcal{T}_{k} ].
\end{equation}

In fact we will only use the 
induction assumption \eqref{Growth1} at $t= \mathcal{T}_{k-1}$, namely  
\begin{equation}\label{Growth1used}
\tilde{\Phi}_{1, \mathcal{T}_{k-1}}(P_{\leq N_p} f) \in \tilde{\Phi}_{1, \mathcal{T}_{k-1}}^{N_p}(P_{\leq N_p} f)  + P_{> N_p}   E_m^\complement +
B^{s' + \sigma}\left((4C)^{k-2} N_p^{s' -s}\right).
\end{equation}
We thus consider $t \in [\mathcal{T}_{k-1}, \mathcal{T}_{k}]$. Using the 
group property of the flow, we have
\begin{align}\nonumber
 & \tilde{\Phi}_{1, t}  P_{\leq N_p} f  - \tilde{\Phi}_{1, t}^{N_p}  P_{\leq N_p} f 
 \\ \nonumber  & =
   \tilde{\Phi}_{\mathcal{T}_{k-1}, t}   \tilde{\Phi}_{ 1, \mathcal{T}_{k-1} } P_{\leq N_p} f - \tilde{\Phi}_{\mathcal{T}_{k-1}, t} \tilde{\Phi}_{ 1, \mathcal{T}_{k-1} }^{N_p}  P_{\leq N_p} f  
  \\ \nonumber &+  \tilde{\Phi}_{\mathcal{T}_{k-1}, t} \tilde{\Phi}_{ 1, \mathcal{T}_{k-1} }^{N_p} P_{\leq N_p} f  - \tilde{\Phi}_{\mathcal{T}_{k-1}, t}^{N_p} \tilde{\Phi}_{ 1, \mathcal{T}_{k-1} }^{N_p}  P_{\leq N_p} f.
\end{align}
We will show that 
\begin{equation}\label{FirstInd}
  \tilde{\Phi}_{\mathcal{T}_{k-1}, t}   \tilde{\Phi}_{ 1, \mathcal{T}_{k-1} } P_{\leq N_p} f - \tilde{\Phi}_{\mathcal{T}_{k-1}, t} \tilde{\Phi}_{ 1, \mathcal{T}_{k-1} }^{N_p} P_{\leq N_p} f   
  \in P_{> N_p}   E_m^\complement + 
  B^{s' + \sigma}\left((4C)^{k-2} 2C N_p^{s' -s}\right)
\end{equation}
and 
\begin{equation}\label{SecondInd}
\tilde{\Phi}_{\mathcal{T}_{k-1}, t} \tilde{\Phi}_{ 1, \mathcal{T}_{k-1} }^{N_p} P_{\leq N_p} f  - \tilde{\Phi}_{\mathcal{T}_{k-1}, t}^{N_p } \tilde{\Phi}_{ 1, \mathcal{T}_{k-1} }^{N_p} P_{\leq N_p} f \in 
  B^{s' + \sigma}\left((4C)^{k-2} 2C   N_p^{s' -s} \right) ,
\end{equation}
from which we deduce \eqref{Growth1Thesis}.

\underline{Proof of \eqref{FirstInd}} : 
by the definition of the truncated flow it is clear that 
$$\tilde{\Phi}_{ 1,  t}^{N} P_{\leq N} f =  P_{\leq N} \tilde{\Phi}_{ 1, t }^{N}  f.$$ 
Recalling \eqref{ndjskdkjsngjdskdjgnjsk}
we see that $\tilde{\Phi}_{ 1, \mathcal{T}_{k-1} }^{N_p}  f \in \Omega_{\frac1m}^\complement$ and thus
$$
 \tilde{\Phi}_{ 1, \mathcal{T}_{k-1} }^{N_p} P_{\leq N_p}  f  = P_{\leq N} \tilde{\Phi}_{ 1, \mathcal{T}_{k-1} }^{N_p}  f \in P_{\leq N_p} E_m^\complement \subset E_m^\complement,
$$
where in the last identity we used $P_{\leq N} E_m^\complement \subset  E_m^\complement$ that easily follows from the 
definition \eqref{FnDef} and from Remark \ref{OmegaDeltaStability}.
\\

By the induction assumption \eqref{Growth1used}, we also have 
$$
\tilde{\Phi}_{1, \mathcal{T}_{k-1}}(P_{\leq N_p} f) \in \tilde{\Phi}_{1, \mathcal{T}_{k-1}}^{N_p}( P_{\leq N_p} f)  
+ P_{> N_p}  E_m^\complement 
+   B^{s' + \sigma}\left((4C)^{k-2}  N_p^{s' -s} \right).
 $$
Since $E_m^{\complement} \subset \Omega_{\frac1m}^\complement$  we can apply 
 Proposition \ref{1/2SmoothingNLSPertTris} to deduce
 \begin{equation}\label{FirstIndPreq}
  \tilde{\Phi}_{\mathcal{T}_{k-1}, t}   \tilde{\Phi}_{ 1, \mathcal{T}_{k-1} } P_{\leq N_p} f - \tilde{\Phi}_{\mathcal{T}_{k-1}, t} 
  \tilde{\Phi}_{ 1, \mathcal{T}_{k-1} }^{N_p} P_{\leq N_p} f   
  \in P_{> N_p}   e^{i(t - \mathcal{T}_{k-1})\Delta} E_n^\complement + 
  B^{s' + \sigma}\left((4C)^{k-2} 2C N_p^{s' -s}\right).
 \end{equation}
On the other hand, looking at the definition \eqref{FnDef} and recalling Remark \eqref{OmegaDeltaStability}, we 
note that we have $e^{i(t - \mathcal{T}_{k-1})\Delta} E_m^{\complement} = E_m^{\complement}$. Thus \eqref{FirstIndPreq} implies \eqref{FirstInd}.

\underline{Proof of \eqref{SecondInd}}: using again 
$ \tilde{\Phi}_{ 1, \mathcal{T}_{k-1} }^{N_p} P_{\leq N_p}  f  \in  E_m^\complement \subseteq \Omega_{\frac1m}^\complement$,
we see that we can deduce~\eqref{SecondInd} from 
Proposition~\ref{1/2SmoothingNLSPert2}. 
\\

 Recall $s'' \in (\frac{s' + s}{2}, s)$. Using Bernstein inequality 
and recalling~\eqref{FnDef} , we can show that for all $f \in F_{m}^{\complement} \subset E_{m}^{\complement}$ and for all 
we have for all $\varepsilon >0$: 
\begin{equation}\label{fndjskdjgf}
  \|   P_{\geq N_p} f \|_{\mathcal C^{s'}_{x}} \lesssim 
 N_p^{s' - s'' + \varepsilon'}  \|   P_{\geq N_p} f \|_{\mathcal C^{s''}_{x}} \lesssim 
  N_p^{\frac{s' - s}{2} } m^{\varepsilon} 
   \lesssim N_p^{\frac{s' - s}{2} } m^{\varepsilon};
\end{equation}
in the second inequality we have chosen $\varepsilon' >0$ sufficiently small.
Combining \eqref{PertGrowthBis}-\eqref{fndjskdjgf} we have proved that 
for all $f \in E^{\complement}$ there 
exists a natural number $\bar \ell(f) \geq 1$ and
a subsequence $N_{p} = N_p(f) \to \infty$ 
such that the following holds. Given  $m \geq \bar \ell(f)$ and $t \in [1, m]$, we have
\begin{equation}
\left\| \tilde{\Phi}_{1, t}(P_{\leq N_p(f)} f) - \tilde{\Phi}_{1, t}^{N_p(f)}(P_{\leq N_p(f)}f)  \right\|_{\mathcal C^{s'}} \lesssim 
(5C)^{m^2}  N_p^{\frac{s' -s}{2}} ,
\end{equation}
for all $p$ that satisfy \eqref{djskhhgfjkdsdjhg}. 
Taking $t \in [m-1, m]$, we see from the estimate above that
for all $f \in E^{\complement}$  there 
exists a $N_{p} = N_p(f) \to \infty$ and 
 $T(f) \geq 1$ such that, given any $t > T(f)$ one has
\begin{equation}
\left\| \tilde{\Phi}_{1, t}(P_{\leq N_p(f)} f) - \tilde{\Phi}_{1, t}^{N_p(f)}(P_{\leq N_p(f)}f)  \right\|_{\mathcal C^{s'}} \lesssim 
(5C)^{t^2}  N_p^{\frac{s' -s}{2}} ,
\end{equation}
for all $p$ sufficiently large (depending on $t$).  
The statement \eqref{QuantControlCstrisFinalBound} then follows, after renaming the constant $C$.
\end{proof}
\subsection{Proof of Theorem \ref{QuantControlCsBisFinalBound}}
Using Proposition~\ref{QuantControlCs}  
 we can  finally control of the growth as $t \to \infty$ of the 
$\mathcal C^{s'}$-norm of $\rho_s$-typical solutions, namely we are ready to prove theorem \ref{QuantControlCsBis}. 
From \eqref{ComparisonFlows} (with $N= \infty$) it suffices to prove that 
for $\rho_s$-almost every $f$ there exists a
 diverging sequence of natural numbers $\{ N_n(f) \}_{n \in \N}$
 and $C(f) >0$, $T(f) \geq 1$ such that for all $t > T(f)$: 
\begin{equation}\label{QuantControlCsBisFinalBound2}
 \limsup_{n \to \infty} \|  \tilde{\Phi}_{1, t} P_{\leq N_n(f)} f \|_{ \mathcal C^{s'}} \lesssim C(f) +  t^{\varepsilon}.
\end{equation}
It follows then from Proposition \ref{QuantControlCs} that 
for $\rho_s$-almost every $f$ there exist 
a
 diverging sequence of natural numbers $\{ N_n(f) \}_{n \in \N}$
 and $C(f) >0, T(f)\geq1$
such that
\begin{equation}\label{jfdusiygygygygy}
 \sup_{n \in \N}\|  \tilde{\Phi}_{1, t}^{N_n(f)} P_{\leq N_n(f)} f \|_{ \mathcal C^{s'}} \lesssim C(f) +  t^{\varepsilon}
\end{equation}
and such that, given any $t > T(f)$ one has
\begin{equation} 
\left\| \tilde{\Phi}_{1, t}(P_{\leq N_n(f)} f) - \tilde{\Phi}_{1, t}^{N_n(f)}( P_{\leq N_n(f)} f)  \right\|_{\mathcal C^{s'}} \lesssim 
C^{t^2}  N_n(f)^{\frac{s' -s}{2}} ,
\end{equation}
for all $n$ sufficiently large.
We thus get by triangle inequality that 
that 
for $\rho_s$-almost every $f$ there exists a
 diverging sequence of natural numbers $\{ N_n(f) \}_{n \in \N}$
 and $C(f) >0, T(f)\geq1$ 
such that, for all $t > T(f)$ one has 
\begin{equation}\label{QuantControlCsBisFinalBound2dfmsjkdgjdnjfdknjgnjhdfngkjn}
 \|  \tilde{\Phi}_{1, t} P_{\leq N_n(f)} f \|_{ \mathcal C^{s'}} \lesssim C(f) +  t^{\varepsilon} + 
 C^{t^2}  N_n(f)^{\frac{s' -s}{2}},
\end{equation}
for all $n$ sufficiently large. 
The statement then follows passing 
to the $\limsup_{n}$ (recall $s'<s$).
\section{Proof of Theorem \ref{thbf2}}
We are now ready to prove Theorem~\ref{thbf2}. As noted already in the introduction, the first part of Theorem \ref{thbf2} 
is a consequence of Theorem \ref{thbf}, since $\omega$-almost surely we have 
 $a_j^{\omega} \in l^{2, s'}$, for all $0 \leq s' < s$. The fact that the curve $\chi^{\omega}(t)$ belongs to $\mathcal C^{2+s'}(\R)$
 ($\omega$-almost surely) is a consequence of 
 Corollary \ref{CorWeak1}, equation \eqref{GaugeMap}, the pseudoconformal change of variable and the
 (almost-sure) local well-posedenss theory developed in Section \ref{Sec:thbf2}.
Thus it only remains to prove
  the existence of the limit curve as $t \to 0$ and the 
  fact that the trajectories in time  $\chi^\omega(\cdot, x)$ have H\"older regularity $\mathcal C^{\frac 12-\varepsilon}$ for all $x\in\mathbb R$. In order to do so we will use
  the crucial estimate \eqref{QuantControlCsBisFinalBound} from Theorem \ref{QuantControlCsBis}.  
We denote with $v_1^\omega$ the initial datum of the NLS equation \eqref{NLSvIntro}, that is distributed as the 
random Fourier series \eqref{TMIATSIntro} and 
$$
v_1^{\omega,n} := P_{\leq n} v_1^\omega,
$$
where $P_{\leq n}$ is the projection of the first $n \in 2^{\N}$ Fourier modes.
Let then $v^{\omega,n} = \Phi_{1, t} v_1^{\omega,n}$ the associated solution.
By pseudoconformal transformation we then have a solution $u^{\omega,n}$ of the cubic NLS
 \eqref{CubicNLS} and by the Hasimoto procedure a curve $\chi^{\omega,n}$ that satisfies 
the binormal flow equation \eqref{BinormalVeryFirst}. 
Invoking \eqref{QuantControlCsBisFinalBound} we see that for almost every $\omega$ it exists $C(\omega) >0$,
$T(\omega) \geq1$
and a diverging sequence of natural numbers $\{ N_{k}(\omega) \}_{k \in \N}$ such that for all  $t > T(\omega)$:
\begin{equation}\label{QuantControlCsBisFinalBoundcasa}
\limsup_{k \to \infty} \|  \Phi_{1, t} v_1^{\omega, N_k(\omega)} \|_{\mathcal C^{s'}} \lesssim C(\omega) +  t^{\varepsilon},
\end{equation}

Let $0 < t_1 < t_2  < 1/T(\omega)$. We can proceed as in \eqref{limitchit}-\eqref{limitchitbis} and use \eqref{chit} to get
\begin{align*}
| \chi^\omega(t_2, x)  -  \chi^\omega(t_1, x) | 
& \leq \limsup_{k\to \infty} | \int_{t_1}^{t_2} \partial_s \chi^{\omega, N_{k}(\omega)}(s, x) ds | 
\\ 
&
\leq   \limsup_{k\to \infty} \int_{t_1}^{t_2} | u^{\omega, N_{k}(\omega)}(x, s) | ds .
 \end{align*}
 Using the pseudo-conformal transformation and changing variable $s= \frac{1}{t}$
 in the integral we get, for all $\varepsilon \in (0, \frac14)$:
 \begin{align}\label{CauchyInTime}
| \chi^\omega(t_2, x) -  \chi^\omega(t_1, x) | 
&
 \leq  \limsup_{k \to \infty} \int_{t_1}^{t_2}  
  \frac{1}{\sqrt{s}} \Big| v^{\omega, N_{k}(\omega)} \left( \frac{1}{s}, \frac{x}{s} \right) \Big| ds 
\\ \nonumber
&
 = 
 \limsup_{k \to \infty} \int_{1/t_2}^{1/t_1}  t^{-\frac{3}{2}}  | v^{\omega, N_{k}(\omega)} \left( t, t x \right) | dt
\\  \nonumber
&
= \limsup_{k \to \infty} \int_{1/t_2}^{1/t_1}  t^{-\frac{3}{2}}  | (\Phi_{1, t} v_1^{\omega, N_{k}(\omega)})(tx) | dt  
\\ 
& \nonumber
\leq   \int_{1/t_2}^{1/t_1} \big( C(\omega) t^{-\frac{3}{2}} + C t^{-\frac{3}{2} + \varepsilon} \big)  dt 
\\ \nonumber
&
\leq 4 \big( C(\omega) t_2^{\frac{1}{2}} + C t_2^{\frac{1}{2} - \varepsilon}   \big) 
\end{align} 
where we have used \eqref{QuantControlCsBisFinalBoundcasa} in the second inequality.

Since $$\lim_{t_2 \to 0} C(\omega) t_2^{\frac{1}{2}} + C t_2^{\frac{1}{2} - \varepsilon}  =0  $$
we have proved that, $\omega$-almost surely, the curve $ \chi^\omega(t, x) $ converges uniformly (w.r.t $x$) to a limit curve 
$$
\chi^\omega(0, x) = \lim_{t \to 0} \chi^\omega(t, x),
$$
as $t \to 0$. 
The limit curve $\chi^\omega(0, x)$ is continuous. 
From \eqref{CauchyInTime} we also get that $\chi^\omega(\cdot, x)$ is $\mathcal C^{\frac 12-\varepsilon}$. This concludes the proof of Theorem \ref{thbf2}.


\begin{thebibliography}{9}

\bibitem{BLTV} V. Banica, R. Luc\`a, N. Tzvetkov, L. Vega,   {\it Blow-up for the 1D cubic NLS},  Comm. Math. Phys. 405 (2024), no. 1, Paper No. 11, 21 pp.





\bibitem{BVJMP}  V.~Banica and L.~Vega,  {\it Turbulent solutions of the binormal flow and the {1D} cubic {S}chr{\"o}dinger equation}, J. Math. Phys. special topic ICMP 2024, to appear.
\bibitem{BVAENS}  V.~Banica and L.~Vega, 
{\it  The initial value problem for the binormal flow with rough data}, 
Ann. Sci. \'Ec. Norm. Sup\'er. 48 (2015), 1421--1453.

\bibitem{BVAnnPDE}   V.~Banica and L.~Vega, {\it Evolution of polygonal lines by the binormal flow},  Ann. PDE  6 (2020), Paper No. 6, 53 pp.

\bibitem{Bishop} R.L. Bishop, {\it There is More than One Way to Frame a Curve}, Am. Math. Mon. 82 (1975), 246-251.

\bibitem{B94} J. Bourgain,  {\it Fourier transform restriction phenomena for certain lattice subsets and applications to nonlinear evolution equations, I: Schr\"odinger equations},  Geom. Funct. Anal. 3 (1993), 107--156.

\bibitem{B94_pak} J.~Bourgain, {\it Periodic nonlinear {S}chr{\"o}dinger equation and invariant measures}, Comm. Math. Phys., 166, (1994), 1--26.

\bibitem{Bo96} J. Bourgain, {\it Invariant measures for the 2D-defocusing nonlinear Schr\"odinger equation}, Comm. Math. Phys. 176 (1996) 421-445.	

\bibitem{BD} J. Bourgain, C. Demeter, {\it The proof of the $l^2$ Decoupling Conjecture}, Annals of Math. 182 (2015), 351--389.


\bibitem{BT}  N.~Burq, L.~Thomann,  {\it Almost sure scattering for the one dimensional nonlinear Schr\"odinger equation}, 
Mem. Amer. Math. Soc. 296 (2024), no. 1480, vii+87 pp.

\bibitem{BT1} N.~Burq, N.~Tzvetkov,  {\it Random data Cauchy theory for supercritical wave equations. I. Local theory}, Invent. Math. 173 (2008), no. 3, 449--475.
%
\bibitem{BT2} N.~Burq, N.~Tzvetkov,  {\it Probabilistic well-posedness for the cubic wave equation}, J. Eur. Math. Soc. (JEMS) 16 (2014), no. 1, 1--30.

\bibitem{ChShUh}  N.H. Chang, J. Shatah and K. Uhlenbeck, {\it Schr\"odinger maps},  Comm. Pure Appl. Math. 53 (2000), 590–-602.


\bibitem{ChErTz}  V Chousionis, MB Erdogan and N Tzirakis, {\it Fractal solutions of linear and nonlinear dispersive partial differential equations},
Proceedings of the London Mathematical Society, 110 (2014), 543--564.

\bibitem{CT} J. Coe and L. Tolomeo, {\it Sharp quasi-invariance threshold for the cubic Szeg\"o equation},  arXiv:2404.14950, (2024).


\bibitem{CO} J. Colliander, T. Oh,  {\it Almost sure well-posedness of the cubic nonlinear Schr\"odinger equation below $L^2(\T)$},  Duke Math. J. 161 (2012) 367--414.

\bibitem{DT} A. Debussche and Y. Tsutsumi, {\it Quasi-invariance of Gaussian measures transported by the cubic
NLS with third-order dispersion on $\T$}, J. Funct. Anal., 281 (2021), pp. Paper No. 109032, 23.

\bibitem{DHV1} 
F.~de la Hoz and L.~Vega,  {\it Vortex filament equation for a regular polygon},  Nonlinearity  27 (2014), 3031--3057.

\bibitem{DiWa} W. Ding and Y. Wang,  {\it Local Schr\"odinger flow into K\"ahler manifolds},  Sci. China Ser. A  44 (2001), 1446–-1464.

\bibitem{ErTz} B. Erdogan and N. Tzirakis, Dispersive Partial Differential Equations: Wellposedness and Applications, London Mathematical Society Student Texts. 86, Cambridge University Press, 2016.

\bibitem{ETsmoothing} B. Erdogan and N. Tzirakis,
{\it Talbot effect for the cubic nonlinear Schr\"odinger equation on the torus}, Math. Res. Lett. 20 (2013), 1081--1090.


\bibitem{JeSm2}
R.~L.~Jerrard and D.~Smets,
{\it On the motion of a curve by its binormal curvature},
J. Eur. Math. Soc. 17 (2015), 1148--1515. 

\bibitem{FG1}
F. Flandoli, M. Gubinelli, {\it The Gibbs ensemble of a vortex filament}, Probab Theory Relat Fields 122 (2002), 317--340. 

\bibitem{FG2}
F. Flandoli, M. Gubinelli, {\it Electron. J. Probab}, 10 (2005), 865--900. 

\bibitem{FoSe} J. Forlano and K. Seong,  {\it Transport of Gaussian measures under the flow of one-dimensional fractional nonlinear Schr\"odinger equations}, 
Comm. PDE, 47 (2022), 1296--1337.

\bibitem{FT} J. Forlano and L. Tolomeo, {\it Quasi-invariance of Gaussian measures of negative regularity for
fractional nonlinear Schr\"odinger equations}, arXiv preprint arXiv:2205.11453, (2022).

\bibitem{FT2} J. Forlano and L. Tolomeo, {\it Quasi-invariance of the gaussian measure for the two-dimensional
stochastic cubic nonlinear wave equation}, arXiv preprint arXiv:2409.20451, (2024).

\bibitem{FTr} J. Forlano and W. J. Trenberth, {\it On the transport of Gaussian measures under the one-
dimensional fractional nonlinear Schr\"odinger equations}, Ann. Inst. H. Poincar\'e C Anal. Non
Lin\'eaire, 36 (2019), pp. 1987–2025.

\bibitem{GLT23} G. Genovese, R. Luc\`a, N. Tzvetkov, {\it Transport of Gaussian measures with exponential cut-off for Hamiltonian PDEs}, Journal d'Analyse Math\'ematique, DOI 10.1007/s11854-023-0292-1 (2023)

\bibitem{GLT22} G. Genovese, R. Luc\`a, and N. Tzvetkov, {\it Quasi-invariance of low regularity Gaussian measures
under the gauge map of the periodic derivative NLS}, J. Funct. Anal., 282 (2022), pp. Paper No.
109263, 45.

\bibitem{GLT22bis}G. Genovese, R. Luc\`a, and N. Tzvetkov, {\it Quasi-invariance of Gaussian measures for the periodic Benjamin-Ono-BBM equation}, Stoch.
Partial Diﬀer. Equ. Anal. Comput., 11 (2023), pp. 651–684.


\bibitem{GH} E. Gussetti, M. Hofmanová, {\it Statistical solutions to the Schr\"odinger map equation in 1D, via the randomly forced Landau-Lifschitz-Gilbert equation}, arXiv:2501.16499, (2025).


  \bibitem{GRV} S.~Guti\'errez, J.~Rivas and L.~Vega, 
  {\it Formation of singularities and self-similar vortex motion under the localized induction approximation}, 
  Commun. PDE 28 (2003) 927--968. 

\bibitem{H} H.~Hasimoto, {\it A soliton in a vortex filament}  J. Fluid Mech.  51 (1972), 477--485.

\bibitem{JeSm} R.~L.~Jerrard and  D.~Smets,  {\it On Schr\"odinger maps from $T^1$ to $S^2$},  Ann. Sci. \'Ec. Norm. Sup\'er.  45 (2012), 637--680. 
 
\bibitem{K} A. Knezevitch, {\it Qualitative quasi-invariance of low regularity Gaussian measures for the 1d quintic nonlinear Schr\"odinger  equation}, arXiv:2502.17094, (2025).
 
 \bibitem{Kita} N. Kita,  {\it Mode generating property of solutions to the nonlinear Schr\"odinger equations in one space dimension}, Nonlinear dispersive equations, 
GAKUTO Internat. Ser. Math. Sci. Appl., Gakkotosho, Tokyo,  26 (2006), 111--128.
 
\bibitem{Ko} N.~Koiso, {\it Vortex filament equation and semilinear Schr\"odinger equation},  Nonlinear Waves, Hokkaido University Technical Report Series in Mathematics, 43 (1996), 221--226.

\bibitem{NSVZ} A. Nahmod, J. Shatah, L. Vega, C. Zeng, {\it Schr\"dinger Maps and their Associated Frame Systems},  Int. Math. Res. Not. 21, (2007), Art. ID rnm088, 29 pp.

\bibitem{NiTa} T. Nishiyama and A. Tani,  {\it Solvability of the Localized Induction Equation for Vortex Motion},  Comm. Math. Phys. 162 (1994), 433--445.

\bibitem{GOTW} T. Gunaratnam, T. Oh, N. Tzvetkov, and H. Weber, {\it Quasi-invariant Gaussian measures
for the nonlinear wave equation in three dimensions}, Probab. Math. Phys., 3 (2022), pp. 343–379.

\bibitem{OS} T. Oh and K. Seong, {\it Quasi-invariant Gaussian measures for the cubic fourth order nonlinear
Schr\"odinger equation in negative Sobolev spaces}, J. Funct. Anal., 281 (2021), pp. Paper No. 109150,
49.

\bibitem{OST} T. Oh, P. Sosoe, and N. Tzvetkov, {\it An optimal regularity result on the quasi-invariant Gaussian
measures for the cubic fourth order nonlinear Schr\"odinger equation}, J.
Ec. polytech. Math., 5 (2018),
pp. 793–841.

\bibitem{OT} T. Oh and N. Tzvetkov, {\it Quasi-invariant Gaussian measures for the cubic fourth order nonlinear
Schr\"odinger equation}, Probab. Theory Related Fields, 169 (2017), pp. 1121–1168.

\bibitem{OT2} T. Oh and N. Tzvetkov, {\it Quasi-invariant Gaussian measures for the two-dimensional defocusing cubic nonlinear wave
equation}, J. Eur. Math. Soc. (JEMS), 22 (2020), pp. 1785–1826.

\bibitem{PZ} R. Paley, A. Zygmund,  {\it On some series of functions (1)},  Proc. Camb. Philos. Soc. 26 (1930), 337–357.

\bibitem{PTV} F. Planchon, N. Tzvetkov, and N. Visciglia, {\it Transport of Gaussian measures by the flow of
the nonlinear Schr\"odinger equation}, Math. Ann., 378 (2020), pp. 389–423.

\bibitem{PTV2} F. Planchon, N. Tzvetkov, and N. Visciglia {\it Modified energies for the periodic generalized KdV equation and applications}, Ann. Inst. H.
Poincar\'e, C Anal. Non Lin\'eaire, 40 (2023), pp. 863–917.

\bibitem{RoRuSt} I. Rodnianski, Y. Rubinstein, and G. Staffilani,  {\it On the global well-posedness of the one-dimensional Schr\"odinger map flow}, 
Anal. PDE 2 (2009), 187--209.

\bibitem{SuSuBa} P.L. Sulem, C. Sulem and C. Bardos, {\it On the continuous limit for a system of classical spins},  Comm. Math. Phys.  107 (1986), 431–-454.

\bibitem{STz2023} C. Sun, N. Tzvetkov, {\it Quasi-invariance of Gaussian measures for the 3d energy critical nonlinear Schr\" odinger equation}, arXiv:2308.12758 [math.AP].

\bibitem{Tzv08} N. Tzvetkov, {\it Invariant measures for the defocusing Nonlinear Schr\"odinger equation}, Ann. Inst. Fourier, 58 (2008) 2543-2604.

\bibitem{sigma} N. Tzvetkov, {\it Quasi-invariant Gaussian measures for one dimensional Hamiltonian PDEs}, Forum Math. Sigma 3 (2015), e28, 35 pp.

\bibitem{Tz24} N. Tzvetkov, {\it  New non degenerate invariant measures for the Benjamin-Ono equation},  Comptes Rendus. Mathématique 362 (2024), 77--86.

\bibitem{ZuGu} Y.L. Zhou and B.L. Guo,  {\it Existence of weak solution for boundary problems of systems of ferro-magnetic chain},  Sci. Sinica Ser. A  27 (1984) 799-–811.

\end{thebibliography}
\end{document}